\documentclass[letterpaper,11pt,reqno]{amsart} 
\usepackage[portrait,margin=3cm]{geometry}
\usepackage{csquotes}
\usepackage{systeme}
\usepackage{mathrsfs,xfrac} 
\usepackage[colorlinks=true,linkcolor=blue,citecolor=blue,urlcolor=blue]{hyperref} 
\usepackage{amsmath,amssymb,amsthm,amsfonts,amsbsy,latexsym,dsfont,color} 
\usepackage[numeric,initials,nobysame]{amsrefs} 
\usepackage[foot]{amsaddr}

\usepackage[utf8]{inputenc}
\usepackage[english]{babel}
\usepackage{comment}

\usepackage[textsize=tiny]{todonotes}
\usepackage{regexpatch}
\makeatletter
\xpatchcmd{\@todo}{\setkeys{todonotes}{#1}}{\setkeys{todonotes}{inline,#1}}{}{}
\makeatother

\usepackage{enumerate}

\newtheorem{thm}{Theorem}[section]
\newtheorem{lem}[thm]{Lemma}
\newtheorem{cor}[thm]{Corollary}
\newtheorem{prop}[thm]{Proposition}
\newtheorem{defn}[thm]{Definition}
\newtheorem{rem}[thm]{Remark}

\newtheorem{ass}[thm]{Assumption}

\renewcommand{\le}{\leqslant} 
\renewcommand{\ge}{\geqslant} 
\renewcommand{\leq}{\leqslant} 
\renewcommand{\geq}{\geqslant} 

\newcommand{\ra}{\rangle}
\newcommand{\la}{\langle}

\newcommand{\eps}{\varepsilon}

\newcommand{\norm}[1]{\left\Vert#1\right\Vert}
\newcommand{\abs}[1]{\left\vert#1\right\vert}

\newcommand{\ie}{\emph{i.e.,}}

 \let\gb=\beta \let\gc=\gamma \let\gd=\delta 
     \let\gl=\lambda             
  
 \let\gD=\Delta  \let\gL=\Lambda

\newcommand{\cA}{\mathcal{A}}\newcommand{\cB}{\mathcal{B}}\newcommand{\cC}{\mathcal{C}}
\newcommand{\cE}{\mathcal{E}}
\newcommand{\cG}{\mathcal{G}}

\newcommand{\cM}{\mathcal{M}}
\newcommand{\cP}{\mathcal{P}}\newcommand{\cR}{\mathcal{R}}

\newcommand{\cV}{\mathcal{V}}
  
\newcommand{\mv}[1]{\boldsymbol{#1}}

\newcommand{\mvt}{\boldsymbol{t}}
\newcommand{\mvy}{\boldsymbol{y}}
\newcommand{\mvz}{\boldsymbol{z}}

\newcommand{\bE}{\mathbb{E}}

\newcommand{\bP}{\mathbb{P}}\newcommand{\bR}{\mathbb{R}}

\newcommand{\sF}{\mathscr{F}}

\DeclareMathOperator{\E}{\mathds{E}}

\DeclareMathOperator{\sech}{sech}

\newcommand{\pushright}[1]{\ifmeasuring@#1\else\omit\hfill$\displaystyle#1$\fi\ignorespaces}

\newcommand{\QW}[1]{{\color{red} QW:#1}}

\begin{document}

\title[Joint parameters estimation in cubic tensor model]{Joint Parameters Estimation in Cubic Tensor Model}
\author{Sumit Mukherjee$^\star$}
\author{Arnab Sen$^\dagger$}
\author{Qiang Wu$^\ddagger$}
\address{$^\star$Department of Statistics, Columbia University, 1255 Amsterdam Ave
New York, New York 10027 \newline 
\indent $^\dagger$$^\ddagger$ School of Mathematics, University of Minnesota, 127 Vincent Hall 206 Church St. SE Min-
neapolis, MN 55455.}
\email{$^\star$sm3949@columbia.edu,$^\dagger$arnab@umn.edu, $^\ddagger$wuq@umn.edu. }
\date{}

\begin{abstract}

We study joint parameter estimation from a single observation in high-dimensional Gibbs measures with cubic tensor interactions, motivated by dense ERGMs, arithmetic-progression models, and inhomogeneous random hypergraphs. Focusing on the maximum pseudolikelihood estimator, we give checkable conditions for joint consistency and asymptotic ill-conditioning. For the edge–triangle ERGM, pseudolikelihood is ill-conditioned in the ferromagnetic regime with nonnegative field, but consistent in a sufficiently strong antiferromagnetic regime. For the edge–three-star ERGM, it is ill-conditioned for all inverse temperatures and external fields. We also study consistency for arithmetic-progression, and inhomogeneous hypergraph models. Our proofs develop nonlinear large-deviation and mean-field approximation tools for cubic tensor Gibbs measures, which have scope for broad applications.
\end{abstract}

\maketitle

\section{Introduction and Main Results}\label{sec:intro}


In this paper we study joint estimation of parameters in a class of Gibbs measures, where one of the underlying sufficient statistics is a cubic tensor. One important class of examples of such cubic tensors models come from the so called exponential random graph models of social networks or ERGMs (see \cite{holland1981exponential,snijders2006new,wasserman1994social,lusher2013exponential}). 
In particular if the motif on the ERGM is a triangle or a three-star, then the corresponding tensor is cubic(~\ref{ssec:ERGM-triangle}). 
Another class of examples of tensors forms arising from combinatorics and probability is the number of 3-Arithmetic Progressions in a random subset of $[n]:=\{1,2,\cdots,n\}$ (see \ref{ssec:3-ap}). Large deviations for the number of 3-APs in a random subset has been studied recently in probability, using the framework of non-linear large deviations~\cite{CD16},  with sharper upper-tail estimates obtained in
\cite{BhattacharyaGangulyShaoZhao20,Warnke17}. 
A natural third class of examples comes from inhomogeneous random hypergraphs, which extend inhomogeneous random graphs to higher-order interactions \cite{bollobas2007phase,lovasz2012large}. 
For more details on this example, we refer the reader to Section~\ref{ssec:inhomogeneous-hypergraph}.

Prior to this work, there has been a significant recent research focus on parameter estimation for Gibbs measures with quadratic interaction, which focus on the so called Ising and Potts models on a weighted graph, and spin glass models. The pioneering work in this area was done in~\cite{Chat07}, where the author studies estimation of 
a single parameter, and establishes $\sqrt{N}$ consistency of the pseudo-likelihood, under very mild assumptions on the underlying matrix which controls the quadratic form. In particular, this work allows the matrix to be a (scaled) adjacency graph (which correspond to Ising models), as well as take both positive and negative values simultaneously (which correspond to the so called spin glass models in statistical physics). In a follow up work, 
\cite{bhattacharya2018inference} considers the same problem of one parameter estimation, and establishes interesting phase transitions in the rate of consistency of the pseudo-likelihood estimator. The problem of joint parameter estimation for Ising models is much more difficult task and exhibit more delicate behavior, this was first studied in \cite{GM20}, where the authors give natural sufficient conditions for $\sqrt{N}$-consistency of both the parameters simultaneously. However, the results of~\cite{GM20} do not cover spin glass models. Recently, in \cite{CSW24} the authors establish joint consistency for quadratic spin glass models using a small ball probability argument. In the more recent work \cite{mukherjee2026joint}, the authors establish sufficient conditions for $\sqrt{N}$-consistency for Potts models, which is an extension of Ising models to multiple colors. To the best of our knowledge, a rigorous study of joint parameter estimation for cubic and higher order tensors has not been carried out in the literature. The aim of the current work is to address this gap, by studying joint estimation of parameters in the cubic tensor case.

\subsection{Our Contributions}

We summarize the main contributions of this paper. We introduce a general framework for joint parameter estimation in Gibbs measures with cubic tensor interactions. This framework is motivated by several concrete models spanning in different fields, including the edge-triangle and edge-three-star ERGM, Gibbs measures tilted by three-term arithmetic progression counts, and inhomogeneous random hypergraph models.

Our first main result Theorem~\ref{thm:main-1} identifies the main statistical mechanism behind joint pseudolikelihood estimation. Roughly speaking, the joint maximum pseudolikelihood estimator is consistent with an explicit quantitative error bound when the conditional local statistics exhibit enough variability across coordinates.  Conversely, when the local statistics become nearly homogeneous, the pseudolikelihood estimator becomes ill-conditioned and estimation might be impossible. 

Building on this, Theorem~\ref{thm:positive-res} gives general and easily checkable structural conditions under which joint estimation succeeds. These conditions are expressed directly in terms of the interaction tensor, and hence can be verified in concrete examples. A key message is that structural inhomogeneity of the tensor can create enough variation in the conditional local statistics to make joint estimation possible. Under these conditions, the maximum pseudolikelihood estimator is jointly $\sqrt N$-consistent.
As a complement,  Theorem~\ref{thm:est-imposs} identifies a broad homogeneous ferromagnetic regime in which pseudolikelihood becomes asymptotically ill-conditioned. This result gives a rigorous explanation for the difficulty of joint estimation in dense ERGM-type models: in homogeneous ferromagnetic regimes, the model may behave as if it has only one effective field, making the linear and cubic parameters hard to distinguish through pseudolikelihood. 

Along the way to prove Theorem~\ref{thm:positive-res} and~\ref{thm:est-imposs}, we develop a mean-field theory for cubic tensor Gibbs measures in Theorems~\ref{thm:mean_field} and~\ref{cor:gap_implies_estimation}. In particular, Theorem~\ref{thm:mean_field} provides a variational approximation for the free energy and a low-complexity description of the conditional mean vectors. Theorem~\ref{cor:gap_implies_estimation} then shows that if the associated variational problem is separated from nearly constant profiles, the local statistics remain sufficiently inhomogeneous and joint $\sqrt N$-consistent estimation follows. Since mean-field approximation has broad applications in the study of high-dimensional probability  and statistics, we expect these results to be useful beyond the estimation problems considered in this paper.

Finally, Section~\ref{sec:examples} applies the general theory to the three families of examples. For the edge-triangle ERGM, we show that pseudolikelihood estimation is ill-conditioned in the ferromagnetic regime with nonnegative field, but that joint estimation becomes possible in a sufficiently strong antiferromagnetic regime. This is in contrast to what happens for the edge-three-star ERGM, where we show that the pseudolikelihood estimation remains ill-conditioned in both ferromagnetic and anti ferromagnetic regimes. For three-term arithmetic progression models, we obtain different conclusions in the cyclic and integer settings, reflecting the different homogeneity properties of the two tensors. For inhomogeneous random hypergraph models, we show that non-constant vertex profiles lead to consistent joint estimation, while sufficiently homogeneous ferromagnetic models fall into the ill-conditioned regime.

\subsection{Definitions and Main Results}

\begin{defn}\label{def:main-set-up}
Suppose $\kappa_-<\kappa_+$ are two real numbers. Let $\mu$ be a probability measure on $[\kappa_-,\kappa_+]$ such that $\kappa_-,\kappa_+\in {\rm supp}(\mu)$. 
For every $\lambda\in \bR$, define the tilted probability measure $\mu_\lambda$ by the Radon-Nikodym derivative
\begin{align*}
    \frac{d\mu_\lambda}{d\mu}(z)=\exp\big(\lambda z-\Lambda(\lambda)\big), \quad\text{ where }\quad \Lambda(\lambda):=\log \int e^{\lambda z}d\mu(z).
\end{align*}

\end{defn}

Throughout the following text, we will use the notation $\kappa = \max\{\abs{\kappa_-}, \abs{\kappa_+}\}$. Given the above general setting, we now introduce a probability distribution on the space $[\kappa_-,\kappa_+]^N$ by the following Radon-Nikodym derivative:
\begin{align}\label{eq:model}
    \frac{dP_{\gb,h}}{d\mu^{\otimes N}}(x) = \frac{\exp(f(x) ) }{Z_N(\gb,h)}, \quad \text{ where }\quad f(x):= \frac{\gb}{3} \cdot \la A, x^{\otimes 3}\ra + h \cdot \la x, \mv1\ra.
\end{align}
Here  $\gb,h$ are real valued parameters, and $A$ is a symmetric $3$-tensor with nonnegative entries, which vanishes along the diagonal. More precisely, the 3-tensor $A$ satisfies
\begin{align*}
    A_{ijk} = \begin{cases}
         0, \quad &\text{if $\{i,j,k\}$ are not all distinct,} \\
         A_{\pi(i)\pi(j)\pi(k)}\ge 0, \quad & \text{for all permutations $\pi \in S_3$. }
    \end{cases}
\end{align*}
Finally, 
\begin{align*}Z_N(\beta,h)=\int_{[\kappa_-,\kappa_+]^N}\exp(f(x))d\mu^{\otimes N}(x)
\end{align*}is the normalizing constant/partition function, which makes \eqref{eq:model} a probability distribution.
\\

    Suppose we have access to a sample $X$ from the ground truth distribution $P_{\gb_0,h_0}$, the goal is to estimate the ground truth $(\gb_0,h_0)$ using the sample. A natural estimator is the maximum likelihood estimate (MLE), obtained by maximizing the log likelihood function. Unfortunately, a major challenge for such models is the intractability of the normalizing constant/partition function, and so the log-likelihood function is hard to work with, both analytically and computationally. To bypass this, the focus is on studying the pseudo-likelihood estimator of Besag~\cite{besag1975statistical}, computing which does not require the knowledge of the normalizing constant, and is hence numerically feasible. 

    Before we formally define the pseudo-likelihood estimator, first note that the conditional distribution of $X_i$ given $(X_j=x_j,j\ne i)$ is $\mu_{\beta m_i(x)+h}$ which has the mean $b_i(x):=\Lambda'(\beta m_i(x)+h)$ (in the notation of Definition~\ref{def:main-set-up}),
where $m_i(x)$ is the local field, given by 
\[
m_i(x) := \sum_{j,k} A_{ijk}x_j x_k \quad \text{for each $i\in [N]$}.
\]
The pseudo-likelihood estimator is obtained by multiplying the densities of all the one-dimensional conditional distributions. Taking a log, we get the log-pseudo-likelihood function, differentiating which
with respect to $(\beta,h)$, we obtain the score functions 
\begin{align*} 
S(\gb,h|x):=\frac{\partial L}{\partial \gb}  = \sum_i m_i(x)(x_i - b_i(x)), \qquad
Q(\gb,h|x):=\frac{\partial L}{\partial h}  = \sum_i (x_i - b_i(x)).
\end{align*} 

The maximum pseudo-likelihood estimator (MPLE) $(\hat{\gb}(x), \hat{h}(x))$ is defined as the unique solution to 
\(
(S(\gb,h|x),Q(\gb,h|x)) = (0,0),
\)
provided such a unique solution exists.
For notational convenience, we will compress the dependence of the MPLE on $x$ afterwards. We further compute the negative Hessian matrix of $L(\gb,h|x)$,
\begin{align*}
H(\gb, h|x):= \begin{pmatrix}
    \sum_i m_i(x)^2 \theta_i(\gb,h|x) & \sum_i m_i(x)  \theta_i(\gb,h|x) \\
     \sum_i m_i(x)  \theta_i(\gb,h|x) & \sum_i \theta_i(\gb,h|x) 
\end{pmatrix}
\end{align*}
where 
\[
\theta_i(\gb,h|x):= \Lambda''(\gb m_i(x) + h), \quad \text{for each $i\in [N]$}.
\]
Note that the determinant of the Hessian
\begin{align}\label{eq:hess-det}
     \abs{H(\gb, h|x)} =  \frac12\sum_{i, j}\theta_i(\gb,h|x) \theta_j(\gb,h|x) (m_i(x) - m_j(x))^2 = N^2 \widetilde{T}_N(x),
\end{align}
where \[
\widetilde{T}_{N}(x):=\frac{1}{2N^2}\sum_{i, j}\theta_i(\gb,h|x) \theta_j(\gb,h|x) (m_i(x) - m_j(x))^2.
\]
Most of the time, we will work with the following related but simpler quantity.
\begin{align}\label{eq:TN}
    T_N(x):= \frac{1}{N}\sum_i (m_i(x) - \bar{m}(x))^2 = \frac{1}{2N^2}\sum_{i, j} (m_i(x) - m_j(x))^2.
\end{align}
For convenience, we will also use the following notations. Let 
\begin{align*}
    \cR_i:=\sum_{j,k}A_{ijk}, \quad \ \ \cR_{ij}:=\sum_{k}A_{ijk}, \quad \text{and} \ \ \bar{\cR} = \frac1N \sum_i \cR_i.
\end{align*}

We impose the following standing assumption that the tensor $A$ has bounded row sums.
\begin{ass}\label{ass:standing}
       Throughout the paper, we assume tensor $A$ has nonnegative entries and further there exists a finite constant $\gc>0$ such that $  \max_{i\in [N]} \sum_{j,k}A_{ijk} \le \gc$.
\end{ass}
We record a few elementary consequences of Assumption~\ref{ass:standing} that will be used throughout.
\begin{enumerate}
    \item Uniform bound on the local fields:
    \begin{equation*} 
            \max_{i } | m_i(x)| \le \gc\kappa^2<\infty.
\end{equation*}
    \item Comparison between $\widetilde T_N$ and $T_N$: Since $\theta_i(\gb,h\mid x)\ge \inf_{|\lambda|\le |\beta|\gamma\kappa^2+|h|} \Lambda''(\gl)$, we have
       \begin{equation*} 
    \widetilde T_N(x) \ge \big(\inf_{|\lambda|\le |\beta|\gamma\kappa^2+|h|} \Lambda''(\lambda)\big)^2\, T_N(x).
    \end{equation*}
    \item Operator norm bound on $\cR$ matrix: 
    $$ \|\cR\|_{\mathrm{op}} \le \max_{i } \sum_{j}\cR_{ij}\le \max_{i } \sum_{j,k}A_{ijk} \le \gc.$$ 
 \end{enumerate}


Now we state our first main result on the existence and consistency of MPLE for the 3-tensor model.

\begin{thm}\label{thm:main-1}
    Suppose $X$ is a sample from the unknown ground truth $P_{\gb_0,h_0}$ for $\gb_0 ,h_0 \in \bR$, and the symmetric tensor $A$ satisfies Assumption~\ref{ass:standing}. 
    Assume that under $P_{\gb_0,h_0}$, we have        
     \[ T_N(X) = \omega_p(N^{-2/7}), \] 
   that is,  $P_{\gb_0,h_0}( T_N(X) > c N^{-2/7}) \to 1 $  for any $c >0$.
Then, with probability tending to one, the MPLE exists. Further, under $P_{\gb_0,h_0}$ we have 
\[
\max \big( |\hat{\gb} - \gb_0|, |\hat{h} - h_0|\big) = O_p(N^{-\frac12}T_N^{-1}).
\]
\end{thm}

By Theorem~\ref{thm:main-1}, the key of pseudo-likelihood estimation is to study the object $T_N(X)$. However, verifying the condition on $T_N(X)$ in Theorem~\ref{thm:main-1} in general is difficult. Our next theorem gives two sufficient conditions on the tensor $A$ which guarantees $T_N(X)=\Omega_p(1)$, and hence joint $\sqrt{N}$-consistency of the pseudo-likelihood estimator.

\begin{thm}\label{thm:positive-res}
For $X \sim P_{\gb_0,h_0}$, suppose the Assumption~\ref{ass:standing} holds. Assume that either $\mathrm{Tr}(\cR^2) = \Omega(N)$ or $\sum_i (\cR_i -\bar{\cR})^2 = \Omega(N)$ holds. If $\Lambda'(h_0)\ne 0$, then we have \[
T_N(X) = \Omega_p(1).
\] 
Thus the maximum pseudo-likelihood estimator $(\hat{\gb}, \hat{h})$ is $\sqrt{N}$-consistent.

\end{thm}

\begin{rem}
There are two natural sufficient conditions under which $\gL'(h_0)\ne 0$. 
\begin{itemize}
    \item $0\notin (\kappa_-,\kappa_+)$: This happens iff the measure $\mu$ is supported on the non-negative real line, or the non-positive real line. Examples are Bernoulli, Binomial, $U[0,1]$.

    \item $0\in(\kappa_-,\kappa_+)$ and $h_0\ne \Phi(0)$:  Since $\gL'$ is strictly increasing and
    $\gL'(\Phi(0))=0$, this implies $\gL'(h_0)\ne 0$. 
 In particular, if $\mu$ is symmetric about $0$, which includes symemtric Rademacher, then $\Phi(0)=0$, and hence
    $h_0\ne 0\Leftrightarrow \gL'(h_0)\ne 0$.
\end{itemize}
\end{rem}
\subsection{mean-field Tensors}

A natural question is whether $\sqrt{N}$-consistency fails, when both the sufficient conditions in Theorem \ref{thm:positive-res} fails. In this case the situation is a bit more delicate. To study this, we introduce the following conditions:
\begin{align}
\textit{mean-field condition}: \quad  \big \| \sum_i A_i^2 \big {\|}_{\rm{op}} = o((\log N)^{-1}) \label{cond:GW},\\
\label{cond:regu}
\textit{asymptotic regularity condition}: \quad  \sum_i (\cR_i -\bar{\cR})^2 = o(N).
\end{align}

\begin{rem}\label{rem:trace-optr-converse}
Assumption \eqref{cond:GW} may not immediately look natural, but it is indeed a slight strengthening of the more natural assumption ${\rm Tr}(\mathcal{R}^2)=o(N/\log N)$ arising from Theorem \ref{thm:positive-res} (see Lemma \ref{lem:trace-optr}{\rm (a)}). The converse implication is false in general. However, Lemma~\ref{lem:trace-optr}{\rm (b)} shows that for tensors satisfying the following two conditions
\begin{align}\label{cond:strong-pseu-regu}
\textit{strong pseudo-regularity condition}: \quad \max_{\cR_{ij}>0}\cR_{ij} \le K\min_{\cR_{ij}>0}\cR_{ij},\\
\label{cond:A_nontrivial}
  \textit{non degeneracy condition}: \bar \cR \ge \alpha > 0 \quad \text{and} \quad \cR_i > 0 \text{ for all } i \in [N].  
\end{align}
the converse does hold. Thus, for tensors satisfying the above two conditions, the mean-field condition is equivalent to the more natural condition ${\rm{Tr}(\cR^2)} = o(N/\log N)$. Condition \eqref{cond:strong-pseu-regu}
  essentially demands comparable co-degrees of the underlying hypergraph, whereas condition \eqref{cond:A_nontrivial} rules out degenerate cases (such as $A\equiv 0$). It turns out that all our examples in Section~\ref{sec:examples} satisfy both the above conditions, and so thus mean-field can be checked via ${\rm{Tr}(\cR^2)}$. 
We point out that the strong pseudo-regularity condition only requires the
\emph{nonzero} entries of the matrix $\cR$ to be comparable to one another.
This is weaker than requiring all entries of $\cR$ to be equal or comparable.
In the hypergraph setting, this corresponds to asking that the positive
pair-codegrees are comparable, rather than requiring pair-codegree regularity
over all pairs. This distinction is useful, since many natural examples have
zero pair-codegrees, which includes  the edge-triangle and edge-three-star ERGM.
\end{rem}

We will now give a structural result for the cubic Gibbs measures, where the underlying tensor $A$ satisfies \eqref{cond:GW}. To state this result, we first need to introduce some notations.

\begin{defn}
    Let $\mu, \kappa_-,\kappa_+, \Lambda,\lambda, \mu_\lambda$ be as in Definition \ref{def:main-set-up}. Then straight-forward calculus gives
\[\E_{Z\sim {\mu_\lambda}}[Z]=\Lambda'(\lambda),\qquad {\mathrm Var}_{Z\sim \mu_\lambda}(Z)=\Lambda''(\lambda). \]
Since $\mu$ is non-degenerate, this implies that $\Lambda''(\lambda)>0$ for all $\lambda\in \mathbb{R}$. Moreover, it is not hard to show that
\[ \lim_{\lambda \to -\infty} \Lambda'(\lambda) = \kappa_{-}, \quad \lim_{\lambda \to +\infty} \Lambda'(\lambda) = \kappa_{+}. \]
In particular, this implies that the function $\lambda\mapsto \Lambda'(\lambda)$ from $\bR\mapsto (\kappa_-,\kappa_+)$ is strictly increasing and maps onto $(\kappa_{-}, \kappa_{+})$, and hence has an inverse $\Phi:(\kappa_-,\kappa_+)\mapsto \bR$.
\\

For $t \in [\kappa_-, \kappa_+],$ define the Legendre transform of $\Lambda$ as
\begin{align}\label{eq:entropy} I(t) := \sup_{ \gl \in \bR} \big( \gl t  -  \Lambda (\gl) \big) \in [0, \infty]=t \Phi(t)-\Lambda(\Phi(t)) =  \mathrm{KL}(\mu_{\Phi(t)}\| \mu) ,\end{align}
 where the second equality follows from a direct computation,
along with the convention
\begin{align*}
 \Phi(\kappa_-) := -\infty, \quad \mu_{-\infty} := \delta_{\kappa_-}, \text{ and } \quad
 \Phi(\kappa_+) := +\infty, \quad \mu_{\infty} := \delta_{\kappa_+}.
\end{align*}
For $t \in [\kappa_-, \kappa_+]^c$, we set $I(t) = \infty$, which ensures that $I(\cdot)$ is lower semi-continuous on $\bR$.
\end{defn}

Next, we introduce the definition of a sequence of low-complexity sets in $[\kappa_-,\kappa_+]^N$. 
\begin{defn}[Low complexity set]\label{def:low_complex}
Let $(E_N)_{N\ge 1} \subseteq [\kappa_-, \kappa_+]^N$ be a sequence of sets. We say that $E_N$ is of low complexity, if for every $\delta>0$ there exists a $\delta\sqrt{N}$-net $\mathcal{D}_N(\delta)$ for the set $E_N$ in Euclidean metric, such that $\log |\mathcal{D}_N(\delta)|=o(N)$.
\end{defn}

Finally, we introduce the notion of mean-field cubic Gibbs measures on $[\kappa_-,\kappa_+]^N$, see the similar notions in~\cite{basak2017universality,lacker2024mean}.
\begin{defn}[Mean-field measure]
    For every $N\ge 1$, let $\mathbb{P}_N$ be a probability measure on $[\kappa_-,\kappa_+]^N$. We will say that the sequence of probability $\mathbb{P}_N$ is mean-field, if
    \begin{align*}
        \inf_{\mathbb{Q}_N\in {\rm prod}([\kappa_-,\kappa_+]^N)}{\rm{KL}}(\mathbb{Q}_N\|\mathbb{P}_N)=o(N).
    \end{align*}
    In words, the distribution $\mathbb{P}_N$ is well approximated in terms of the Kullback-Leibler divergence by product measures. In particular, if $\mathbb{P}_N=P_{\beta,h}$ from \eqref{eq:model}, then a straight-forward calculation gives
    \begin{align*}
        \inf_{\mathbb{Q}_N\in {\rm prod}([\kappa_-,\kappa_+]^N)}\mathrm{KL}(\mathbb{Q}_N\|\mathbb{P}_N)=\log Z_N(\beta,h)-\sup_{ t \in [\kappa_-,\kappa_+]^N}\big( f(t) - I( t)\big),
    \end{align*}
    where $I(t):=\sum_iI(t_i)$ with $I(t_i)$ defined in~\eqref{eq:entropy}.
    Consequently, the measure $P_{\beta,h}$ is mean-field iff
    \begin{align}\label{eq:prod_mf}
    \log Z_N(\gb,h) - \sup_{ t \in [\kappa_-,\kappa_+]^N}\big( f(t) - I( t)\big) = o(N).
    \end{align}
\end{defn}

\begin{thm}\label{thm:mean_field}
    Suppose the tensor $A$ satisfies the Assumption~\ref{ass:standing}, and the mean-field condition~\eqref{cond:GW}. 
    \begin{enumerate}
    \item[(i)]
    Then the $3$-tensor model $P_{\beta,h}$ 
    is mean-field, in the sense that \eqref{eq:prod_mf} holds.

\item[(ii)]

Moreover, the set of all conditional means 
\(\mathcal{M}_N:=\{b(x): x\in [\kappa_-, \kappa_+]^N\}\) has low complexity in the sense of Definition \ref{def:low_complex}.
\end{enumerate}
\end{thm}

We will now show that for mean-field tensors (i.e.~when \eqref{cond:GW} holds), joint estimation using pseudo-likelihood may or may not be possible, depending on the parameter $(\beta_0,h_0)$. Our first result gives a sufficient condition for joint $\sqrt{N}$ estimation for mean-field tensors.

\begin{thm}\label{cor:gap_implies_estimation}
 Let
$X \sim P_{\beta_0,h_0}$, where the underlying  tensor $A$ satisfies  Assumption \ref{ass:standing} and the mean-field condition \eqref{cond:GW}. Suppose that there exist constants
$\varepsilon>0$ and $\delta>0$ such that, for all sufficiently large $N$,
\[
\sup_{\substack{y \in[\kappa_-,\kappa_+]^N\\ \sum_i (y_i-\bar y)^2\le \varepsilon N}}
\bigl(f(y)-I(y)\bigr)
\le
\sup_{y\in[\kappa_-,\kappa_+]^N}\bigl(f(y)-I(y)\bigr)-\delta N.
\]
Then there exists $c>0$ such that 
\[
P_{\gb_0,h_0}( T_N(X) \ge c) \to 1.
\]
Consequently, the joint pseudo-likelihood estimator is $\sqrt{N}$-consistent at $(\beta_0,h_0)$.
\end{thm}

Our second result shows that estimation may also be difficult using pseudo-likelihood under certain parameter regimes, under the following ``well-connectedness'' assumptions on the tensor $A$. We first assume that the tensor $A$ satisfies \eqref{cond:A_nontrivial}.
Then the Markov transition matrix $\cP:=D^{-1}\cR$ is well defined, where
where $D = \mathrm{diag}(\cR_1,\ldots,\cR_N)$.
Since $\cP = D^{-1/2}\big(D^{-1/2}\cR D^{-1/2}\big)D^{1/2}$
is similar to the symmetric matrix $D^{-1/2}\cR D^{-1/2}$, all its eigenvalues are real and can be listed as  
\[
1 = \gl_1 \ge \gl_2 \ge \cdots \ge \gl_N \ge -1.
\]
The well-connectedness of $A$ is then encoded in the following {\em spectral gap condition}: 
\begin{align}\label{cond:spec-gap}
   \quad 1-\gl_2\ge \gd \quad \text{for some constant $\gd>0$.}
\end{align}

Finally, we need the following technical definition on the reference measure $\mu$. 
\begin{defn}[Stochastic non-negativity]
    The reference measure $\mu$ is stochastic non-negative if $I(t) \le I(-t)$ for $t\ge 0$. Note that this definition implies $\kappa_+ \ge \abs{\kappa_-}$. 
\end{defn}
We now give some examples of stochastically non-negative measures. All the implied claims are verified in \cite[Proposition 1.3]{BDM23}.

\begin{itemize}
\item
Any measure which is supported on non-negative reals is stochastically non-negative. In particular, this includes
 Bernoulli($p$), and $U[0,1]$.
 
 \item 
 Any symmetric measure is stochastically non-negative. In particular, this includes Rademacher and $U[-1,1]$.
 
 \item 
 More generally, any measure which is a non-negative tilt of a symmetric measure is stochastically non-negative. In particular, this includes asymmetric Rademacher distributions $\mu = p \gd_{-1} + (1-p) \gd_{1}$ with $p\le 1/2$.

\end{itemize}

Our next result gives a sufficient condition for $T_N$ to be $o_p(1)$ for mean-field regular tensors, suggesting that estimation using pseudo-likelihood can be difficult in certain parameter regimes.
\begin{thm}\label{thm:est-imposs}
Suppose that the tensor $A$  satisfies Assumption~\ref{ass:standing},  as well as conditions~\eqref{cond:GW}, \eqref{cond:regu}, \eqref{cond:A_nontrivial} and \eqref{cond:spec-gap}. Assume that the reference measure $\mu$ is stochastically non-negative. Let $X \sim P_{\gb_0, h_0}$ with $\beta_0>0$ and $h_0 \ge 0$ . If either (i) $h_0>0$ or (ii) $\kappa_{-}\ge 0$, then
\[
T_N(X)=o_p(1).
\]
\end{thm}

\begin{rem}
    In particular, if the tensor is either the complete tensor $A_{ijk}=\frac{1}{N^2}1_{\{i\ne j\ne k\}}$, or a scaled Erd\H{o}s-R\'enyi hypergraph with parameter $p$ fixed, then extending the arguments of \cite[Theorem 1.6]{GM20} for the quadratic case it should be possible to show that consistent joint estimation of parameters is impossible using any estimator, and not just the pseudo-likelihood estimator. Note that both these tensors are both (approximately) regular (i.e.~satisfies \eqref{cond:regu}) and mean-field (i.e.~satisfies \eqref{cond:GW}). This demonstrates that for certain parameter regimes, estimation can indeed be hard (if not impossible) for mean-field regular tensors. Since this should be a straight-forward extension of the arguments of \cite[Theorem 1.6]{GM20}, we do not repeat this here. 
\end{rem}
\begin{rem}
    We formulate the paper under the global standing assumption that
\(A_{ijk}\ge0\), since this covers the main examples considered in Section~\ref{sec:examples} and
keeps the notation simple.  
The non-negativity assumption is only essential for Theorem~\ref{thm:est-imposs}. In contrast, Theorems~\ref{thm:main-1},
\ref{thm:positive-res}, ~\ref{thm:mean_field} and~\ref{cor:gap_implies_estimation} can be extended to signed
symmetric tensors vanishing on diagonals, provided Assumption~\ref{ass:standing}
is replaced by the absolute row-sum condition $\max_i\sum_{j,k}|A_{ijk}|\le \gamma.$
In that signed extension, the bounds in the concentration results of
Section~\ref{sec:prelim} should be in terms of the absolute tensor
$A^\dagger=(|A_{ijk}|)_{i,j,k\in[N]}$. We do not state this extra generality separately. An extra caveat for Theorem~\ref{cor:gap_implies_estimation} is that a mean-field type condition for $A^\dagger$ needs be assumed.
\end{rem}

\subsection{Future Scope}


We view this paper as a first step toward a broader understanding of inference in cubic, and more generally higher-order, interaction models. Several natural directions remain open. One important problem is to establish limiting distributions for the maximum pseudolikelihood estimator in regimes where it is consistent. Such results would allow one to construct confidence sets for the unknown parameters. This question is closely related to the study of limiting distributions for conditionally centered sums of spins, or magnetizations, which are themselves objects of independent interest in probability and statistical physics. As is well known, magnetization fluctuations in tensor models may be non-Gaussian, particularly near phase-transition boundaries.
Another natural direction is to incorporate a quadratic interaction term, that is, a matrix-valued tensor, leading to a three-parameter model. Establishing sufficient conditions for estimability in this setting would require controlling local fields generated jointly by the cubic and quadratic tensors. This interaction appears difficult to handle directly with the techniques developed here.
A third direction is to extend our analysis to spin-glass tensors, allowing both positive and negative entries, in the spirit of the quadratic results of \cite{CSW24}. As emphasized there, spin-glass models require different tools, and we expect the corresponding inference theory to exhibit qualitatively different behavior.
Finally, it would be interesting to extend our results to tensors of order higher than three. Some of our arguments seem to admit relatively straightforward extensions, while others appear to require new ideas. Moreover, we expect the behavior of higher-order tensor models to depend substantially on the parity of the interaction order, even versus odd, as has been observed for complete tensors in \cite{mukherjee2021fluctuations}.

\subsection{Structure of the Paper}
The rest of the paper is organized as follows. Section~\ref{sec:examples}
discusses applications to the main examples, including the edge-triangle and edge-three-star ERGM,
the 3-term arithmetic progression model, and inhomogeneous random hypergraph models, and
deduces their specific estimation consequences from Theorems~\ref{thm:main-1},
\ref{thm:positive-res}, \ref{cor:gap_implies_estimation}, and~\ref{thm:est-imposs}.
Section~\ref{sec:prelim} collects the preliminary concentration estimates, the mean-field criterion,
the connectivity criterion, and auxiliary lemmas. Section~\ref{sec:proof of positive-res} proves
Theorem~\ref{thm:positive-res}. Section~\ref{sec:mean-field-analysis} proves
Theorems~\ref{thm:mean_field}, \ref{cor:gap_implies_estimation}, and~\ref{thm:est-imposs}.
Appendix~\ref{sec:proof of main-1} proves Theorem~\ref{thm:main-1}, and
Appendix~\ref{sec:proof of example} contains the proofs of the application results from
Section~\ref{sec:examples}. Appendix~\ref{sec:proof of auxiliary} collects the proofs of the auxiliary lemmas in Section~\ref{ssec:auxiliary}.

\section{Applications}\label{sec:examples}

The general cubic tensor model applies to several natural families of examples. In this
section we record three representative applications: edge-triangle and edge-three-star
ERGMs, Gibbs measures tilted by three-term arithmetic progression counts, and
inhomogeneous random hypergraph tensors. The proofs reduce to verifying the structural
hypotheses of the general theorems and are given in Appendix~\ref{sec:proof of example}.


\subsection{Exponential Random Graph Models}\label{ssec:ERGM-triangle}

We first consider dense ERGMs whose sufficient statistics are the edge count and a cubic
motif count. ERGMs model network data through exponential-family weights on graph
features such as edge density and local subgraph counts
\cite{holland1981exponential,lusher2013exponential,snijders2006new,wasserman1994social}.
Dense ERGMs are known to exhibit degeneracy and near Erd\H{o}s--R\'enyi behavior in
parts of parameter space \cite{chatterjee2013estimating,schweinberger2011instability}, which suggests that joint estimation of the edge and motif parameters can be difficult
in certain parts of the parameter regime. Our results make this precise for
pseudolikelihood: the edge--triangle model is ill-conditioned in the ferromagnetic
nonnegative-field regime but becomes estimable in a sufficiently strong antiferromagnetic
regime, while the edge--three-star model is ill-conditioned for all fixed parameters.
\subsubsection{Edge-Triangle ERGM}
Setting $N=\binom{n}{2}$, we consider the following probability distribution on the space of edge weights of the complete graph $K_n$, identified with the upper-triangular entries of a symmetric $n\times n$ matrix $G=(G_{ij})_{1\le i,j\le n}$ with $G_{ii}=0$ and $G_{ji}=G_{ij}$:
\begin{align}\label{eq:triangle}
\frac{dP_{\gb,h}}{d\mu^{\otimes N}}(G) = \frac{\exp(\frac{\beta}{3n}\sum_{i,j,k}G_{ij}G_{jk}G_{ki} +h\sum_{i<j} G_{ij}) }{Z_N(\gb,h)}.
\end{align}
This is a special case of the general model \eqref{eq:model}, with tensor
\begin{align}\label{eq:ergm-triangle}
A_{e_0e_1e_2}
=
n^{-1}1_{\{(e_0,e_1,e_2)\text{ form a triangle}\}}, \qquad \text{for any edges $e_0, e_1, e_2$}.
\end{align}
In particular, when $\mu=\frac12(\delta_0+\delta_1)$, the configuration space becomes $\{0,1\}^N$, and \eqref{eq:triangle} reduces to the classical edge-triangle exponential random graph model, commonly referred to as  ERGM in the statistics and social science literature, see~\cite{holland1981exponential,wasserman1994social,snijders2006new}.

\begin{thm}\label{thm:app-triangle}
Consider the edge-triangle ERGM defined in~\eqref{eq:triangle}, let $X \sim P_{\gb_0,h_0}$. Then the following holds:
\begin{enumerate}
    \item[(a)]
    Suppose $\mu$ is stochastically non-negative, $\gb_0>0, h_0\ge0$.  If either $h_0>0 $ or $ \kappa_-\ge 0$, then $T_N(X)=o_p(1)$ under $P_{\beta_0,h_0}$.
    
    \item[(b)]
    Suppose $\kappa_-=0$ and $\mu(\{0\})>0$. For any $h_0\in \bR$ there exists $L(h_0)>0$, such that for all $\beta_0\le -L(h_0)$ we have $P_{\beta_0,h_0}(T_N>c)\to 1$, for some $c>0$ depending on $(\beta_0,h_0,\mu)$. Consequently, the maximum pseudolikelihood estimator exists with probability tending to one
and is jointly $\sqrt{N}$-consistent.
\end{enumerate}
\end{thm}

\begin{rem}
    Part (a) of the above theorem suggests a bottleneck for estimation using pseudo-likelihood in this regime. Similar findings have been reported in the literature on ERGMs~\cite{winstein2026wasserstein, bhamidi2008mixing, chatterjee2013estimating, eldan2018exponential}.
In contrast, part (b) shows that consistent estimation of both parameters is possible in the strong anti-ferromagnetic regime $\{\beta_0\le -L,\ h_0\in\mathbb{R}\}$, provided that $L = L(h_0)>0$ is chosen sufficiently large. To the best of our knowledge, this has not been established before, although \cite[Theorem 7.1]{chatterjee2013estimating} is a step in that direction. 
\end{rem}

\subsubsection{Edge-Three-Star ERGM}
We next consider the edge-three-star ERGM. A
three-star consists of three distinct edges sharing a common vertex. As
before, set \(N=\binom{n}{2}\), and identify the space of edge weights of
\(K_n\) with the upper-triangular entries of a symmetric \(n\times n\)
matrix \(G=(G_{ij})_{1\le i,j\le n}\), where \(G_{ii}=0\) and
\(G_{ji}=G_{ij}\). Now the Gibbs measure is 
\begin{equation}
\frac{dP_{\beta,h}}{d\mu^{\otimes N}}(G)
= Z_N(\beta,h)^{-1}\exp\Big(
\frac{\beta}{3n^2}
\sum_{i,j,k,\ell  
\, { \rm distinct}}
G_{ij}G_{ik}G_{i\ell}
+
h\sum_{i<j}G_{ij}
\Big)
\label{eq:three-star}
\end{equation}
This is a special case of the general model~\eqref{eq:model}, with tensor
\begin{equation}
A_{e_0e_1e_2}
=
\frac{1}{n^2}
\mathbf 1\{(e_0,e_1,e_2)\text{ form a three-star}\},
\label{eq:ergm-three-star}
\end{equation}
for any edges \(e_0,e_1,e_2\in E(K_n)\). Here, the event in
\eqref{eq:ergm-three-star} means that \(e_0,e_1,e_2\) are distinct and
share a common vertex. In particular, when
\(
\mu=\frac12(\delta_0+\delta_1),
\)
the configuration space becomes \(\{0,1\}^N\), and
\eqref{eq:three-star} reduces, up to the normalization of the three-star
parameter, to the classical edge-three-star ERGM in
\cite{chatterjee2013estimating,snijders2006new,wasserman1994social}.

\begin{thm}\label{thm:app-3-star}
Consider the edge-three-star ERGM defined in~\eqref{eq:three-star}, and let $X \sim P_{\gb_0,h_0}$. Then  for any $\gb_0, h_0\in \mathbb{R}$, we have $T_N(X)=o_p(1)$ under $P_{\beta_0,h_0}$.
\end{thm}

Both the edge-triangle and edge-three-star ERGMs satisfy the mean-field condition~\eqref{cond:GW}, the asymptotic regularity condition~\eqref{cond:regu}, and the spectral-gap condition~\eqref{cond:spec-gap}, as verified in Lemmas~\ref{lem:triangle_mean} and~\ref{lem:three-star-mean-field} in the Appendix. Despite these common structural properties, the two models have different variational behavior. For the edge-triangle model, Theorem~\ref{thm:est-imposs} gives $T_N(X)=o_p(1)$ in the ferromagnetic regime, thus pseudolikelihood estimation is asymptotically ill-conditioned. In the strongly antiferromagnetic regime, however, Lemma~\ref{lem:triangle} shows that near-maximizers of the variational problem remain a positive distance from constant vectors. Theorem~\ref{cor:gap_implies_estimation} then yields joint $\sqrt{N}$-consistent estimation. For the edge-three-star model, Lemma~\ref{lem:three-star-variational} shows that near-maximizers are asymptotically close to constant vectors for every fixed $(\beta_0,h_0)$, in both the ferromagnetic and antiferromagnetic regimes. Consequently, $T_N(X)=o_p(1)$ throughout the parameter space.

\color{black}

\subsection{Three Terms Arithmetic Progression}\label{ssec:3-ap}

Our second application concerns Gibbs measures tilted by the number of nontrivial
three-term arithmetic progressions. 
The extremal problem of finding large subsets of \([N]\) with no nontrivial
3-AP goes back to Roth's theorem~\cite{Roth53}, and quantitative refinements remain an
active topic~\cite{bourgain1999,kelley2023strong,Sanders11}. Its probabilistic counterpart
is the study of upper tails for 3-AP counts in random subsets, where nonlinear large
deviation methods were developed in~\cite{CD16} and sharpened in
\cite{BhattacharyaGangulyShaoZhao20,Warnke17}.

We consider both the integer and cyclic versions, since they lead to different tensor
geometries. In \(\mathbb Z\), a nontrivial 3-AP is a triple of the form
\((a,a+r,a+2r)\) with \(r\ne0\), equivalently a triple \((a,b,c)\) with
\(a+c=2b\) and \(a,b,c\) distinct. In \(\mathbb Z/N\mathbb Z\), the same definition is
used modulo \(N\), again requiring the three entries to be pairwise distinct. For a set
\(S\subset [N]\) or \(S\subset \mathbb Z/N\mathbb Z\), let
\(T_3^{\mathrm{int}}(S)\) and \(T_3^{\mathrm{cyc}}(S)\) denote the corresponding numbers
of nontrivial 3-APs contained in \(S\). Equivalently, a three-point set
\(\{a,b,c\}\) contributes if some ordering of its elements forms a nontrivial arithmetic
progression in the relevant ambient group.

We now define the associated Gibbs measures. For \(x\in\{0,1\}^N\), identify \(x\) with
the subset \(\{i:x_i=1\}\). Let \(P_{\beta,h}^{\mathrm{int}}\) and
\(P_{\beta,h}^{\mathrm{cyc}}\) be the Gibbs distributions of the form~\eqref{eq:model},
with reference measure \(\mu=\frac12\delta_0+\frac12\delta_1\), obtained from the tensors
\[
A^{\mathrm{int}}_{a,b,c}
:=
\frac1N
\mathbf 1_{\{\{a,b,c\}\text{ forms a nontrivial 3-AP in }\mathbb Z\}},
\quad 
A^{\mathrm{cyc}}_{a,b,c}
:=
\frac1N
\mathbf 1_{\{\{a,b,c\}\text{ forms a nontrivial 3-AP in }\mathbb Z/N\mathbb Z\}}.
\]
for \(a,b,c\in[N]\). Thus $h$ controls the density of the random set, while
\(\beta\) controls the arithmetic interaction: positive $\beta$ favors sets with many
3-APs, whereas negative $\beta$ favors progression-sparse sets. Our results on pseudolikelihood estimation for the $3$-AP models above are summarized in the following theorem.

\begin{thm}\label{thm:app-3ap}
Consider the 3-AP models defined above with \(\mu=\frac12\delta_0+\frac12\delta_1\). 
\begin{enumerate}
\item [(i)] In the cyclic case \(A=A^{\mathrm{cyc}}\), the following statements hold.
\begin{enumerate}
\item If \(\beta_0>0\) and \(h_0\ge0\), then $  T_N(X)=o_p(1)$
under \(P^{\mathrm{cyc}}_{\beta_0,h_0}\). 


\item For every \(h_0\in\bR\), there exists \(L=L(h_0)>0\) such that, whenever \(\beta_0\le -L\), there is a constant \(c>0\) with $ P^{\mathrm{cyc}}_{\beta_0,h_0}\bigl(T_N(X)\ge c\bigr)\to 1.$
Consequently, the MPLE exists with probability tending to one and is jointly \(\sqrt N\)-consistent for \((\beta_0,h_0)\).
\end{enumerate}

\item [(ii)] In the integer case \(A=A^{\mathrm{int}}\), for every fixed \((\beta_0,h_0)\in\bR^2\), $T_N(X)=\Omega_p(1)$
under \(P^{\mathrm{int}}_{\beta_0,h_0}\). Consequently, the MPLE exists with probability tending to one and is jointly \(\sqrt N\)-consistent for \((\beta_0,h_0)\).
\end{enumerate}
\end{thm}

The integer and cyclic models exhibit different behavior. In the integer model, the number of $3$-APs containing a given integer is affected by its distance from the boundary. This creates substantial variation in the row sums, with $\sum_{i=1}^N (\cR_i -\bar \cR)^2 = \Omega(N)$ and Theorem~\ref{thm:positive-res} therefore gives consistency of the pseudolikelihood estimator. 

In the cyclic model, translation invariance makes all row sums equal. Moreover, the tensor satisfies the mean-field condition~\eqref{cond:GW} and the spectral-gap condition~\eqref{cond:spec-gap}; see Lemmas~\ref{lem:3ap_mean} and~\ref{lem:3ap_gap}. Theorem~\ref{thm:est-imposs} then implies that $T_N(X)=o_p(1)$ in the ferromagnetic regime. On the other hand, as in the edge-triangle ERGM, sufficiently strong antiferromagnetic interactions force near-maximizers of the variational problem to remain far from constant vectors. Establishing this fact for cyclic $3$-APs requires a more delicate argument, given in Lemma~\ref{lem:3AP}. It follows that pseudolikelihood estimation is still consistent in the strongly antiferromagnetic regime.

\color{black}
\subsection{Inhomogeneous Random Hypergraphs}\label{ssec:inhomogeneous-hypergraph}

Our final example comes from inhomogeneous random hypergraphs. In latent-kernel models,
edge probabilities are governed by latent vertex positions rather than by a single global
density; inhomogeneous random hypergraphs give the analogous framework for higher-order interactions, where
the probability of a hyperedge depends on the latent types of all participating vertices
\cite{bollobas2007phase,lovasz2012large,balasubramanian2021nonparametric,elek2012measure,zhao2015hypergraph}.
Here we use such a kernel only as a source of structured random cubic tensors. Conditional
on the realized tensor, the statistical problem is to estimate the two Gibbs parameters
\((\beta,h)\), not the kernel \(g\) itself.

Let \(g:[0,1]^3\to[0,1]\) be a symmetric continuous function. For each unordered triple
\(\{i,j,k\}\subset[N]\) with distinct entries, let
\[
\xi_{\{i,j,k\}}\sim {\rm Ber}\bigl(g(i/N,j/N,k/N)\bigr)
\]
independently over all such triples. Define \(G_N(i,j,k)=\xi_{\{i,j,k\}}\) whenever
\(i,j,k\) are distinct, extend this definition symmetrically over permutations of the indices,
and set \(G_N(i,j,k)=0\) whenever \(\{i,j,k\}\) are not all distinct. We then define
\(
A_{ijk}:=N^{-2}G_N(i,j,k).
\)
Conditional on \(A\), write \(P^A_{\beta,h}\) for the Gibbs measure defined by~\eqref{eq:model}.

The relevant deterministic profile is
\[
G(x):=\int_{[0,1]^2}g(x,y,z)\,dy\,dz,\qquad
\bar G:=\int_{[0,1]^3}g(x,y,z)\,dx\,dy\,dz .
\]
This profile gives the limiting row sums of the tensor. If $G$ is nonconstant, these row sums vary across coordinates, creating enough variation in the local fields for joint estimation. If $G\equiv \bar G$ and $g$ is uniformly positive, the tensor is asymptotically homogeneous and well connected. In this case, pseudolikelihood estimation becomes ill-conditioned in the ferromagnetic regime with a nonnegative external field. The precise statement is as follows.

\begin{thm}
\label{thm:app-hypergraph}
Let $\mathbb P_A$ denote the law of the random tensor $A$. The following conditional conclusions hold on an event with
$\mathbb P_A$-probability tending to one  for $X\sim P^A_{\beta_0,h_0}$.
\begin{enumerate}
\item[(a)] If $G$ is not constant and $\Lambda'(h_0)\neq0 $, then $T_N(X)=\Omega_{P_{\beta_0, h_0}^A}(1).$
Consequently, the MPLE exists with $P_{\beta_0, h_0}^A$ probability tending to one and is jointly \(\sqrt N\)-consistent for $(\beta_0,h_0)$. 

\item[(b)] Assume that \(G\equiv \bar G\), that \(\inf_{[0,1]^3}g>0\), and that \(\mu\) is stochastically nonnegative. If \(\beta_0>0\), \(h_0\ge0\), and either \(h_0>0\) or \(\kappa_-\ge0\), then $ T_N(X)=o_{P_{\beta_0, h_0}^A}(1).$
\end{enumerate}
\end{thm}

\begin{rem}
Note that the conclusion of Theorem~\ref{thm:app-hypergraph} (b) does not directly imply that estimation using pseudo-likelihood is impossible. However, extending the contiguity arguments of \cite{GM20}, one should be able to show that consistent estimation of both parameters is impossible in the above setting  using {\it any} estimators, whenever the function $G$ is constant.  Since the arguments would be very similar, we skip it here. 
\end{rem}

\section{Preliminaries}\label{sec:prelim}
In this section, we first record and prove some necessary concentration results. Then we introduce two important lemmas which gives simpler criteria for mean-field and well-connectivity conditions. Finally, we record some auxiliary lemmas used in other proofs.
\subsection{Concentration Results}
We begin with the following definition of local perturbation of a configuration in $\bR^N$. 
\begin{defn}
Fix a constant $a_*: = \frac{\kappa_-+\kappa_+}{2}$, for a vector $x\in [\kappa_-, \kappa_+]^N$, let $ x^{(i)}\in [\kappa_-, \kappa_+]^N$ be the vector obtained by replacing the $i$-th coordinate of $x$ by $a_*$, i.e.
$x_k^{(i)} = x_k$ if $k \neq i$ and $x_k^{(k)} = a_*.$
\end{defn}
\begin{lem}\label{lem:unify-concen}
Let $\{g_i\}_{i=1}^N$ be a collection of functions defined on $[\kappa_-,\kappa_+]^N$. Assume it further satisfies the following properties,
\begin{enumerate}
    \item [(i)] For every $i \in [N]$ there exists a finite constant $d_i$, such that  $\norm{g_i}_{\infty} \le d_i.$
    \item [(ii)] For each $i,j\in [N]$, there exists finite constant $c_{ij}\ge 0$ such that 
    \[
    \max_{x \in [\kappa_-, \kappa_+]^N} \big| g_i(x) - g_i(x^{(j)})\big| \le c_{ij}.
    \]
    
\end{enumerate}
Then for $X\sim P_{\gb_0, h_0}$, under the Assumption~\ref{ass:standing}, with $b_i(X) = \gL'(\gb_0 m_i(X) +h_0) $ we have 
\begin{align*}
      \E\big[ \big(\sum_{i} g_i(X) (X_i - b_i(X)) \big)^2\big] \le 
      4\kappa^2 \Big((1+ \gamma |\beta_0| \kappa^3)\sum_i d_i^2+\sum_{i, j} (c_{ii}d_j+d_i c_{ji}) \Big),
\end{align*}

\end{lem}
\begin{proof}
We expand the square as
\begin{align}\label{eq:prelim}
     \E\big[ \big(\sum_ig_i(X) (X_i - b_i(X)) \big)^2\big] & =  \sum_{i, j} \E \Big[  ( g_i(X) g_j(X) (X_i - b_i(X))(X_j - b_j(X))\Big].
     \end{align}
Now consider the cases when $i=j$ and $i\ne j$ separately.
\begin{itemize}
\item{$i=j$:} In this case, using the fact $ \abs{x_i - b_i(x)} \le 2\kappa $ for each $i\in [N]$, the expectation in the RHS of \eqref{eq:prelim} is bounded by $4\kappa^2d_i^2$.

    \item{$i\ne j$:} In this case, the expectation in the RHS of \eqref{eq:prelim} can be decomposed as follows:
\begin{align}\label{eq:prelim2}
\begin{split}
    &  
\E\left[(g_i(X)-g_i(X^{(i)}))g_j(X)(X_i - b_i(X))(X_j - b_j(X))\right]
    \\
    +&
\E\left[  g_i(X^{(i)})(g_j(X)-g_j(X^{(i)}))(X_i - b_i(X))(X_j - b_j(X))\right]\\
    +& \E\left[  g_i(X^{(i)})g_j(X^{(i)})(X_i - b_i(X))(X_j-b_j(X^{(i)}))\right]\\
    +& \E\left[  g_i(X^{(i)})g_j(X^{(i)})(X_i - b_i(X))(b_j(X^{(i)}) - b_j(X))\right].
\end{split}
\end{align}
\end{itemize}
Since $ \E[X_i \mid X^{(i)}] = b_i(X) $ and $ X^{(i)} $ does not depend on $ X_i $, it follows that, for $ i \ne j $,
\begin{align*}
\E \big[ g_i(X^{(i)}) g_{j}(X^{(i)}) (X_i - b_i(X))(X_j - b_j(X^{(i)}))\big \rvert 
X^{(i)}\big] = 0,
\end{align*}
which implies that the third term in the RHS of \eqref{eq:prelim2} equals $0$. Using the bounds stated in the lemma, the sum of the first, second and fourth terms in the RHS of \eqref{eq:prelim2} can be bounded by
\begin{align*}
4\kappa^2\big (c_{ii}d_j+d_ic_{ji} \big) + 4 |\beta_0| \kappa^5 d_id_j\mathcal{R}_{ji},
\end{align*}
where, to bound the fourth term, we use the facts that $|x_i -b_i(x)|\le 2 \kappa$ and 
\begin{align}\label{eq:bound-tanh}
  |b_j(x)-b_j(x^{(i)})|\le |\beta_0| \kappa^2 |m_j(x)-m_j(x^{(i)})| \le 
  2|\beta_0|\kappa^4\sum_k A_{jik}=2|\beta_0| \kappa^4\mathcal{R}_{ji},
\end{align}

where the first inequality in \eqref{eq:bound-tanh} follows from the fact 
\begin{align}\label{eq:lambda-lipsch}
    |\gL'(a') - \gL'(a)| \le \sup_{\gl} \gL''(\gl) \cdot |a - a'| \le \kappa^2|a - a'|.
\end{align}
The bound stated in the lemma follows on combining the above bounds, and using Assumption \ref{ass:standing} to get
 $ \sum_{i, j} d_i d_j \mathcal{R}_{ji}\le \|\mathcal{R}\|_{\mathrm{op}} \|{\bf d}\|_2^2\le \gamma \sum_i d_i^2.$
\end{proof}
Applying the above lemma gives second-moment control for specific functions of the 3-tensor Ising model.
\begin{cor}\label{lem:concen}
Let $ X \sim P_{\gb_0,h_0} $. Under Assumption~\ref{ass:standing}, for any family $(\phi_i)_{i\in[N]}$ of bounded Lipschitz functions $ \phi_i:[-\gamma \kappa^2, \gamma\kappa^2]\to\mathbb{R} $ and for any $ a \in\mathbb{R}^N $, we have
    \begin{align}
        & \mathbb E\Big[\Big(\sum_i (X_i-b_i(X))\phi_i(m_i(X))\Big)^2\Big]
\le C_1 \sum_i \|\phi_i\|_\infty^2 + C_2\sum_i \|\phi_i\|_{\mathrm{Lip}}^2, \label{eq:concen1}\\
        & \mathbb E\Big[\Big(\sum_i a_i(X_i-b_i(X))\Big)^2\Big]
\le C_1 \sum_i a_i^2 , \label{eq:concen2}\\
        & \E  \Big[ \Big( \sum_{i,k} \cR_{i k} X_k (X_i - b_i(X)) \Big)^2\Big] \le C_3 N,\label{eq:concen3}\\
   &  \E\Big[ \Big( \sum_{i,j,k} A_{ijk}(X_iX_jX_k -b_i(X) b_j(X) b_k(X)) \Big)^2 \Big] \le C_4 N + C_5 N \mathrm{Tr}(\mathcal{R}^2). \label{eq:concen4}
    \end{align}
    where $C_1, C_2, C_3,C_4$ and $C_5$ are constants that only depend on $\beta_0, \gamma, \kappa.$ In particular, one can take
\begin{align*}
C_1=  4\kappa^2\bigl(1+\gamma|\beta_0|\kappa^3+\gamma\kappa^2\bigr),  \ \ \ C_2 = 4\gamma\kappa^4,  \ \ \  C_3 = (8+4\gamma |\gb_0| \kappa^3)\gamma^2 \kappa^4, \\
C_4 = 12\gamma\kappa^4 + 48\gamma^2\kappa^6 + 12\gamma^3\kappa^8 + 108|\beta_0|\gamma^3\kappa^9, \ \ \ 
C_5 = 72|\beta_0|\gamma\kappa^9.
\end{align*}

If we further assume the mean-field condition~\eqref{cond:GW}, then 
\eqref{eq:concen4} can be bounded as
\begin{align}
    \E\Big[ \Big( \sum_{i,j,k} A_{ijk}(X_iX_jX_k -b_i(X) b_j(X) b_k(X)) \Big)^2 \Big]  = o( N^2 (\log N)^{-1}). \label{eq:concen5}
\end{align}
\end{cor}

\begin{proof}
    Note that \eqref{eq:concen2} is an immediate consequence of 
    \eqref{eq:concen1} once we take $\phi_i \equiv a_i.$
    The proofs of~\eqref{eq:concen1} and \eqref{eq:concen3} follow from the application of Lemma~\ref{lem:unify-concen} by setting $g_i(x) = \phi_i(m_i(x))$ and $ \sum_{k} \cR_{ik} x_k$ respectively. For~\eqref{eq:concen1}, we have 
    $\abs{g_i(x)} = \abs{ \phi_i(m_i(x))} \le \norm{\phi_i}_{\infty}$.
For $i,j\in[N]$,
\[
|g_i(x)-g_i(x^{(j)})|
\le
\|\phi_i\|_{\mathrm{Lip}}\, |m_i(x)-m_i(x^{(j)})|
\le
2\kappa^2 \|\phi_i\|_{\mathrm{Lip}} \cR_{ij},
\]
so we may take $c_{ij}:=2\kappa^2 \|\phi_i\|_{\mathrm{Lip}} \cR_{ij}.$
Lemma~\ref{lem:unify-concen} therefore gives
\[
\mathbb E\Big[\Big(\sum_i (X_i-b_i(X))\phi_i(m_i(X))\Big)^2\Big]
\le
4\kappa^2(1+\gamma|\beta_0|\kappa^3)\sum_i \|\phi_i\|_\infty^2
+
8\kappa^4 \sum_{i, j} \|\phi_i\|_\infty \|\phi_j\|_{\mathrm{Lip}} \cR_{ji}.
\]
Using $2uv\le u^2+v^2$ and the fact $\sum_j \cR_{ij}\le \gamma$ for every $i$, we have 
\[
\sum_{i, j} \|\phi_i\|_\infty \|\phi_j\|_{\mathrm{Lip}} \cR_{ji}
\le
\frac{\gamma}{2}\sum_i \|\phi_i\|_\infty^2
+
\frac{\gamma}{2}\sum_i \|\phi_i\|_{\mathrm{Lip}}^2.
\]
Substituting this into the previous bound yields \eqref{eq:concen1}. 

If $g_i(x)=\sum_k \cR_{ik} x_k$, then $d_i=\kappa \gamma$ and $c_{ij}=\kappa \cR_{ij}$, which yields the bound \eqref{eq:concen3}.

Finally, we prove the bound~\eqref{eq:concen4}. Using the symmetry of the tensor $A$, we write
\begin{align*}
 \sum_{i,j,k} A_{ijk}(X_iX_jX_k &-b_i(X) b_j(X) b_k(X)) 
=  \sum_{i,j,k}A_{ijk}(X_i - b_i(X)) X_j X_k\\
&+ \sum_{i,j,k} A_{ijk}(X_i -b_i(X)) b_j(X) X_k 
 +\sum_{i,j,k} A_{ijk}(X_i -b_i(X))b_j(X) b_k(X).
\end{align*} 
Using $(u+v + w)^2 \le 3(u^2+ v^2+ w^2)$, it suffices to bound the second moment of each term on the RHS. For the first term, note that $\sum_{j,k} A_{ijk} x_j x_k = m_i(x) $,  we can apply \eqref{eq:concen1} with $\phi_i(t) = t$ with $\norm{\phi_i}_\infty = \gamma \kappa^2 $ and $\norm{\phi_i}_{\mathrm{Lip}} = 1 $ to obtain 
\begin{align}
    \E \Big[ \Big( \sum_{i,j,k} A_{ijk}X_jX_k (X_i - b_i(X)) \Big)^2\Big]  \le 4N\gamma^2\kappa^6\bigl(1+\gamma|\beta_0|\kappa^3+\gamma\kappa^2\bigr)
+ 4N\gamma\kappa^4. \label{eq:first-term}
\end{align}
For the second term, we invoke Lemma~\ref{lem:unify-concen} with $g_i(x) = \sum_{j,k} A_{ijk} b_j(x)x_k$. Observe that $\norm{g_i}_{\infty} \le \kappa^2 \gamma =: d_i$ and 
\begin{align*}
    \big| g_i(x)  - g_i(x^{(\ell)})\big|
    & \le  \sum_{j} A_{i  j \ell} |b_j(x) x_\ell|
      + \sum_{j, k}A_{ijk}\big|  b_{j}(x) - b_j(x^{(\ell)})\big| |x_k|  \\
    & \stackrel{\eqref{eq:bound-tanh}}{\leq}
      \kappa^2 \cR_{i\ell}
      + 2|\beta_0| \kappa^5\sum_{j} \cR_{ij} \cR_{j\ell} =
      \kappa^2\cR_{i\ell}
      + 2|\beta_0| \kappa^5 (\cR^2)_{i\ell}
      =: c_{i\ell}.
\end{align*}
Then we have
\begin{align}
   & \E \Big[ \Big( \sum_{i,j,k} A_{ijk}(X_i -b_i(X)) b_j(X) X_k \Big)^2 \Big] \nonumber \\
\le & 4\kappa^2\Bigg(
(1+\gamma|\beta_0|\kappa^3)N\gamma^2\kappa^4
+\sum_{i,j}\Big(2|\beta_0|\gamma\kappa^7(\cR^2)_{ii}
+\gamma\kappa^4\cR_{ji}
+2|\beta_0|\gamma\kappa^7(\cR^2)_{ji}\Big)
\Bigg) \nonumber\\
= &
4\kappa^2\Big(
(1+\gamma|\beta_0|\kappa^3)N\gamma^2\kappa^4
+2|\beta_0|\gamma\kappa^7 N \mathrm{Tr}(\cR^2)
+\gamma\kappa^4 \mathbf{1}^T \cR \mathbf{1}
+2|\beta_0|\gamma\kappa^7 \mathbf{1}^T \cR^2 \mathbf{1}
\Big).
\label{eq:second-term}
\end{align}
Similarly, we can bound the second moment of the third term. In the setting of Lemma~\ref{lem:unify-concen}, take $g_i(x) = \sum_{j, k}A_{ijk} b_j(x) b_k(x)$. Note that $\norm{g_i}_{\infty} \le \kappa^2\gamma =: d_i$
and

\begin{align*}
    \bigl|g_i(x)-g_i(x^{(\ell)})\bigr|
    &\leq
    \sum_{j,k}A_{ijk}
    \bigl|b_j(x)-b_j(x^{(\ell)})\bigr|\,|x_k|
    +
    \sum_j A_{ij\ell}
    |b_j(x^{(\ell)})|\,|x_\ell-a_*|  \\
    &\leq
    \kappa^2 \cR_{i\ell}
    +
    2|\beta_0|\kappa^5
    \sum_j \cR_{ij}\cR_{j\ell} 
    =\kappa^2\cR_{i\ell}
    +
    2|\beta_0|\kappa^5(\cR^2)_{i\ell}
    =:c_{i\ell}.
\end{align*}
Thus we have 
\begin{align}
& \E \Big[ \Big( \sum_{i,j,k} A_{ijk} (X_i - b_i(X)) b_j(X) b_k(X) \Big)^2 \Big] \nonumber \\
\le &
4\kappa^2 \Big(
(1+\gamma|\beta_0|\kappa^3)N\gamma^2 \kappa^4
+4|\beta_0|\gamma \kappa^7 \sum_{i, j} \big( (\cR^2)_{ii}+ (\cR^2)_{ji} \big)
\Big)\nonumber \\
= &
4\kappa^2 \Big(
(1+\gamma|\beta_0|\kappa^3)N\gamma^2\kappa^4
+4|\beta_0|\gamma \kappa^7 N \mathrm{Tr}(\cR^2)
+4|\beta_0|\gamma \kappa^7 \mv{1}^T \cR^2 \mv{1}
\Big).
\label{eq:third-term}
\end{align}
Combining \eqref{eq:first-term}, \eqref{eq:second-term} and \eqref{eq:third-term},  the bound \eqref{eq:concen4} now follows from the fact that $\mathbf{1}^T \mathcal{R} \mathbf{1} \le N \| \mathcal{R} \|_{\mathrm{op}} \le \gamma N$ and  $\mathbf{1}^T \mathcal{R}^2 \mathbf{1} \le N \| \mathcal{R} \|^2_{\mathrm{op}} \le \gamma^2 N.$

To obtain \eqref{eq:concen5}, note that  the mean-field condition~\eqref{cond:GW} and the inequality in Lemma~\ref{lem:trace-optr} yields $\mathrm{Tr}(\cR^2) = o(N (\log N)^{-1})$. 
\end{proof}

\subsection{Mean-Field and Connectivity Criteria}

Recall that $\cR$ is the symmetric $N\times N$ matrix with entries
$\cR_{ij}=\sum_k A_{ijk}$, and that $A_i=(A_{ijk})_{j,k\in[N]}$ denotes the $i$-th matrix slice of the tensor. The next lemma relates the mean-field quantity $\big\|\sum_i A_i^2\big\|_{\rm op}$ to the squared Frobenius norm of $\cR$. Part~{\rm (a)} gives the general lower bound $N\big\|\sum_i A_i^2\big\|_{\rm op}\ge \mathrm{Tr}(\cR^2)$. Although the reverse inequality does not hold in general, part~{\rm (b)} shows that the two quantities are comparable when the nonzero entries of $\cR$ are comparable. We call this the strong pseudo-regularity condition in \eqref{cond:strong-pseu-regu}.
\begin{lem}\label{lem:trace-optr}
Let $A$ be a symmetric $3$-tensor with non-negative entries. 
\begin{enumerate}
    \item[(a)] We have
     \[  \mathrm{Tr}(\cR^2) \le N \big \| \sum_i A^2_i \big \|_{\mathrm{op}}  \le N \| \cR\|^2_{\mathrm{op}}.\]
\item[(b)] Assume further that $A$ satisfies the row-boundedness condition in Assumption~\ref{ass:standing} and the non-degeneracy condition \eqref{cond:A_nontrivial}. Suppose there exist constant $K>0$ such that the following strong pseudo-regularity condition \eqref{cond:strong-pseu-regu} holds.
Then we have
\[
N\Big\|\sum_i A_i^2\Big\|_{\rm op}
\le
\frac{\gamma K}{\alpha}{\rm{Tr}}(\cR^2)
\]
where $\alpha$ is the lower bound of $\bar \cR$ in \eqref{cond:A_nontrivial} and $\gamma$ is the upper bound of row sum of $A$ in the Assumption~\ref{ass:standing}. In particular, if ${\rm{Tr}}(\cR^2) = o(N /\log N)$, then the mean-field condition~\eqref{cond:GW} holds.
\end{enumerate}
\end{lem}
 \begin{proof}
(a) We begin by proving the first inequality. Using the identity
$ A_i \mathbf{1} = \cR \mathbf{e_i}$, we have
\begin{align*}
    \mathbf{1}^T \big( \sum_i A_i^2 \big) \mathbf{1}  &= \sum_i  \mathbf{e_i}^T \cR^2 \mathbf{e_i} = \mathrm{Tr}(\cR^2).
\end{align*}
The inequality now follows from the fact that 
$\mathbf{1}^T \big( \sum_i A_i^2 \big) \mathbf{1} \le N \big \| \sum_i A^2_i \big \|_{\mathrm{op}}$.

For the second inequality, we note that since the entries of $\cR$ are non-negative,
 \[ \| \cR^2\|_{\mathrm{op}} = \sup_{ x \in \mathbb{R}_+^N: \| x\|_2=1  } \sum_{i, j} x_i x_j (\cR^2)_{i, j}.\]
 Using the decomposition $\cR = \sum_k A_k$, we then bound 
 \begin{align*}
\| \cR^2\|_{\mathrm{op}} &= \sup_{ x \in \mathbb{R}_+^N: \| x\|_2=1  } \sum_{i, j} x_i x_j \big( \sum_k  (A_k^2)_{i, j} + \sum_{k \ne k'}  (A_k A_{k'})_{i, j}  \big) \\
&\ge  \sup_{ x \in \mathbb{R}_+^N: \| x\|_2=1  } \sum_{i, j} x_i x_j \sum_k  (A_k^2)_{i, j}  = \big \| \sum_k A^2_k \big \|_{\mathrm{op}}.
 \end{align*}
 
 (b) Set $M:=\sum_i A_i^2 .$ Since $M$ has nonnegative entries, 
 $\|M\|_{\rm op}
\le
\max_j\sum_k M_{jk}.$

For each $j$,
\[
\sum_k M_{jk}
=
\sum_{i,\ell,k} A_{ij\ell}A_{i\ell k}
=
\sum_{i,\ell} A_{ji\ell}\cR_{i\ell}
\le
\left(\max_{a,b}\cR_{ab}\right)\cR_j
\le
\gamma \max_{a,b}\cR_{ab}.
\]
Let $\rho:=\min_{\cR_{ab}>0}\cR_{ab}.$
By the strong pseudo-regularity assumption \eqref{cond:strong-pseu-regu},
$\max_{a,b}\cR_{ab}\le K\rho$.
On the other hand,
\[
\mathrm{Tr}(\cR^2)
=
\sum_{a,b} \cR_{ab}^2
\ge
\rho\sum_{a,b} \cR_{ab}
=
\rho\sum_{a} \cR_a
\ge
\alpha N\rho.
\]
Therefore
\(
\max_{a,b}\cR_{ab}
\le
K\rho
\le
\frac{K}{\alpha}\frac{\mathrm{Tr}(\cR^2)}{N}.
\)
Combining the estimates gives
\[
\Big\|\sum_i A_i^2\Big\|_{\rm op}
=
\|M\|_{\rm op}
\le
\frac{\gamma K}{\alpha}
\frac{\mathrm{Tr}(\cR^2)}{N},
\]
which proves part~{\rm (b)}.
\end{proof}
Under the same the strong pseudo-regularity assumption \eqref{cond:strong-pseu-regu},
the  next lemma gives an easy-to-check sufficient condition, namely an {\em unweighted} spectral gap for the support graph of $\cR$, for the {\em weighted} spectral gap condition \eqref{cond:spec-gap}.
\begin{lem}\label{lem:trace-optr-converse}
Assume that $A$ satisfies the hypotheses of Lemma~\ref{lem:trace-optr}{\rm (b)}. Let
$B_{ij}:=\mathbf 1_{\{\cR_{ij}>0\}}$ denote the adjacency matrix of the support graph of $\cR$ on vertex set $[N]$.
Let $D_B:={\rm diag}(d_{1,B},\ldots,d_{N,B})$, where $d_{i,B}:=\sum_jB_{ij}$ is the degree of vertex $i\in[N]$. If this support graph has a spectral gap, that is, $ 1-\lambda_2(D_B^{-1}B)\ge \delta$
for some constant $\delta>0$, then $\cR$ satisfies the spectral gap condition \eqref{cond:spec-gap}. More precisely, with $D={\rm diag}(\cR_1,\ldots,\cR_N)$,
\[
    1-\lambda_2(D^{-1}\cR)\ge \frac{\delta}{K}.
\]
Moreover, if all off-diagonal entries of $\cR$ are positive and comparable in the sense that
\begin{align}\label{eq:RR}
   \max_{i\neq j}\cR_{ij} \le K \min_{i\neq j}\cR_{ij},
\end{align}
then $ 1-\lambda_2(D^{-1}\cR)\ge K^{-1}.$
\end{lem}

\begin{proof}
 By the variational characterization of the spectral gap,
\[
1-\lambda_2(D^{-1}\cR)
=
\inf_{u\not\equiv {\rm const}}
\frac{\frac12\sum_{i, j} \cR_{ij}(u_i-u_j)^2}
     {\min_{a\in\mathbb R}\sum_i \cR_i(u_i-a)^2}.
\]
It follows from \eqref{cond:strong-pseu-regu} that $\rho B_{ij}\le \cR_{ij}\le K\rho B_{ij}$ where $\rho:=\min_{\cR_{ab}>0}\cR_{ab}.$ Hence, 
we have $\cR_i\le K\rho d_{i,B}$ for every $i\in [N]$, and
\[
1-\lambda_2(D^{-1}\cR)
\ge
\frac1K
\inf_{u\not\equiv {\rm const}}
\frac{\frac12\sum_{i, j} B_{ij}(u_i-u_j)^2}
     {\min_{a\in\mathbb R}\sum_i d_{i, B}(u_i-a)^2}
\ge
\frac{\delta}{K}.
\]
 This proves the first assertion.

Finally, if \eqref{eq:RR} holds, 
then the support graph $B$ is the complete graph and the strong pseudo-regularity condition holds. The  complete graph has spectral gap  $1+\frac{1}{N-1}\ge 1=\delta$, so the previous bound gives $ 1-\lambda_2(D^{-1}\cR)\ge K^{-1},$ which completes the proof of the lemma.
\end{proof}

\subsection{Auxiliary Lemmas
}\label{ssec:auxiliary}
We include some auxiliary lemmas in this subsection. All proofs are included in the Appendix. The first one is about the entropy function for the general reference measure $\mu$ defined in Section~\ref{sec:intro}.
\begin{lem} \label{lem:rate_as_finite}
For $\mu$-a.s. every $x$, $ I(x) < \infty.$
\end{lem}
\begin{lem}\label{lem:low_complex}
    For any sequence of non-negative reals $\{\varepsilon_N\}_{N\ge 1}$ with $\eps_N \to 0$ as $N\to \infty$, the set $E_N=\{y\in [\kappa_-,\kappa_+]^N:\sum_i(y_i-\bar{y})^2\le N\varepsilon_N\}$ is of low complexity, in the sense of Definition \ref{def:low_complex}.
\end{lem}

We record some elementary inequalities that will be used in later proofs.
\begin{lem}\label{lem:aux-ineq1}
    For any $a, b, c \ge 0,$
    \[ \tfrac{1}{3}\big( a^3 + b^3+ c^3 \big )  - abc \ge \tfrac{1}{6} \big [ (a^{3/2} - b^{3/2})^2 + (b^{3/2} - c^{3/2})^2+ (c^{3/2} - a^{3/2})^2\big ]. \]
\end{lem}

\begin{lem}\label{lem:elementary}
 Fix $c > 1$ and $N \ge 1.$   Let $y_1, y_2, \ldots, y_N \ge 0$ and $z_i = y_i^{c}$ for all $1 \le i \le N$. Then 
    \[ N^{-1} \sum_i (y_i - \bar y)^2 \le \big( N^{-1} \sum_i (z_i - \bar z)^2\big)^{1/c}.  \]
\end{lem}

\section{Proof of Theorem~\ref{thm:positive-res}}
\label{sec:proof of positive-res}

\begin{lem}\label{lem:approx}
    Under the Assumption~\ref{ass:standing}. For $X \sim P_{\gb, h}$. The following statements hold. 
    \begin{itemize}
        \item [(i)] For $z \in [\kappa_-, \kappa_+]^N, \ w \in (\kappa_-, \kappa_+)^N$, let 
        %
        \[ g(z, w):= \sum_i I(w_i) + (z_i-w_i)\Phi(w_i) = \sum_i z_i \Phi(w_i) - \gL(\Phi(w_i)). \]
        Then we have $g(X, b(X)) - I(b(X)) = o_p(N)$, where $I(x) = \sum_i I(x_i)$ with $I(x_i)$ is defined in~\eqref{eq:entropy}. 
        \item [(ii)] Let $\{\eps_N\}_{N\ge 1}$ be any sequence of non-negative reals converging to $0$ as $N \to \infty$, and let 
        \[
        \cA_N:= \{x \in [\kappa_-, \kappa_+]^N: T_N(x) \le \eps_N\}.
        \]

\noindent Then we have
        \begin{align*}
         1_{\{X \in \cA_N\}} ( f(b(X)) - f(X)) = o_p(N), \quad \text{and}
\quad         1_{\{X \in \cA_N\}} ( f(\bar{X}\mv1 ) - f(X)) = o_p(N).
         \end{align*}

         \item[(iii)]
         If $\Lambda'(h)\ne 0$, then there exists $\eta>0$ such that
         \[
\bP(\mathcal A_N,\ \bar b(X)\notin J_\eta)\to0,
\qquad J_\eta:=[\kappa_-+\eta,\kappa_+-\eta]\setminus[-\eta,\eta].
\]
        
    \end{itemize}
\end{lem}

\begin{proof}

\noindent (i)
Invoking Corollary \ref{lem:concen}, there exists a sequence of positive reals $(\eps'_N)_{N\ge 1}$ with $\eps_N' \to 0$ as $N\to \infty$, such that setting
\begin{align}
\begin{split}
 & \cA^{(1)}_N:= \big\{x\in [\kappa_-, \kappa_+]^N: \big|  \sum_i   (x_i  - b_i(x)) \big|  \le \eps'_N N \big \},\label{eq:a1} \\
     & \cA^{(2)}_N:=\big  \{x \in [\kappa_-, \kappa_+]^N: \big| \sum_i (x_i  - b_i(x))m_i(x) \big|  \le \eps'_N N \big  \}, 
     \\ 
     & \mathcal{A}^{(3)}_N:= \big \{x\in [\kappa_-, \kappa_+]^N: \big| \sum_{j,k=1}^N\cR_{jk} (x_j  - b_j(x))x_k \big|  \le \eps'_N N \big \}, \\ 
    &\cA^{(4)}_N:=\big  \{ x\in[\kappa_-, \kappa_+]^N: \big|  \sum_{k} \cR_k(x_k -b_k(x)) \big|  \le \eps'_N N \big  \},
\end{split}
\end{align}
we have
\begin{align}\label{eq:a0}
   \lim_N \mathbb{P}\big(X\in \cap_{i=1}^4\cA_N^{(i)}\big)=1.
\end{align}

Noting that $I(w) = g(w,w)$, by definition
\begin{align*}
    \abs{g(x,b(x)) - I(b(x))} & = \Big |\sum_i (x_i-b_i(x))\Phi(b_i(x))\Big| = \Big | \sum_i (x_i -b_i(x)) (\gb m_i(x) +h)\Big|.
\end{align*}
where the second equality uses the fact that $$\Phi(b_i(x)) = \Phi(\gL'(\beta m_i(x) +h)) = \beta m_i(x) +h.$$ On the set $x\in  \mathcal{A}^{(1)}_N \cap \cA^{(2)}_N$, the RHS above is bounded by $(|\beta|+|h|)N\varepsilon_N'$, and so the conclusion of part (i) follows.

\noindent (ii) Note that 
\begin{align}\label{eq:fbound}
| f(x) - f(b(x))| \le& \frac{|\gb|}{3} \Big| \sum_{i,j,k}A_{ijk}(x_ix_jx_k - b_i(x)b_j(x) b_k(x))\Big| +\Big| h\sum_i(x_i - b_i(x))\Big|,\\
\label{eq:fbound0}\abs{f(x) - f(\bar{x}\mv1)} \le& \frac{|\gb|}{3} \Big| \sum_{i,j,k}A_{ijk}(x_ix_jx_k - \bar{x}^3)\Big| .
\end{align}
For $x \in \cA^{(1)}_N$ we have 
\begin{align}\label{eq:fbound1}
\Big|\sum_i(x_i - b_i(x))\Big|= o(N),
 \end{align}
 so it suffices to bound the first term in the RHS of \eqref{eq:fbound} and \eqref{eq:fbound0}. To this effect, noting that $\sum_{i,j,k}A_{ijk} x_ix_jx_k = \sum_{i} x_i m_i(x)$, by Cauchy-Schwarz, for $x \in \cA_N$, we have 
\begin{align}\label{eq:f-cs}
\Big|\sum_{i} x_i (m_i(x) - \bar{m}(x)) \Big| \le \sqrt{\sum_i x_i^2 }\sqrt{\sum_i \Big(m_i(x)-\bar{m}(x)\Big)^2}\le  \kappa N \sqrt{\eps_N}  = o(N).
\end{align}
For $x \in \cA^{(1)}_N$, using \eqref{eq:fbound1} we have
\begin{align}\label{eq:red-2}
\bar{m}(x)\sum_{i} (x_i - b_i(x)) = o(N). 
\end{align}
For $x \in \cA_N$, using the bound~\eqref{eq:lambda-lipsch} and the fact that $b_i(x)=\gL'(\beta m_i(x)+h)$, we get 
\begin{align}\label{eq:b}
    \sum_i (b_i(x) -\bar{b}(x) )^2 \le \abs{\gb}^2 \kappa^4\sum_i(m_i(x) - \bar{m}(x))^2 \le \abs{\gb}^2 \kappa^4 \eps_N N = o(N).
\end{align}
Based on the above fact and $\norm{\cR} \le \gamma$, we have $N\abs{\bar{m}(x)}=\abs{\sum_{j,k}\cR_{jk} x_jx_k}\le N\kappa^2\gamma$. Combining this with the preceding display gives
\begin{align}\label{eq:red-3}
    \Big| \sum_{i} (b_i(x)-\bar{b}(x)) \bar{m}(x) \Big|  = o(N).
\end{align}
For $x \in \cA_N^{(3)}$, we have 
\begin{align}\label{eq:red-4}
\Big| N\bar{m}(x)- \sum_{j,k} \cR_{jk}  b_j(x)x_k \Big|  =\Big|  \sum_{jk}\mathcal{R}_{jk}(x_j-b_j(x))x_k \Big| = o(N).
\end{align}
Also, Cauchy-Schwarz gives 
\begin{align}\label{eq:cs-2}
     \Big| \sum_{j,k} \cR_{jk} (b_j(x) - \bar{b}(x))x_k \Big|
   \le   \sqrt{\sum_{j=1}(b_j(x) - \bar{b}(x))^2} \cdot \sqrt{\sum_j\Big(\sum_{k} \cR_{jk}x_k\Big)^2} = o(N),
\end{align}
where the last step is based on~\eqref{eq:b} and $\abs{\cR_{j}} \le \gamma$ for $j \in [N]$.
Next, for $x \in \cA^{(4)}_N$, we have
\begin{align}\label{eq:red-5}
\Big| \sum_{k} \cR_{k} x_k -  \sum_{k} \cR_{k} b_k(x) \Big| = o(N)
\end{align}
Finally, by another application of Cauchy-Schwarz inequality along with~\eqref{eq:b}, we have 
\begin{align}\label{eq:red-6}
\Big| \sum_{k} \cR_{k} b_k(x) - \bar{b}(x) \sum_{k} \cR_{k} \Big| = o(N).
\end{align}
Combining the bounds in~\eqref{eq:red-4}, \eqref{eq:cs-2}, \eqref{eq:red-5}, \eqref{eq:red-6}, for $x\in \cap_{i=1}^4 \cA^{(i)}_N \cap \cA_N$ we get
\begin{align}\label{eq:cs-3}
|N\bar{m}(x)-N \bar{R}\bar{b}(x)^2|=o(N).
\end{align}
Using~\eqref{eq:f-cs},\eqref{eq:red-2},\eqref{eq:red-3}, \eqref{eq:cs-3}, for $x \in \cap_{i=1}^4 \cA^{(i)}_N \cap \cA_N $ we have
\begin{align}\label{eq:red-7}
\Big|\sum_{i,j,k} A_{ijk} x_ix_jx_k -\bar{b}(x)^3\sum_{i,j,k} A_{ijk} \Big| = o(N).
\end{align}
Moreover, for $x \in \cA^{(1)}_N$, by \eqref{eq:fbound1} and the facts that  $\sum_{i,j,k}A_{ijk}\le \gamma N$ and $| x_i|,|b_i(x)|\le \kappa$, 
 we have
\[
\left|\bar x^3-\bar b(x)^3\right|\sum_{i,j,k}A_{ijk}
\le 3\kappa^2\gamma N|\bar x-\bar b(x)|
=o(N).
\]
The bound for the first term in the RHS of \eqref{eq:fbound0} follows from \eqref{eq:red-7} combined with the above display. For bounding the first term in the RHS of \eqref{eq:fbound}, we use the following elementary inequality for $x,y,z, w \in \bR$:
\[
\abs{xyz - w^3}\le \abs{x-w} \abs{yz} +\abs{y-w}\abs{wz} + \abs{z-w} \abs{w}^2.
\]
Using this and the fact $\| \gL' \|_\infty \le \kappa$, we have 
\begin{align*}
\begin{split}
     \Big|\sum_{i,j,k} A_{ijk} \Big(b_i(x)b_j(x)b_k(x) -(\bar{b}(x))^3\Big)\Big|
   \le  &  3\kappa^2\sum_{i,j,k} A_{ijk}\abs{b_i(x) -\bar{b}(x)} \\
     \le & 3\kappa^2 \sqrt{\sum_{i}(b_i(x) - \bar{b}(x))^2} \cdot \sqrt{\sum_{i}\Big(\sum_{j,k}A_{ijk}\Big)^2} \\
      =&  o(N),
\end{split}
\end{align*}
where the last step is based on~\eqref{eq:b} and $\sum_{j,k}A_{ijk} \le \gamma$ in Assumption~\ref{ass:standing}. The above display, along with~\eqref{eq:red-7} gives
\[\sum_{ijk}A_{ijk}(x_ix_jx_k-b_i(x)b_j(x)b_k(x))=o(N),\]
which bounds the first term in the RHS of \eqref{eq:fbound}, and thus finishes the proof of part~(ii). \\

\noindent (iii) 
By~\eqref{eq:cs-3}, 
there exists a deterministic
sequence $\eta_N\to 0$ such that if we define $\cG_N:=\cap_{i=1}^4 \cA^{(i)}_N$, then 
\begin{align}\label{eq:barb_fp_6}
x\in \mathcal{A}_N\cap \mathcal{G}_N\Rightarrow |\bar{m}(x)-\bar{\mathcal{R}}\bar{b}(x)^2\Big|\le \eta_N, \text{ and }\bP(X \in \cG_N^c )\to 0.
\end{align}

Fix $x\in \cA_N\cap \cG_N$. Since $x\in\cA_N$, by \eqref{eq:lambda-lipsch} and Cauchy-Schwarz,
\begin{align}
\abs{\bar b(x)-\gL'(\gb \bar m(x)+h)}
&\le \frac1N \sum_i \abs{\gL'(\gb m_i(x)+h)-\gL'(\gb \bar m(x)+h)} \nonumber\\
&\le \frac{|\beta| \kappa^2}{N}\sum_i \abs{m_i(x)-\bar m(x)}
\le |\beta| \kappa^2 \sqrt{\eps_N}. \label{eq:barb_fp_5}
\end{align}

Using \eqref{eq:lambda-lipsch} once more, \eqref{eq:barb_fp_6} yields
\begin{equation}\label{eq:barb_fp_7}
\abs{
\gL'(\gb \bar m(x)+h)-\gL'(\gb \bar{\cR}\,\bar b(x)^2+h)
}
\le
|\beta| \kappa^2 | \bar m(x)-\bar{\cR}\,\bar b(x)^2|
\le  |\beta| \kappa^2\eta_N.
\end{equation}
Combining \eqref{eq:barb_fp_5} and \eqref{eq:barb_fp_7}, we obtain that on
$\mathcal A_N\cap\mathcal G_N$,
\[
\abs{\bar b(X)-\gL'(\gb \bar{\cR}\,\bar b(X)^2+h)}\le \delta_N, \text{ where }\delta_N := |\beta|\kappa^2\sqrt{\eps_N}+|\beta|\kappa^2\eta_N \to 0. 
\]
 Define
$F(t,r):=t-\Lambda'(\beta r t^2+h)$. Then  $F(0,r)=-\Lambda'(h) \ne 0$ for every
$r \in [0, \gamma]$. By continuity and compactness, there exist $\eta'>0$ and $c>0$ such that
$|F(t,r)|\ge c$ whenever $|t|\le\eta'$ and $r\in[0,\gamma]$. Since $|\delta_N|<c$ for all $N$ large enough and $\bar{\cR}\in[0,\gamma]$, it follows that 
$\mathcal A_N\cap\mathcal G_N\cap\{|\bar b(x)|\le\eta'\}=\emptyset$. Hence
\[
\limsup_N \bP(X \in \mathcal A_N, |\bar b(X)|\le\eta') \le \limsup_N \bP(X \in \mathcal A_N\cap\mathcal G_N^c) = 0.
\]

Set $M:=|\beta|\gamma\kappa^2+|h|$. By Assumption~\ref{ass:standing}, $|m_i(x)|\le \gamma\kappa^2$ for all $i\in[N]$, and hence $\beta m_i(x)+h\in[-M,M]$. Since $\Lambda'$ is continuous and takes values in $(\kappa_-,\kappa_+)$, there exists $\eta''>0$ such that $\Lambda'([-M,M])\subset[\kappa_-+\eta'',\kappa_+-\eta'']$. Therefore, $\bar b(x)\in[\kappa_-+\eta'',\kappa_+-\eta'']$ for every $x\in[\kappa_-,\kappa_+]^N$. Taking $\eta=\min\{\eta',\eta''\}>0$ gives
\[
\bP( X \in \mathcal A_N,\ \bar b(X)\in J_\eta^c)\to 0.
\]
\end{proof}

\begin{defn}\label{def:Qt}
Let $\mu,\Lambda,\Phi, \kappa_-,\kappa_+$ be as in Definition \ref{def:main-set-up}.
 For $t \in [\kappa_-, \kappa_+]$, let
\[
Q_t := \mu_{\Phi(t)}^{\otimes N}.
\]
Equivalently, under $Q_t$, the coordinates
$X_1, \dots, X_N$ are i.i.d.\ with common law $\mu_{\Phi(t)}$, and hence
$\mathbb{E}_{Q_t}[X_i] = t$ and $\mathrm{Var}_{Q_t}(X_i) = v(t),$
where $v(t) = \gL''(\Phi(t))$ for $t \in (\kappa_-,\kappa_+) $ and $v(\kappa_\pm) = 0.$
\end{defn}

\begin{lem}\label{lem:measure-change}
     Under the Assumption~\ref{ass:standing}, for $X \sim P_{\gb,h}$.  Suppose there exists $\eps_N \to 0$ and a sequence of low complexity sets $(E_N)_{N\ge 1}$ in the sense of Definition \ref{def:low_complex} such that 
    \begin{align}\label{eq:non-mean-field-cond}
    \liminf_{N} P_{\gb,h}(\abs{f(X)-f(b(X
    )) }\le N \eps_N \ \text{and} \ b(X)\in E_N)> 0.
    \end{align}
    Then the following holds.

   \begin{enumerate}
       \item[(i)] We have 
       \begin{align*}
       \log Z_N(\gb,h)  = \sup_{y \in [\kappa_-,\kappa_+]^N} \big( f(y) - I(y)\big) + o(N).
  \end{align*}
   \item [(ii)] Suppose that there exists a sequence of sets $(C_N)_{N \ge 1}$ with $C_N \subseteq[\kappa_-,\kappa_+]^N$ such that 
   \[\liminf_{N}\frac{1}{N}\Big (\sup_{y \in [\kappa_-,\kappa_+]^N} \big( f(y) - I(y)\big)-\sup_{y \in C_N}\big( f(y) - I(y)\big)\Big )>0.\] Then we have 
   \[
    \lim_{N}P_{\gb,h}(|f(X)-f(b(X))|\le N\varepsilon_N, b(X) \in C_N\cap E_N)=0.
   \]
   
   \item [(iii)]  Assume that $\Lambda'(h) \ne 0$ , and there exists $\eta >0$ such that 
   \[
   \limsup_{N\to\infty} \sup_{t \in J_\eta}\frac{1}{N} \log Q_t(\mathcal{A}_N)  <0,
   \]
   where  $Q_t$ as in Definition \ref{def:Qt}, $J_\eta$ and $\cA_N$ are as defined in Lemma~\ref{lem:approx}. Then 
   \[
   \lim_{N\to\infty} P_{\gb, h} (\mathcal{A}_N) = 0.
   \]
   \end{enumerate}
\end{lem}

\begin{proof}
We first prove a general bound, that we will use to verify both parts (i) and (ii).
By Lemma \ref{lem:approx} part (i), there exists a sequence of non-negative reals $(\varepsilon_N')_{N \ge 1}$ converging to $0$, such that
\begin{align}\label{eq:g0}P_{\gb,h}(|g(X,b(X))-I(b(X))|\le N\varepsilon_N')\to 1.
\end{align}
Let $D_N\subseteq [\kappa_-,\kappa_+]^N$ be arbitrary, and set \begin{align}
\begin{split}\label{eq:b0}
\mathcal{B}_N = \cB_N(D_N) := \big\{x\in [\kappa_-,\kappa_+]^N: & |f(x)-f(b(x))|\le N\varepsilon_N, \\
& b(x)\in  D_N\cap E_N, \ |g(x,b(x))-I(b(x))|\le N\varepsilon_N'\big\}.
\end{split}
\end{align}
Then we have
\begin{align}\label{eq:g1}
  \notag& \int_{x\in \cB_N} \exp(f(x)) d \mu^{\otimes N}  \\
   \notag & \le   \exp(N(\varepsilon_N+\varepsilon_N')) \int_{x\in \cB_N} \exp \big(f(b(x))+g(x,b(x)) - I(b(x))\big) d \mu^{\otimes N} \\
    & \le  \exp\Big(o(N) + \sup_{y\in D_N} \big( f(y) - I(y)\big) \Big)  \int_{x \in \cB_N} \exp(g(x,b(x))) d \mu^{\otimes N}.
\end{align}
Recall that from the proof of Lemma~\ref{lem:approx}(iii),
\[ b_i(x) = \gL'(\beta m_i(x) +h) \in [\kappa_{-} +\eta'', \kappa_{+} -\eta''] \quad \text{for some } \eta''>0. \]
Proceeding to estimate the last term in the RHS of \eqref{eq:g1}, note that if $x\in \cB_N$ then $b(x)\in E_N \cap [\kappa_{-} +\eta'', \kappa_{+} -\eta'']^N$, $E_N$ is a low complexity set. Thus, for every $\varepsilon>0$, there exists a set $\mathcal{D}_N(\varepsilon)$, obtained if necessary by projecting each point of a low-complexity net for $E_N$ coordinatewise onto the interval $[\kappa_{-}+\eta'',\kappa_{+}-\eta'']$, such that $\mathcal{D}_N(\varepsilon)$ is an $\varepsilon\sqrt{N}$-net of $E_N\cap[\kappa_{-}+\eta'',\kappa_{+}-\eta'']^N$ in Euclidean metric and $\log |\mathcal{D}_N(\varepsilon)|=o(N)$. Consequently, for every $x\in [\kappa_-,\kappa_+]^N$ such that $b(x)\in E_N$, there exists $y\in \mathcal{D}_N(\varepsilon)$ such that $\|b(x)-y\|_2^2\le N\varepsilon^2$.  Note that for function $g(u,t)=u\Phi(t) - \gL(\Phi(t))$, where $u \in [\kappa_-,\kappa_+], t \in [\kappa_- + \eta'',\kappa_+-\eta'']$, we have 
\[
\abs{\partial_t g (u, t)} = \abs{(u-t)\Phi'(t)} \le C.
\]
By the mean value theorem, and since $b_i(x),y_i \in [\kappa_-+\eta'',\kappa_+-\eta'']$, it follows that
\begin{align*}
    |g(x,b(x))-g(x,y)|  & =\Big| \sum_i  g(x_i,b_i(x))-g(x_i,y_i) \Big|\\
    & \le C\sum_i\big| b_i(x) - y_i \big| \le C\sqrt{N}\sqrt{\sum_i (b_i-y_i)^2}\le C N\varepsilon.
\end{align*}
This gives
\begin{align}\label{eq:g2}
   \notag \int_{x\in \cB_N} e^{g(x,b(x))} d\mu^{\otimes N}\le & \sum_{y\in \mathcal{D}_N(\varepsilon)} \int_{x: \|b(x)-y\|_2^2\le N\varepsilon^2} e^{g(x,b(x))} d\mu^{\otimes N}\\
    \notag\le &e^{CN\varepsilon}\sum_{y\in \mathcal{D}_N(\varepsilon)} \int_{x:\|b(x)-y\|_2^2\le N\varepsilon^2}e^{g(x,y)} d\mu^{\otimes N}\\
   \le &e^{CN\varepsilon}\sum_{y\in \mathcal{D}_N(\varepsilon)}  \int_{x\in [\kappa_-,\kappa_+]^N} e^{g(x,y)}d\mu^{\otimes N} = e^{CN\varepsilon}|\mathcal{D}_N(\varepsilon)|,
\end{align}
    where the last line uses the fact that for any $y \in [\kappa_{-}+\eta'',\kappa_{+}-\eta'']^N$ we have
    \begin{align*}
        \int_{x\in [\kappa_-, \kappa_+]^N}e^{g(x,y)} d \mu^{\otimes N}=1.
    \end{align*}
    Since $\log |\mathcal{D}_N(\varepsilon)|=o(N)$, combining \eqref{eq:g1} and \eqref{eq:g2} we get
    \begin{align*}\limsup_{N}\frac{1}{N}\Big[\log \int_{x\in \cB_N}\exp(f(x)) d\mu^{\otimes N}-\sup_{y\in D_N} \big( f(y)-I(y)\big)\Big]\le \varepsilon.
    \end{align*}
    Since $\varepsilon>0$ is arbitrary, this gives
    \begin{align}\label{eq:g3}
    \limsup_N \frac{1}{N}\Big[\log \int_{x\in \mathcal{B}_N}\exp(f(x)) d\mu^{\otimes N}-\sup_{y\in D_N} \big( f(y)-I(y)\big)\Big]\le 0.
    \end{align}
Armed with this estimate, we now start with the proof of part (i).

\noindent (i)
Let $\cB_N$ be as defined in~\eqref{eq:b0} with $D_N = [\kappa_-,\kappa_+]^N$. Then using \eqref{eq:non-mean-field-cond} and \eqref{eq:g0}, it follows that
\(
\delta :=\liminf_N \mathbb{P}(\mathcal{B}_N)>0.
\)
 Thus, for all $N$ large enough we have $\mathbb{P}(\mathcal{B}_N)\ge \delta/2$, which gives
\begin{align*}
   Z_N(\beta,h)=\int_{[\kappa_-,\kappa_+]^N}e^{f(x)} d\mu^{\otimes N}\le \frac{2}{\delta}\int_{x\in \mathcal{B}_N}e^{f(x)}d\mu^{\otimes N}.
\end{align*}
Since $\delta>0$ is fixed, taking $\log$, dividing by $N$, and letting $N\to\infty$ on both sides of the above equation, we get
\begin{align*}
    \limsup_N\frac{1}{N}\Big[\log Z_N(\beta,h)- \sup_{y\in [\kappa_-,\kappa_+]^N}\big( f(y)-I(y)\big)\Big]\le 0,
\end{align*}
where we use \eqref{eq:g3} with $D_N=[\kappa_-,\kappa_+]^N$. This gives the desired upper bound for the asymptotics of $\log Z_N(\beta,h).$
The lower bound follows by invoking~\cite[Theorem 1]{Yan20}, which gives
\[
\log Z_N(\gb,h) \ge \sup_{y\in[\kappa_-,\kappa_+]^N} \big( f(y) - I(y) \big).
\]
This completes the proof of part (i).

\noindent (ii)
Let $\cB_N$ be as defined in~\eqref{eq:b0} with $D_N = C_N$, where $C_N$ is in the statement of the lemma. By \eqref{eq:g0}, it suffices to show that $\mathbb{P}(\mathcal{B}_N)\to 0$. To this effect, note that
\begin{align*}
    \mathbb{P}(\mathcal{B}_N)=\frac{1}{Z_N(\beta,h)}\int_{x\in \mathcal{B}_N}e^{f(x)}d\mu^{\otimes N}.
\end{align*}
Invoking \eqref{eq:g3} with $D_N=C_N$, we obtain
\begin{align*}
 \log\Big(\int_{x\in \mathcal{B}_N}e^{f(x)}d\mu^{\otimes N}\Big) \le \sup_{y\in C_N}\big( f(y)-I(y)\big) + o(N).
\end{align*}
On the other hand, by part (i) above, we get
\begin{align*}
  \log Z_N(\beta,h) = \sup_{y\in [\kappa_-,\kappa_+]^N}\big( f(y)-I(y)\big) + o(N).
\end{align*}
Therefore, 
\[\frac{1}{N} \log \mathbb{P}(\mathcal{B}_N) \le \frac{1}{N}\Big[ \sup_{y\in C_N} \big( f(y)-I(y)\big) - \sup_{y\in [\kappa_-,\kappa_+]^N} \big( f(y)-I(y)\big) \Big] + o(1).   \]
By hypothesis of part (ii), we deduce that $ \limsup_N N^{-1} \log \mathbb{P}(\mathcal{B}_N) < 0.$ This completes the proof of part(ii).

\noindent (iii)  Assume, for contradiction, that
\(
\limsup_{N\to\infty}P_{\gb,h}(X \in \mathcal A_N)>0.
\)
Choose $\eta>0$ such that Lemma~\ref{lem:approx} part (iii) holds with $J_{2\eta}$.  Let $\mathcal A_N^{(1)}$ be as in \eqref{eq:a1}. By \eqref{eq:a0}, $\bP\big( X \not \in \mathcal A_N^{(1)}\big)\to 0.$ Also, by Lemma~\ref{lem:approx} part (ii), there exists a sequence
$\delta_N\downarrow0$ such that
\[
P_{\gb,h}\left( X \in 
\mathcal A_N, \abs{f(X)-f(\bar X\mathbf 1)}>N\delta_N
\right)\to0.
\]
Finally, by Lemma~\ref{lem:approx} part (iii), $P_{\gb,h} \left(X \in \mathcal A_N, \bar b(X)\notin J_{2\eta}\right)\to0.$
Define
\[
\cB_N
:=
\mathcal A_N
\cap \mathcal A_N^{(1)}
\cap
\left\{\abs{f(x)-f(\bar x\mathbf 1)}\le N\delta_N\right\}
\cap
\{\bar b(x)\in J_{2\eta}\}.
\]
Then, we conclude that $\limsup_{N\to\infty}\bP(X \in \cB_N)>0.$ 

We now estimate the probability of $\cB_N$ from above. If $x\in\cB_N$, then
$x\in\mathcal A_N^{(1)}$ gives
$\abs{\bar x-\bar b(x)} \le \varepsilon_N'$.
Since $\bar b(x)\in J_{2\eta}$, it follows that, for all sufficiently large $N$, $\bar x\in J_\eta.$ Moreover, on $\cB_N$,
\begin{align*}
f(x)\le f(\bar x\mathbf 1)+N\delta_N = \frac{\beta}{3}\bar x^3\sum_{i,j,k}A_{ijk}+hN\bar x + N\delta_N 
= N\beta_N\bar x^3+Nh\bar x+N\delta_N,
\end{align*}
where $\beta_N:=\beta\bar{\mathcal R}/3$. Therefore
\begin{align}\label{eq:b_int}
\notag&\int_{\cB_N}e^{f(x)}d\mu^{\otimes N}(x)
\notag\le 
e^{N\delta_N}
\int_{\cB_N}
\exp\{N\beta_N\bar x^3+Nh\bar x\}
d\mu^{\otimes N}(x)\\
\notag= &e^{N\delta_N}
\int_{\cB_N}
\exp\big(N\beta_N\bar x^3+Nh\bar x-NI(\bar{x})+\Phi(\bar{x})N\bar{x}-N\Lambda(\Phi(\bar{x})\big)
d\mu^{\otimes N}(x)\\
\le &e^{N\delta_N} \sup_{t\in J_\eta}
\exp\big( N(\beta_Nt^3+ht-I(t))\big)
\int_{\cB_N}
\exp\big(\Phi(\bar{x})N\bar{x}-N\Lambda(\Phi(\bar{x})\big)
d\mu^{\otimes N}(x).
\end{align}
We next use a discretization of $J_\eta$. Let $\mathcal P_L$ be a partition of $J_\eta$ into intervals of length at most $\frac{1}{L}$, and for each interval $K\in\mathcal P_L$ choose a point $t_K\in K$. The number of intervals is at most $2\kappa L$. Since $J_\eta$ is compactly contained in $(\kappa_-,\kappa_+)$, the functions
$\Phi$ and $I$ are bounded and Lipschitz on $J_\eta$. Hence, uniformly for $x$ with $\bar x\in K$,
\[
\Big|\Big(\Phi(\bar{x})N\bar x-N\Lambda(\Phi(\bar{x}))\Big)
-
\Big(\Phi(t_K)N\bar x-N\Lambda(\Phi(t_K))\Big)\Big|\le \frac{CN}{L}, 
\]
where the constant $C$ depends only on $\mu,\eta$.
Using the above display gives
\begin{align*}
   &\int_{\cB_N}
\exp\big(\Phi(\bar{x})N\bar{x}-N\Lambda(\Phi(\bar{x})\big)
d\mu^{\otimes N}(x)\\
\le & e^{\frac{CN}{L}} \int_{\cB_N}
\exp\big(\Phi(t_K)N\bar{x}-N\Lambda(\Phi(t_K)\big)
d\mu^{\otimes N}(x)\\
=&e^{\frac{CN}{L}} Q_{t_K}(\mathcal{B}_N)\le e^{\frac{CN}{L}}\sup_{t\in J_\eta}Q_t(\mathcal{B}_N).
\end{align*}
Combining the above bound along with \eqref{eq:b_int} gives
\begin{align}\label{eq:ub_BN}
\int_{\cB_N}e^{f(x)}d\mu^{\otimes N}(x)\le e^{o(N)+\frac{CN}{L}} \sup_{t\in J_\eta}
\exp\big( N\beta_Nt^3+Nht-NI(t)\big) \sup_{t\in J_\eta}Q_t(\mathcal{B}_N).
\end{align}


On the other hand, by part (i) of the lemma,
\[
\log Z_N(\beta,h)
=
\sup_{y\in[\kappa_-,\kappa_+]^N}\big( f(y)-I(y)\big)+o(N).
\]
In particular, by restricting the supremum to constant vectors $y=t\mathbf 1$,
\begin{equation}\label{eq:lb_ZN}
Z_N(\beta,h)
\ge
\exp\Big(
\sup_{t\in[\kappa_-,\kappa_+]}
\big(N\beta_Nt^3+Nht-NI(t)\big)+o(N)
\Big).
\end{equation}
After normalization, \eqref{eq:ub_BN} and \eqref{eq:lb_ZN} give
\begin{align*}
P_{\gb,h}(X \in \cB_N)
&=
\frac{1}{Z_N(\beta,h)}
\int_{\cB_N}e^{f(x)}d\mu^{\otimes N}(x)\le
e^{o(N)+\frac{CN}{L}}
\sup_{t\in J_\eta}Q_t(\cB_N).
\end{align*}
By the hypothesis, and since $\cB_N\subseteq \cA_N$, we have
\[
\limsup_{N\to\infty}\frac1N
\log P_{\gb,h}(X \in \mathcal{B}_N)\le \limsup_{N\to\infty}\frac1N
\log\sup_{t\in J_\eta}Q_t(\mathcal A_N)+\frac{C}{L}<0,
\]
where the last inequality holds for $L$ large enough. This contradicts $\limsup\limits_{N\to\infty}\mathbb{P}(X \in \mathcal{B}_N)>0$, and hence completes the proof.

\end{proof}

\begin{lem}\label{lem:iid-measure} 
    Assume that the tensor $A$ satisfies the Assumption~\ref{ass:standing} and for any $t\in (\kappa_-,\kappa_+)$ and $X \sim Q_t$. Then the following statements hold. 
    \begin{enumerate}
    \item [(i)]
    We have 
    \begin{align*}
    \E_{Q_t}\Big(\sum_i(m_i(X) - \bar{m}(X))^2\Big) & = t^4 \sum_i (\cR_i - \bar{\cR})^2 + 2\gL''(\Phi(t))^2  \sum_i\norm{A_i - \bar{A}}_F^2
      \\ 
      & \quad + 4t^2\gL''(\Phi(t)) \sum_{i, j}\big(  \cR_{ij} - \frac1N \cR_{j}\big)^2,
    \end{align*}
    where  the matrices $A_{i} := (A_{ijk})_{j,k\in[N]}$ and $\bar{A} := \frac1N \sum_i A_i$.
    \item [(ii)] 
    For any $\gl>0$, we have
    \[
    Q_t(\abs{T_N(X) - \E_{Q_t}(T_N(X))} \ge \gl ) \le 2 \exp(-cN\gl^2),
    \]
    where $c>0$ is some constant independent of $N$ and $t$.
    \end{enumerate}
\end{lem}

\begin{proof}
\noindent (i)   Note that $Q_t$ is an i.i.d.\ measure with marginal mean $t$, in the following proof, we use $\E$ for the expectation with respect to $Q_t$

\begin{align*}
    \E\Big( \sum_i (m_i(X) - \bar{m}(X))^2\Big) & = \sum_i \left(\text{Var}(m_i(X) - \bar{m}(X)) +[ \E(m_i(X) - \bar{m}(X)]^2 \right). 
\end{align*}
First, we note the following  
\begin{align*}
    \sum_{i} (\E(m_i(X) - \bar{m}(X)))^2 = \sum_{i} \Big(\sum_{j,k}( A_{ijk}  -\frac1N \cR_{jk} )t^2 \Big)^2 =t^4 \sum_{i} (\cR_i - \bar{\cR})^2.
\end{align*}
Now we expand the variance of $m_i(X) - \bar{m}(X)$,
\begin{align*}
    \text{Var} (m_i(X) - \bar{m}(X))& = \text{Var} \Big( \sum_{j,k}(A_{ijk} X_j X_k - \frac1N \cR_{jk} X_jX_k)\Big) \\ 
    & =  \sum_{j,k,j',k'} (A_{ijk} - \frac1N\cR_{jk})(A_{ij'k'} - \frac1N\cR_{j'k'}) \text{Cov}(X_jX_k, X_{j'} X_{k'}).
\end{align*}
The covariance term is given by
\begin{align*}
    \text{Cov}(X_jX_k, X_{j'} X_{k'}) = 
    \begin{cases}
    & 
    \Lambda''(\Phi(t))^2+2t^2  \Lambda''(\Phi(t)), \quad \text{if $\big|\{j,k\} \cap \{j',k'\}\big|=2$}, \\
    &  t^2\gL''(\Phi(t)), \quad  \text{if $\big|\{j,k\} \cap \{j',k'\}\big|=1$}, \\
    & 0, \quad \text{if $\big|\{j,k\} \cap \{j',k'\}\big|=0$}.
    \end{cases}
\end{align*}
There are two different ways to pair $j,j',k,k'$ such that $\big|\{j,k\} \cap \{j',k'\}\big|=2$, the corresponding contribution is 
\[
2(\gL''(\Phi(t))^2 + 2 t^2\gL''(\Phi(t))) \sum_{j,k} (A_{ijk} - \frac1N \cR_{jk})^2.
\]
On the other hand, there are four different ways to get $\big|\{j,k\} \cap \{j',k'\}\big|=1$, the contribution in this case is 
\[
4t^2\gL''(\Phi(t)) \Big(\sum_{j}\Big( \sum_{k} (A_{ijk} - \frac1N \cR_{jk})\Big)^2 -\sum_{k} (A_{ijk} - \frac1N \cR_{jk})^2 \Big).
\]
Putting them together and simplifying, it gives 
\begin{align*}
    \text{Var} (m_i(X) - \bar{m}(X)) & = 2\gL''(\Phi(t))^2  \sum_{j,k} (A_{ijk} - \frac1N \cR_{jk})^2 
      \\
      & \quad + 4t^2\gL''(\Phi(t)) \sum_{j}\Big( \sum_{k} (A_{ijk} - \frac1N \cR_{jk})\Big)^2\\
      & = 2\gL''(\Phi(t))^2  \sum_{j,k} (A_{ijk} - \frac1N \cR_{jk})^2
      + 4t^2\gL''(\Phi(t)) \sum_{j}\Big(  \cR_{ij} - \frac1N \cR_{j}\Big)^2 \\
      & = 2\gL''(\Phi(t))^2  \norm{A_i - \bar{A}}_F^2
      + 4t^2\gL''(\Phi(t)) \sum_{j}\Big(  \cR_{ij} - \frac1N \cR_{j}\Big)^2,
\end{align*}
where $A_i, \bar{A}$ defined in the lemma. Now summing over $i$, it gives the desired result.\\

\noindent (ii) 
For each $i\in[N]$, let $X_i'$ be an independent copy of $X_i$, independent of $X$, and define $X^{(i)} = (X_1, \ldots, X_{i-1}, X_i', X_{i+1},\ldots, X_N)$. For each $k\in [N]$, let $\Delta_k = \abs{m_k(X) - m_k(X^{(i)})}.$
Using $\abs{m_k(x)}\le \gamma\kappa^2$ for all $x$ and $k$, we get
\begin{align*}
\abs{T_N(X)-T_N(X^{(i)})}
&=
\Big|
\frac1N\sum_k\big(m_k(X)-\bar m(X)\big)^2
-
\frac1N\sum_k\big(m_k(X^{(i)})-\bar m(X^{(i)})\big)^2
\Big|  \\
&\le
\frac{4\gamma\kappa^2}{N}
\sum_k
\left(
\abs{m_k(X)-m_k(X^{(i)})}
+
\abs{\bar m(X)-\bar m(X^{(i)})}
\right).
\end{align*}
Since $\abs{\bar m(X)-\bar m(X^{(i)})}
\le
N^{-1}\sum_k\Delta_k$, 
it follows that
\[
\big| T_N(X)-T_N(X^{(i)}) \big|
\le
\frac{8\gamma\kappa^2}{N}\sum_{k}\Delta_k.
\]
Next, by the definition of $m_k$ and the fact that $\abs{X_i-X_i'}\le 2\kappa$,
\[
\Delta_k
\le
\abs{X_i-X_i'}
\Big(
\sum_{\ell} A_{ki\ell}\,|X_\ell|
+
\sum_j A_{kji}\,|X_j|
\Big)
\le
4\kappa^2\sum_{\ell} A_{ki\ell}.
\]
By Assumption~\ref{ass:standing},
\[
\sum_k\Delta_k
\le
4\kappa^2\sum_{k,\ell} A_{ki\ell}
= 4\kappa^2\sum_{k,\ell} A_{ik\ell} \le
4\gamma\kappa^2.
\]
Hence, for each $i\in[N]$,
\[
\sup_{X,X^{(i)}}\abs{T_N(X)-T_N(X^{(i)})}
\le
\frac{32\gamma^2\kappa^4}{N}.
\]
Applying the bounded difference inequality \cite[Theorem~6.2]{BLM13}, we obtain
\[
Q_t\left(\abs{T_N(X)-\E_{Q_t}T_N(X)}\ge \lambda\right)
\le
2\exp\Big(-\frac{2N\lambda^2}{1024\gamma^4\kappa^8}\Big)
\le
2\exp(-cN\lambda^2),
\]
where $c>0$ is a constant independent of $N$ and $t$.
\end{proof}

\begin{proof}[Proof of Theorem~\ref{thm:positive-res}]
Throughout the proof we write $\beta=\beta_0$ and $h=h_0$.
Assume, toward a contradiction, that $T_N(X)\ne \Omega_p(1)$. Then there exists
a deterministic sequence $\varepsilon_N\downarrow 0$ and a subsequence, still denoted
by $N$, such that
\[
\liminf_{N\to\infty}\mathbb P(\cA_N)>0,
\qquad
\cA_N:=\{x\in[\kappa_-,\kappa_+]^N:T_N(x)\le \varepsilon_N\}.
\]

We first verify the hypothesis of Lemma~\ref{lem:measure-change}. Since $\Lambda'$
is Lipschitz on bounded intervals, for $x\in \cA_N$,
\[
\sum_i (b_i(x)-\bar b(x))^2
\le C\sum_i (m_i(x)-\bar m(x))^2
\le C N\varepsilon_N.
\]
Thus, we have
$b(x)\in E_N$ on $\cA_N$, where
\[
E_N:=\Big\{y\in[\kappa_-,\kappa_+]^N:
\sum_i (y_i-\bar y)^2\le C\varepsilon_N N\Big\}.
\]
By Lemma~\ref{lem:low_complex}, $E_N$ is of low complexity.

Next, by Lemma~\ref{lem:approx}(ii), there exists a deterministic sequence $\delta_N\downarrow 0$ such that
\[
\liminf_{N\to\infty}
\mathbb P\left(\cA_N,\ |f(X)-f(b(X))|\le N\delta_N\right)>0.
\]
Since $b(X)\in E_N$ on $\cA_N$, it follows that
\[
\liminf_{N\to\infty}
\mathbb P\left(|f(X)-f(b(X))|\le N\delta_N,\ b(X)\in E_N\right)>0.
\]
Therefore, the hypothesis~\eqref{eq:non-mean-field-cond} of Lemma~\ref{lem:measure-change} is satisfied.

It remains to verify the exponential $Q_t$-bound required in Lemma~\ref{lem:measure-change}(iii). Let $\eta>0$ be the value supplied by Lemma~\ref{lem:measure-change}(iii). Since $J_\eta$
is compactly contained in $(\kappa_-,\kappa_+)$ and is bounded away from $0$, there exist constants $a_\eta,v_\eta>0$ such that, for every $t\in J_\eta$,
\[
|t|\ge a_\eta,
\qquad
\Lambda''(\Phi(t))\ge v_\eta.
\]
By Lemma~\ref{lem:iid-measure}(i),
\[
N\mathbb E_{Q_t}T_N(X)
=
t^4\sum_i(\cR_i-\bar \cR)^2
+2\Lambda''(\Phi(t))^2\sum_i\|A_i-\bar A\|_F^2
+4t^2\Lambda''(\Phi(t))
\sum_{i, j}
\big(\cR_{ij}-\frac{1}{N}\cR_j\big)^2 .
\]
All three terms on the right are nonnegative.

If $\sum_i(\cR_i-\bar \cR)^2=\Omega(N)$, then the first term gives, uniformly over $t\in J_\eta$,
\(
\mathbb E_{Q_t}T_N(X)\ge c_1>0
\)
for all sufficiently large $N$.

On the other hand, suppose that $\mathrm{Tr}(\cR^2)=\Omega(N)$. Then
\[
\sum_{i, j}
\big(\cR_{ij}-\frac{1}{N}\cR_j\big)^2
=
\mathrm{Tr}(\cR^2)-\frac{1}{N}\sum_j \cR_j^2.
\]
Since $|\cR_j|\le \gamma$ by Assumption~\ref{ass:standing}, we have
\(
\frac{1}{N}\sum_j \cR_j^2\le \gamma^2.
\)
Therefore, if $\mathrm{Tr}(\cR^2)\ge cN$, then for all sufficiently large $N$,
\[
\sum_{i, j}
\big(\cR_{ij}-\frac{1}{N}\cR_j\big)^2
\ge \frac{c}{2}N.
\]
Using $|t|\ge a_\eta$ and $\Lambda''(\Phi(t))\ge v_\eta$, the third term gives
\(
\mathbb E_{Q_t}T_N(X)\ge c_2>0
\)
uniformly over $t\in J_\eta$ for all sufficiently large $N$.

Thus, in either case, there exists $c_0>0$ such that
\(
\inf_{t\in J_\eta}\mathbb E_{Q_t}T_N(X)\ge c_0
\)
for all sufficiently large $N$. Since $\varepsilon_N\downarrow 0$, for all sufficiently large $N$,
$\varepsilon_N\le c_0/2$. Hence, by Lemma~\ref{lem:iid-measure}(ii),
\[
Q_t(\cA_N)
=
Q_t(T_N(X)\le \varepsilon_N)
\le
Q_t\left(|T_N(X)-\mathbb E_{Q_t}T_N(X)|\ge \frac{c_0}{2}\right)
\le
2\exp(-cN),
\]
uniformly over $t\in J_\eta$. Therefore,
\(
\limsup_{N\to\infty}\sup_{t\in J_\eta}\frac{1}{N}\log Q_t(\cA_N)<0.
\)
Recalling that $\Lambda'(h) \ne 0$, Lemma~\ref{lem:measure-change}(iii) now implies $\bP(\cA_N)\to 0,$
which is a contradiction. Hence $T_N(X)=\Omega_p(1)$.
Finally, Theorem~\ref{thm:main-1} implies $\max \big(|\hat\beta-\beta|, |\hat h-h| \big)= O_p(N^{-1/2}),$
so the MPLE is $\sqrt N$-consistent.
\end{proof}

\color{black}

\section{Mean-Field Analysis}\label{sec:mean-field-analysis}

In this section we prove the three mean-field results stated in Section~\ref{sec:intro}. 
Theorem~\ref{thm:mean_field} provides the main mean-field approximation: it identifies the variational limit of the free energy and shows that the conditional mean vectors form a low-complexity family. We then use this structure in Theorem~\ref{cor:gap_implies_estimation} to prove that a variational separation from nearly constant profiles yields macroscopic inhomogeneity of the local fields and hence joint \(\sqrt N\)-consistency of the pseudolikelihood estimator. 
The section concludes with Theorem~\ref{thm:est-imposs}, which gives the complementary homogeneous ferromagnetic picture: under regularity and well-connectedness assumptions, the local fields become asymptotically homogeneous and pseudolikelihood estimation becomes ill-conditioned.

\subsection{Proof of Theorem~\ref{thm:mean_field}} 

 Recall that $ f(x) =  \gb  \la A, x^{\otimes 3}\ra/3+ h \la x, \mv1\ra$.
Let $\nabla f(x) = (\partial_1 f(x), \ldots, \partial_N f(x)) $ be the gradient of $f$, where $\partial_i f(x) =  \beta m_i(x) +h$. Set
\[ \cV  = \big \{ \nabla f(x): x \in [\kappa_{-}, \kappa_{+}]^N \big \} \subseteq \mathbb{R}^N.\]
The Gaussian width of a subset $K$ of $\mathbb{R}^N$ is given by 
\[ \mathsf{GW}(K) =  \E \sup_{ x \in K} \langle x, g \rangle, \]
where $g$ is a standard Gaussian vector on $\mathbb{R}^N.$  

A crucial ingredient in our proof is the following mean-field upper bound from \cite{augeri2020nonlinear} for the log-partition function under a general spin distribution, where the error term is controlled by the Gaussian width of the image of $\nabla f$.
\begin{thm}[Corollary~1.2 of \cite{augeri2020nonlinear}, reformulated in our case] \label{thm:augeri}
There is a numerical constant $C>0$ such that 
    \[
    \log Z_N(\gb,h) \le \sup_{t \in [\kappa_{-},\kappa_{+}]^N} \Big ( \gb \la A, t^{\otimes 3}\ra/3 + h \la t, \mv1\ra - I(t)\Big) + C \kappa^{2/3} N^{1/3}  \mathsf{GW}(\cV)^{2/3}.
    \]
\end{thm}

\begin{lem}[Gaussian width bound]\label{lem:GW_bound}
We have
\[ \mathsf{GW}(\cV) \le \sqrt{2} \kappa^2 |\beta| N \sqrt{\log N} \big \| \sum_i A^2_i \big \|^{1/2}_{\mathrm{op}}.  \]
\end{lem}
\begin{proof}
 By definition,
\[
\mathsf{GW}(\cV)=\E \sup_{x\in[\kappa_-,\kappa_+]^N}\langle \nabla f(x),g\rangle =\E \Big [ h\langle \mv 1,g\rangle + \sup_{x\in[\kappa_-,\kappa_+]^N} \beta\sum_i g_i\langle x,A_i x\rangle \Big].
\]
Using $\E\langle \mv 1,g\rangle=0$, we obtain
\begin{align*}
\mathsf{GW}(\cV)
&\le  |\beta| \E \sup_{x\in[\kappa_-,\kappa_+]^N} \Big|    \sum_i g_i\langle x,A_i x\rangle \Big|    \le  |\beta| \E \sup_{ \| x\|_2 \le \kappa \sqrt{N}} \Big|   \big \langle x, \sum_i  g_i A_i x \big \rangle    \Big|\\
&\le   \kappa^2 |\beta| N \E \big \| \sum_i  g_i A_i \big \|_{\mathrm{op}}.
\end{align*}

Finally, by the matrix Gaussian series bound (e.g.\ Theorem~4.1.1 in \cite{tropp2015introduction}),
\[
\E\Big\|\sum_i g_i A_i\Big\|_{\mathrm{op}}
\le \sqrt{2\log N}\,\Big\|\sum_i A_i^2\Big\|_{\mathrm{op}}^{1/2}.
\]
Combining the last two displays yields the desired bound on $\mathsf{GW}(\cV)$.
\end{proof}

For a subset $K$ of $\mathbb{R}^N$ and for any $\delta>0$, let $\mathsf{N}(K, \delta)$ denote the covering number of $K$, that is, the smallest number of $\ell^2$-balls in $\mathbb{R}^N$ with radii $\delta$ needed to cover $K$. Note that $K \subseteq [-1,1]^N$ is of low complexity, in the sense of Definition~\ref{def:low_complex}, if
$\log \mathsf{N}(K,\delta \sqrt{N}) = o(N)$ for each fixed $\delta>0$.

 \begin{lem} \label{lem:GW_to_low_complexity}
 There exists a constant $C>0$ such that for any $\delta>0$ and any $K \subseteq \mathbb{R}^N$, 
    \[ \mathsf{N}(K, \delta)\le \exp \big( C \delta^{-2}\mathsf{GW}(K)^2 \big). \]
 \end{lem}
 \begin{proof}
Sudakov's minoration inequality on $\mathbb{R}^N$  states that there exists a constant $c>0$ such that for any $K \subseteq \mathbb{R}^N$ and for any $\delta>0$, 
\[ \mathsf{GW}(K) \ge c \delta \sqrt{\log \mathsf{N}(K, \delta)},\]
from which the desired upper bound on the covering number of $K$ follows immediately.
 \end{proof}

\begin{proof}[Proof of Theorem~\ref{thm:mean_field}]
(i)  By the mean-field condition~\eqref{cond:GW} and Lemma~\ref{lem:GW_bound}, we obtain $\mathsf{GW}(\cV) = o(N).$   Then Lemma~\ref{lem:GW_to_low_complexity} implies that, for every fixed $\delta>0$,
\[
\log \mathsf{N}(\cV,\delta\sqrt N)
\le
C\delta^{-2}\frac{\mathsf{GW}(\cV)^2}{N}
=o(N).
\]

The upper bound in part (i) now follows directly from Theorem~\ref{thm:augeri}, whereas the lower bound follows from \cite[Theorem~1]{Yan20} and does not require the mean-field assumption.

\noindent (ii) In part (i), we already showed that $\mathsf{GW}(\cV) = o(N)$. Together with Lemma~\ref{lem:GW_to_low_complexity}, this implies that $\log \mathsf{N}(\cV,\delta\sqrt{N}) = o(N)$ for each $\delta>0$. We can express 
\[ b_i(x) = \gL'(\partial_i f(x)), \quad \phi_h(t) = \Lambda'\Big( \tfrac{t + 2h}{3}\Big).\]
By \eqref{eq:lambda-lipsch},   $\gL'$ is $\kappa^2$-Lipschitz. Therefore, we deduce that
\[ \log \mathsf{N}(\mathcal{M}_N, \delta\sqrt{N}) \le \log \mathsf{N}(\cV, \delta \kappa^{-2} \sqrt{N}) = o(N) \] for each fixed $\delta>0$. Hence, the set $\mathcal{M}_N$ has low complexity.
\end{proof}

\subsection{Proof of Theorem~\ref{cor:gap_implies_estimation}}

The proof of Theorem~\ref{cor:gap_implies_estimation} is based on the following proposition.
\begin{prop}\label{thm:conc}
    Suppose the Assumption~\ref{ass:standing} and the mean-field condition~\eqref{cond:GW} hold and  let $X\sim P_{\gb_0, h_0}$. Fix $\delta>0$, and let
    $$S_N:=\Big\{x \in [\kappa_-,\kappa_+]^N: \big(f(x) - I(x) \big) - \sup_{y\in [\kappa_-,\kappa_+]^N}\big( f(y) - I(y)\big)\le -N\delta \Big \}.$$ Then we have $\bP(b({ X})\in S_N)\to 0.$
\end{prop}
\begin{proof}[Proof of Proposition~\ref{thm:conc}]
By \eqref{eq:concen2}, with $a_i=1$, and by \eqref{eq:concen4} of Corollary \ref{lem:concen}, together with the mean-field condition \eqref{cond:GW} and the Lemma~\ref{lem:trace-optr}, we have
\[\E\Big[\sum_i(X_i-b_i(X))\Big]^2=o(N^2),\qquad \E\Big[\sum_{ijk}A_{ijk}(X_iX_jX_k-b_i(X)b_j(X)b_k(X))\Big]^2=o(N^2).\]
This immediately implies that 
$\E \big( f(X)-f(b(X))\big)^2=o(N^2)$. Thus, there exists a sequence of non-negative reals
$(\varepsilon_N)_{N\ge 1}$ converging to $0$, such that
 \begin{align*}
 \mathbb{P}(|f(X)-f(b(X))|\le N\varepsilon_N)\to 1.
 \end{align*}
By  Theorem~\ref{thm:mean_field}(ii), the set $\cM_N:=\{b(x):x\in[\kappa_-,\kappa_+]^N\}$ has low complexity. Taking $E_N=\cM_N$, we have $\bP(b(X)\in E_N)=1$. Therefore the hypothesis~\eqref{eq:non-mean-field-cond} of Lemma~\ref{lem:measure-change} is satisfied.

Now take $C_N=S_N$ in Lemma~\ref{lem:measure-change}(ii). By the definition of $S_N$,
\[
\sup_{y\in S_N}\big( f(y)-I(y) \big)
\le
\sup_{y\in[\kappa_-,\kappa_+]^N}\big( f(y)-I(y)\big)-\delta N,
\]
so the variational gap condition in Lemma~\ref{lem:measure-change}(ii) holds. Hence
\[
\bP\big(|f(X)-f(b(X))|\le N\varepsilon_N,\ b(X)\in S_N\cap E_N\big)\to0.
\]
Since the event $\{|f(X)-f(b(X))|\le N\varepsilon_N,\ b(X)\in E_N\}$
has probability tending to one, it follows that $\bP(b(X)\in S_N)\to 0$.
This proves the proposition.
\end{proof}

We are now ready to prove Theorem~\ref{cor:gap_implies_estimation}.
\begin{proof}[Proof of Theorem~\ref{cor:gap_implies_estimation}]
If $\beta_0=0$, then $f(y)-I(y)=\sum_i \bigl(h_0 y_i-I(y_i)\bigr)$
and it is easy to see that this is maximized at a constant vector $t\mathbf{1}$. Since such a vector satisfies $\sum_i (y_i-\bar y)^2=0$,
this contradicts the assumed variational gap. Therefore, necessarily $\beta_0\neq~0$.

Let  
\[
E_N:=\bigl\{y\in[\kappa_-,\kappa_+]^N:\sum_i (y_i-\bar y)^2\le \eps N\bigr\}.
\]
By the assumed variational gap, for all sufficiently large $N$ we have
\[
E_N \subseteq
S_N:=\Bigl\{y\in[\kappa_-,\kappa_+]^N:
f(y)-I(y)-\sup_{z\in[\kappa_-,\kappa_+]^N}\bigl(f(z)-I(z)\bigr)\le -\delta N
\Bigr\},
\]
where $S_N$ is the set appearing in Proposition~\ref{thm:conc}. Therefore,  Proposition~\ref{thm:conc} yields
\begin{equation}\label{eq:b_not_constant}
\bP\bigl(b(X)\in E_N\bigr) =   \bP \Bigl(\sum_i (b_i(X)-\bar b(X))^2\le \eps N\Bigr)\to 0.  
\end{equation}

From the mean value theorem and Assumption~\ref{ass:standing}, it follows  that  for each $i,j\in[N]$, 
\[
|b_i(X)-b_j(X)|
=
\bigl|\Lambda'(\beta_0 m_i(X)+h_0)-\Lambda'(\beta_0 m_j(X)+h_0)\bigr|
\le |\beta_0| \kappa^2 |m_i(X)-m_j(X)|,
\]
Consequently, \eqref{eq:b_not_constant} implies that 
\[
\bP\bigl(T_N(X)\le \eps |\beta_0|^{-2} \kappa^{-4}\bigr)
\to 0,\]
as claimed.
\end{proof}

\subsection{Proof of Theorem~\ref{thm:est-imposs}}

\begin{lem}\label{lem:spec-help}
Assume Assumption~\ref{ass:standing} and conditions~\eqref{cond:regu}, \eqref{cond:A_nontrivial} and~\eqref{cond:spec-gap} hold. Then given $\eps>0$,  there exists a constant $\eta>0$ such that, for all sufficiently large $N$ and all $z\in[\kappa_-,\kappa_+]^N$ such that 
$\sum_i (z_i - \bar z)^2 \ge \eps N$ where $\bar z := N^{-1}\sum_i z_i$, we have 
\[
\sum_{i,j} \cR_{ij}\,(z_i - z_j)^2  \ge \eta \sum_i (z_i - \bar z)^2.
\]
.
\end{lem}

\begin{proof}[Proof of Lemma~\ref{lem:spec-help} ]
Recall from Section~\ref{sec:intro} that
\[
\cP = D^{-1}\cR, \quad \text{where } D = \mathrm{diag}(\cR(1), \ldots, \cR(N)).
\]
This is the transition matrix of the random walk on the weighted graph with edge weights $\cR_{ij}$ for $i \ne j \in [N]$. A positive spectral gap ensures that the graph is connected, so $\cP$ is irreducible and there exists a unique stationary distribution $\pi$, given by
\[
\pi(i) = \frac{\cR_i}{N\bar{\cR}}, \ \  i  \in [N].
\]
It is easy to see that $\cP$ is reversible with respect to $\pi.$

Since $\cP$ satisfies the spectral gap condition~\eqref{cond:spec-gap}, we have the following Poincar\'e inequality: for any $f:[N] \to \bR$,
\begin{align}\label{eq:poincare}
    \cE(f,f) \ge (1 - \gl_2)\,\mathrm{Var}_{\pi}(f) \ge \gd\,\mathrm{Var}_{\pi}(f),
\end{align}
where the left-hand side is the associated Dirichlet form
\begin{align*}
    \cE(f,f) 
    := \frac12 \sum_{i,j} \pi(i)\cP(i,j)\big(f(i) - f(j)\big)^2
    = \frac{1}{2N\bar{\cR}} \sum_{i,j} \cR_{ij}\big(f(i) - f(j)\big)^2.
\end{align*}
Let $\nu$ be the uniform distribution on $[N]$. Then
\begin{align*}
\|\pi - \nu\|_1
&= \sum_i \bigl|\pi(i) - \nu(i)\bigr| = \sum_i \left|\frac{\cR_i}{N\bar{\cR}} - \frac{1}{N}\right| \\
&\le \frac{1}{N \alpha} \sum_i \bigl|\cR_i - \bar{\cR}\bigr| \le \frac{1}{\alpha} \sqrt{\frac{1}{N}\sum_i \bigl(\cR_i - \bar{\cR}\bigr)^2}.
\end{align*}
The penultimate inequality above follows from the assumption~\eqref{cond:A_nontrivial}, and the last inequality follows from the Cauchy--Schwarz inequality. Therefore, 
by the asymptotic regularity condition~\eqref{cond:regu}, we have 
\(
\norm{\pi - \nu}_1 = o(1).
\)
Since $|z_i| \le \kappa$ for all $i$, we have 
\begin{align*}
    \abs{\E_{\pi} z - \E_{\nu}z}  \le \kappa\norm{\nu- \pi}_1 = o(1), \qquad
    \abs{\E_{\pi} (z^2) - \E_{\nu}(z^2)}  \le \kappa^2 \norm{\nu- \pi}_1 = o(1),
\end{align*}
uniformly in $z \in [\kappa_{-}, \kappa_+]^N$. This implies that 
\(
\abs{\text{Var}_{\pi}(z) - \text{Var}_{\nu}(z)}  = o(1), 
\)
again uniformly in $z \in [\kappa_{-}, \kappa_+]^N$. Combining this with~\eqref{eq:poincare} applied to $f(i)=z_i$ for each $i \in [N]$, we obtain
\begin{align*}
    \frac{1}{2N\bar{\cR}} \sum_{i,j} \cR_{ij}(z_i - z_j)^2 
    \ge \gd\bigl(\text{Var}_{\nu}(z) - o(1) \bigr) = \gd \bigl(\frac{1}{N}\sum_i (z_i - \bar z)^2 - o(1)\bigr) \ge \frac{\gd}{2} \frac{1}{N}\sum_i (z_i - \bar z)^2,
\end{align*}
where the last inequality holds for sufficiently large $N$ since $N^{-1}\sum_i (z_i - \bar z)^2 \ge \eps.$ The lemma now follows from the assumption that $\bar{\cR} \ge \alpha >0.$
\end{proof}

\begin{proof}[Proof of Theorem~\ref{thm:est-imposs}]
Fix any $\eps>0$. Let 
\[ y \in \cC_N:=\{x\in [\kappa_-,\kappa_+]^N: \mathrm{Var}(x):= \frac{1}{N}\sum_i (x_i - \bar{x})^2 \ge \eps \}.\]
Throughout the proof, we will use notation $w=|y|$.  The main step is the following claim.

\noindent \textbf{Claim.} There exists a constant $c_0 > 0$ such that, for all sufficiently large $N$,
\begin{equation}\label{claim:separation_nonconstant}
   \sup_{y\in \cC_N: \mathrm{Var}(w) \ge \eps/2} \big( f(y) - I(y) \big)
    \le \sup_{y \in [\kappa_{-},\kappa_{+}]^N} \big( f(y) - I(y)\big) - c_0 N. 
\end{equation}
By Lemma~\ref{lem:aux-ineq1} and the symmetry of the tensor $A$, we obtain for every $y \in [\kappa_-,\kappa_+]^N$ that
\begin{align}\label{A_holder}
    \sum_{i,j,k} A_{ijk} y_iy_jy_k & \le   \sum_{i,j,k} A_{ijk} w_i^3 - \frac12 \sum_{i,j,k} A_{ijk} \big(w_i^{3/2} - w_{j}^{3/2}\big)^2 \nonumber \\
    & = \sum_{i,j,k} A_{ijk} w_i^3 - \frac12 \sum_{i,j} \cR_{ij} \big(w_i^{3/2} - w_j^{3/2}\big)^2.
\end{align}
 Applying Lemma~\ref{lem:elementary} with $c= 3/2$ to non-negative vector $w$ such that $\mathrm{Var}(w) \ge \eps/2$, we deduce that 
 \begin{align}\label{var_z_lb}
    \sum_i (z_i  - \bar z)^2 \ge (\eps/2)^{3/2} N,  \qquad \text{ where } z_i := w_i^{3/2} .
\end{align}
Set
\[
    a_\kappa:=\frac{\kappa_+-\kappa_-}{\kappa^{3/2}},
    \qquad
    \widetilde z_i:=\kappa_-+a_\kappa z_i .
\]
Since $0\le z_i\le \kappa^{3/2}$, we have
$\widetilde z\in[\kappa_-,\kappa_+]^N$. Moreover, by
\eqref{var_z_lb},
\[
    \sum_i
    (\widetilde z_i-\overline{\widetilde z})^2
    =
    a_\kappa^2\sum_i(z_i-\bar z)^2
    \ge
    a_\kappa^2\left(\frac{\varepsilon}{2}\right)^{3/2}N.
\]
Applying Lemma~\ref{lem:spec-help} to $\widetilde z$ and cancelling
the common factor $a_\kappa^2$, there exists $\eta>0$ such that,
for all sufficiently large $N$,
\[
    \sum_{i, j} \cR_{ij}(z_i-z_j)^2
    \ge
    \eta\sum_i(z_i-\bar z)^2.
\]
Combining this with \eqref{A_holder}, we obtain

\begin{align}\label{A_holder2}
     \sum_{i,j,k} A_{ijk} y_i y_j y_k
     &\le \sum_{i,j,k} A_{ijk} w_i^3
       - \frac{\eta}{2} \sum_i (z_i - \bar z)^2 \le \sum_{i,j,k} A_{ijk} |y_i|^3
       - 2^{-5/2}\eta \eps^{3/2} N,
\end{align}
where the second inequality follows from \eqref{var_z_lb}.

Now fix $y \in \cC_N$ with $\mathrm{Var}(w) \ge \eps/2$. By \eqref{A_holder2}, using that $\beta_0>0$, $h_0\ge 0$, and stochastic non-negativity, we obtain that for sufficiently large $N$
\begin{align*}
    f(y) - I(y) & \le \gb_0/3 \sum_{i,j,k} A_{ijk} \abs{y_i}^3 + h_0 \sum_{i} |y_i| - \sum_{i} I(|y_i|) - c_1N \\
     &= \sum_{i} \big( \gb_0/3 \cR_i  \abs{y_i}^3 + h_0  |y_i| -  I(|y_i|) \big) - c_1N,
\end{align*}
where $c_1 = \frac{\gb_0}{3} 2^{-5/2}\eta \eps^{3/2} >0$. Since $\sum_i (\cR_i - \bar \cR)^2  = o(N)$ and $|y_i| \le \kappa,$ it follows from Cauchy-Schwarz inequality that 
\[ \sup_{y \in [\kappa_-, \kappa_+]^N} \frac13\Big| \sum_{i}  \gb_0 \cR_i  \abs{y_i}^3  - \sum_{i}  \gb_0 \bar \cR  \abs{y_i}^3  \Big|  = o(N).  \]
Therefore, 
     \begin{align*}
     \big( f(y) - I(y) \big) &\le \sup_{y \in \cC_N} \Big\{ \sum_{i} \big( \frac{\gb_0}{3} \bar \cR  \abs{y_i}^3 + h_0  |y_i| -  I(|y_i|) \big)\Big\}  - c_1N + o(N) \\
    &\le N \sup_{t \in [\kappa_-, \kappa_+],  \ t \ge 0} \big( \frac{\gb_0}{3} \bar \cR  t^3 + h_0  t -  I(t) \big) - c_1N + o(N) \\
      &= \sup_{t \in [\kappa_-, \kappa_+],  \ t \ge 0} \big( f(t \cdot \mv1) - I(t\cdot \mv1)\big) - c_1N + o(N) \\
    & \le \sup_{y \in [\kappa_-,\kappa_+]^N} \big( f(y) - I(y)\big) - c_1N + o(N).
\end{align*}
Now taking the supremum of the left-hand side over all $y\in \cC_N$ with $\mathrm{Var}(w)\ge \eps/2$, and setting  $c_0 = c_1/2$, we obtain \eqref{claim:separation_nonconstant} .

Next, under the assumption that either $\kappa_{-} \ge 0$ or $h_0 > 0$, we remove the additional constraint $\mathrm{Var}(w)\ge \eps/2$ from the left-hand side of \eqref{claim:separation_nonconstant}.

Assume that $h_0 >0.$ 
We first claim that we can find $\varrho = \varrho (\kappa, \eps)>0$ such that for any $y \in \cC_N$,  
either
\begin{enumerate}
    \item [(1)]  $|\{ i \in [N] : y_i < -\varrho \}| \ge \varrho N $, or 
\item [(2)] $\mathrm{Var}(w) \ge \eps/2.$
\end{enumerate}
Indeed if (1) fails, then $|\{i:y_i<-\varrho\}|<\varrho N$. Writing $d:=w-y$, we have
\[
\sqrt{\mathrm{Var}(w)}
\ge \sqrt{\mathrm{Var}(y)}-\sqrt{\mathrm{Var}(d)}
\ge \sqrt{\eps}-\Big(\frac1N\sum_i d_i^2\Big)^{1/2}.
\]
Also, $d_i=|y_i|-y_i = 2|y_i| 1_{ \{ y_i < 0 \} }$. So, $d_i^2\le 4\varrho^2$ on $\{-\varrho<y_i<0\}$ and $d_i^2\le 4\kappa^2$ on $\{y_i<-\varrho\}$. Hence
\[
\frac1N\sum_i d_i^2
\le 4\varrho^2+4\kappa^2\frac{|\{i:y_i<-\varrho\}|}{N}
\le 4\varrho^2+4\kappa^2\varrho,
\]
and therefore
\[
\sqrt{\mathrm{Var}(w)}
\ge
\sqrt{\eps}-2\sqrt{\varrho^2+\kappa^2\varrho}.
\]
Choosing $\varrho=\varrho(\kappa,\eps)>0$ small enough so that $2\sqrt{\varrho^2+\kappa^2\varrho}\le \sqrt{\eps}-\sqrt{\eps/2},$
we get $\mathrm{Var}(w)\ge \eps/2$.

Now let $y \in \cC_N$ such that $\mathrm{Var}(w) < \eps/2$. Then there exists $\varrho>0$ such that $|\{ i \in [N] : y_i < -\varrho \}| \ge \varrho N.$ Since $A_{ijk}\ge 0$, $\beta_0>0$, $h_0>0$, and by stochastic non-negativity, we have
\begin{align*}
    f(y)  - I(y) \le f(|y|)  - I(|y|)  - h_0 \sum_{i: y_i < 0} (|y_i| - y_i) \le f(|y|)  - I(|y|)  - 2h_0 \varrho^2 N.
\end{align*}
The above inequality, combined with \eqref{claim:separation_nonconstant}, yields
\begin{equation}\label{claim:separation_nonconstant2}
   \sup_{y\in \cC_N} \big( f(y) - I(y) \big)
    \le \sup_{y \in [\kappa_{-},\kappa_{+}]^N} \big( f(y) - I(y)\big) - \vartheta N, 
\end{equation}
for some constant $\vartheta>0$.  If $\kappa_{-} \ge 0$, then $w=y$, so for every $y\in \cC_N$, we have $\mathrm{Var}(w) \ge \eps \ge \eps/2$. Hence the constraint in \eqref{claim:separation_nonconstant} is automatically satisfied, and \eqref{claim:separation_nonconstant2} follows.

Consequently, $\cC_N \subseteq S_N$, where $S_N$ is as defined in Proposition~\ref{thm:conc}. Under the mean-field condition~\eqref{cond:GW}, Proposition~\ref{thm:conc} yields
\begin{equation}\label{eq:b_conv_prob_0}
 \bP\bigl(b(X) \in \cC_N\bigr)
= \bP\Bigl(\sum_i \bigl(b_i(X) - \bar b(X)\bigr)^2 \ge \eps N\Bigr)
\to 0 \quad \text{as } N\to\infty,   
\end{equation}
which holds for any fixed $\eps>0$.
Recall that $|m_i(x)|\le \gamma \kappa^2$. Since $\Lambda''$ is continuous and strictly positive, set
\( c_* = \inf_{|t| \le \gamma \kappa^2 } \Lambda''(\beta_0 t + h_0) >0.  
\)
Then by the mean value theorem, for any $x \in [\kappa_{-}, \kappa_+]^N$,
\[ |b_i(x) - b_j(x)|  = \big|\Lambda'(\beta_0 m_i(x)+ h_0) - \Lambda'(\beta_0 m_j(x)+ h_0) \big|  \ge \beta_0 c_* | m_i(x)  -   m_j(x)|.  \]
Hence, $\sum_i \bigl(m_i(X) - \bar m(X)\bigr)^2 \le (\beta_0 c_*)^{-2} \sum_i \bigl(b_i(X) - \bar b(X)\bigr)^2 $.  Therefore, by \eqref{eq:b_conv_prob_0}, $T_N(X) = o_p(1),$ as desired.

\section{Acknowledgements}
SM gratefully acknowledges NSF for partial support during this research (DMS-2515519). The research of AS is partly supported by Simons Foundation MP-TSM-00002716.
\end{proof}


\bibliography{aap-submission/est}

\appendix

\section{Proof of Theorem~\ref{thm:main-1}}\label{sec:proof of main-1}

In this section, we prove Theorem~\ref{thm:main-1}. The pseudo-likelihood function is strongly concave whenever $T_N(X)>0$, and by hypothesis this occurs with probability tending to one. Hence, to establish the existence of the MPLE it suffices to show that, with probability tending to one,
\begin{equation}\label{pseudo_diverges}
   \lim_{\abs{\gb}+\abs{h} \to \infty} L(\gb, h|X) = -\infty.  
\end{equation}
 We first state the following sufficient condition for~\eqref{pseudo_diverges}, which is of similar nature as in~\cite[Lemma 2.3]{CSW24}.
 
\begin{lem}\label{lem:suff_cond}
Let $x \in [\kappa_-, \kappa_+]^N$ such that $I(x_i) < \infty$ for each $i$ and assume that there exist distinct $1 \le i, j, k, \ell \le N$ and $a \in \bR$ such that
    \begin{align} \label{suff_cond}
    \begin{split}
        & x_i \neq \kappa_-, m_i(x) >a, \,\, x_j \neq \kappa_+,  m_j(x) > a, \\
        & x_k \neq \kappa_-,  m_k(x) < a, \,\,x_l \neq \kappa_+,  m_l(x) < a.
        \end{split}
    \end{align}
	Then $\lim_{ |\beta| + |h| \to \infty}  L(\beta, h| x)  =  - \infty.$    
\end{lem}
\begin{proof}

Recall the log pseudo-likelihood function, 
\[
L(\gb, h|x) = \sum_i x_i(\gb m_i(x) + h) - \gL(\gb m_i(x) +h)=\sum_i T(\gb m_i(x)+h, x_i),
\]
where $\gL(\lambda) = \log \int_{\kappa_-}^{\kappa_+} \exp(\gl z) d\mu(z)$ for the reference measure $\mu$, and $T(\gl,x) := x \gl - \gL(\gl)$ for $\lambda\in \bR, x\in [\kappa_-, \kappa_+]$. Then we have
\(
   T(\lambda,x)\le \sup_{\lambda\in \bR} T(\lambda,x)=I(x),
\)
where the equality part is due to the fact in Definition~\ref{def:main-set-up}. Thus,  if $I(x)<\infty$, then the map $\lambda\mapsto T_i(\lambda,x)$ is bounded above.

 We now claim  that
\begin{align}\label{eq:defer}
    \lim_{\lambda\to\infty}T(\lambda,x)=-\infty\text{ if }x< \kappa_+,\quad \lim_{\lambda\to-\infty}T(\lambda,x)=-\infty\text{ if }x> \kappa_-.
\end{align}
 To see the claim, note that $\kappa_+ \in \text{supp}(\mu)$, thus $\mu([\kappa_+ - \gd, \kappa_+])>0$ for every $\gd>0$. For $\gl\ge 0$, we have 
\begin{align*}
\gL(\gl) & =\log \Big( \int_{\kappa_-}^{\kappa_+} e^{\gl z} d\mu(z)\Big)
 \ge   \log \Big( \int_{\kappa_+-\gd}^{\kappa_+} e^{\gl z} d\mu(z)\Big) \\
& \ge  \log \Big( e^{\gl (\kappa_+ - \gd)} \cdot \mu([\kappa_+ - \gd, \kappa_+]) \Big)
 = \gl(\kappa_+-\gd) + \log(\mu([\kappa_+ - \gd, \kappa_+])).
\end{align*}
For $x \in (\kappa_-, \kappa_+)$, 
taking $\gd = \frac{\kappa_+- x}{2}>0$, we have 
\begin{align*}
x\gl  - \gL(\gl) \le \frac{(x - \kappa_+)\gl}{2} - \log(\mu([\kappa_+ - \gd, \kappa_+])).
\end{align*}
As $\gl \to +\infty$, it is clear that $x\gl - \gL(\gl)\to - \infty$ since $x-\kappa_+ < 0$. The argument for the case $x > \kappa_{-}$ is similar.

Since $I(x_r)<\infty$ for all $r\in [N]$ by assumption, it suffices to show that \[\lim_{\abs{\gb} + \abs{h}\to \infty} T(\gb m_s(x) +h,x_s) = -\infty\] for some  $s\in \{i,j,k,\ell\}$ where the indices $i,j,k,\ell$ are as in \eqref{suff_cond}.
We argue by contradiction. Suppose for $i,j,k,\ell$ satisfying~\eqref{suff_cond}, we have 
\(
T(\gb m_s(x) +h,x_s) \ge -K \ \text{for $s \in \{i,j,k,\ell\}$},
\)
for a common subsequence $(\beta,h)$ satisfying $|\beta|+|h|\to \infty$.
By slight abuse of notation, we are going to restrict to this common subsequence. By~\eqref{eq:defer}, since $x_j, x_\ell \neq \kappa_+$, and $x_i, x_k\ne \kappa-$ then there exists a constant $K_1$ (depending only on $K, \mu, x_i, x_j, x_k, x_\ell$) such that 
\begin{align*}
 \gb m_j(x) +h \le K_1 \quad &\text{and} \quad  \gb m_\ell(x) +h \le K_1,\\
 \gb m_i(x) +h \ge -K_1 \quad& \text{and} \quad  \gb m_k(x) +h \ge -K_1.
\end{align*}
This confines $h$ to the interval:
\begin{equation}\label{interval_h}
 \max(-\beta m_i(x) - K_1, -\beta m_k(x) - K_1) \le h \le  \min(-\beta m_j(x) + K_1, -\beta m_l(x) + K_1).   
\end{equation}
For this interval to be non-empty, its upper bound must be greater than its lower bound.

 \textbf{Case $\beta \ge 0$}: The conditions $m_j(x) > a > m_k(x)$ imply $m_j(x) > m_k(x)$. The non-empty interval constraint implies $-\beta m_k(x) - K_1 < -\beta m_j(x) + K_1$, which simplifies to $\beta(m_j(x) - m_k(x)) < 2K_1$. Therefore, we have 
    \begin{equation}\label{positive_beta}
     \beta < 2K_1/(m_j(x) - m_k(x)).
    \end{equation}
    
 \textbf{Case $\beta < 0$}:  The conditions $m_i(x) > a > m_l(x)$ imply $m_i(x) > m_l(x)$. The constraint implies $-\beta m_i(x) - K_1 < -\beta m_l(x) + \tilde{K}$, which simplifies to $-\beta(m_i(x) - m_l(x)) < 2K_1$. 
        Therefore, we have 
            \begin{equation}\label{negative_beta}
 - \beta < 2K_1/(m_i(x) - m_l(x)).
 \end{equation}
Combining \eqref{positive_beta} and \eqref{negative_beta}, we obtain
\[ |\beta| < \max \Big(\frac{2K_1}{m_j(x) - m_k(x)}, \frac{2K_1}{m_i(x) - m_l(x)}  \Big),  \]
which, in combination with \eqref{interval_h} forces $|h|$ to be bounded. So, $|\beta| + |h|$ can not be arbitrarily large. This concludes the proof of the lemma.
\end{proof}

In proving existence, the main difficulty is verifying condition~\eqref{suff_cond} under the assumptions of Theorem~\ref{thm:main-1}. To that end, we introduce the following lemma, which quantifies the idea that if a bounded sequence exhibits nontrivial variability, then a linear fraction of its entries are separated by a fixed gap.

\begin{lem}\label{exist-lem1}
    Let $\gd>0$ and  $\gamma>0$ and define
    \begin{align}\label{def:parameters}
    \eps =  \frac{\delta}{64 \gamma^2\kappa^4},  \ \   \zeta = \min \Big( \frac{\sqrt{\delta}}{4}, \frac{\gamma\kappa^2}{4} \Big),  \,\, \text{and} \,\,  \   K =\frac{\gamma\kappa^2}{ \zeta}
    \end{align}
    For any real sequence $m_1, m_2, \ldots, m_N$ with $\abs{m_i}\le \gamma\kappa^2$ satisfying $N^{-2} \sum_{i, j} (m_i - m_j)^2 \ge \gd$,
    there exists an integer $r$ with $|r| \le K$ and $
    -\gamma \kappa^2 \le (r-1) \zeta < r \zeta \le \gamma \kappa^2$ such that 
    \begin{align*}
    |\{i \in [N]: m_i \in  [-\gamma\kappa^2,  (r-1)\zeta]  \}|  \ge \eps N \text{ and } |\{i \in [N]: m_i \in  [  r \zeta, \gamma \kappa^2 ]  \}|  \ge \eps N.
    	\end{align*}
\end{lem}

\begin{proof}
Let  $\abs{m_i} \le \gamma \kappa^2$ for all $i\in [N]$ and
    \begin{equation} \label{eq:restricted_var}
        \sum_{i,j=1 }^N (m_i - m_j)^2 \ge  \delta N^2.
    \end{equation}
    Define $  r = \max \left\{ s \in \mathbb{Z} : \left| \left\{ i \in [N] : m_i \in [s \zeta, \gamma \kappa^2] \right\} \right| \ge \eps N \right\}.$
This set is nonempty because for integer $s\le -\gamma\kappa^2/\zeta$ we have $[s \zeta,\gamma\kappa^2]\supseteq[-\gamma\kappa^2,\gamma\kappa^2]$, hence the count is $N\ge\varepsilon N$. By definition, $r \zeta\le\gamma\kappa^2$, so $r\le \gamma\kappa^2/\zeta$.

    Suppose, for contradiction, that either $(r-1)\zeta < -\gamma\kappa^2$ or 
    \[ -\gamma\kappa^2 \le (r-1)\zeta  \ \ \text{ and } \ \ 
    \left| \left\{ i \in [N] : m_i \in [-\gamma\kappa^2, (r - 1)\zeta] \right\} \right| < \eps N.
    \]
     In either of the cases we have that at most $2\eps N$ of the $m_i$ can lie in the interval $[-\gamma\kappa^2, \gamma\kappa^2] \setminus [(r-1)\zeta, (r+1)\zeta]$, because by maximality of $r$ we have
    \( |\{ i: m_i \in [(r+1)\zeta, \gamma\kappa^2]\}|< \eps N.\)
    We now derive an upper bound on the left-hand side of \eqref{eq:restricted_var} by considering the following cases.

 \medskip
    \noindent
  $(i)$ The total contribution from pairs $i, j$ such that at least one of $m_i$ or $m_j$ lies outside $[(r-1)\zeta, (r+1)\zeta]$ is at most
    \(
    2 \cdot N \cdot 2\eps N \cdot 4\gamma^2\kappa^4 \le \frac{\delta}{4} N^2.
    \)

    \medskip
    \noindent
    $(ii)$ The total contribution from pairs $i, j$ such that both $m_i$ and $m_j$ lie within $[(r-1)\zeta, (r+1)\zeta]$ is at most
    \(
    N^2 \cdot (2\zeta)^2 \le \frac{\delta}{4} N^2.
    \)

    \medskip
    \noindent
    By the choice of $\eps$ and $\delta$, the sum of these contributions is at most $\delta N^2/2$, which contradicts \eqref{eq:restricted_var}. Therefore,
    \(
    \left| \left\{ i \in [N] : m_i \le (r - 1)\zeta \right\} \right| \ge \eps N
    \)
    and the lemma follows.

\end{proof}
\begin{proof}[Proof of Theorem~\ref{thm:main-1}: Existence of MPLE]
By hypothesis, there exists a slowly growing sequence $c_N \to \infty$ such that 
\[ P_{\gb_0,h_0} \big (N^{-2} \sum_{i, j} (m_i(X) - m_j(X))^2 \ge c_N N^{-2/7} \big) \to 1. \]
Set $\delta = \delta_N = c_N N^{-2/7}$ and define $\eps = \eps_N, \zeta = \zeta_N$ and $K = K_N$ according to \eqref{def:parameters} based on $\delta$ and $\gamma.$
For $-\infty< \ell<u< \infty,$ let
	\[ I_{\ell, u} :=  \left\{ i \in [N] : m_i(X) \in [\ell, u] \right\}.\]
By Lemma~\ref{exist-lem1}, with probability tending to one, 
\begin{equation}\label{eq:r_exists}
   \exists \text{ integer } r \in [-K, K]  \text{ such that } |I_{r \zeta, \gamma\kappa^2}| \ge  \eps N \text{ and  } |I_{-\gamma\kappa^2, (r-1) \zeta}| \ge  \eps N.
\end{equation}

We now claim that  the following events have negligible probability as $N \to \infty$:  there exists an integer  $s \in [-K, K]$ such that 
\begin{align}
	\label{exist-proof:eq2}
	|I_{s\zeta , \gamma\kappa^2}| \ge \eps N\,\, & \mbox{and}\,\,\{ X_i=\kappa_+ \text{ for all } i \in I_{(s  - 1/2) \zeta, \gamma\kappa^2} \},\\
    	\label{exist-proof:eq22}
	|I_{s\zeta , \gamma\kappa^2}| \ge \eps N\,\, & \mbox{and}\,\,\{ X_i=\kappa_- \text{ for all } i \in I_{(s  - 1/2) \zeta, \gamma\kappa^2} \},\\
    \label{exist-proof:eq1}
|I_{-\gamma\kappa^2, (s-1)\zeta}| \ge \eps N\,\,& \text{and}\,\,\{ X_i=\kappa_+  \text{ for all } i \in I_{-\gamma\kappa^2, (s-1/2)\zeta} \} ,\\
 \label{exist-proof:eq12}
|I_{-\gamma\kappa^2, (s-1)\zeta}| \ge \eps N\,\, & \text{and}\,\,\{ X_i=\kappa_-\text{ for all } i \in I_{-\gamma \kappa^2, (s-1/2)\zeta} \}.
\end{align}
Once the claim is established, we can combine with \eqref{eq:r_exists} to deduce that 
with probability tending to one, there exists an integer $r \in [-K, K]$ such that 
\begin{eqnarray*}
   |I_{(r  - 1/2) \zeta, \gamma\kappa^2}| \ge 	|I_{r\zeta , \gamma\kappa^2}| \ge \eps N \ge 2, &\text{ and}& \  \{ X_i : i \in I_{(r  - 1/2) \zeta, \gamma\kappa^2} \}\neq   \{\kappa_-\}, \{\kappa_+\} , \\
|I_{-\gamma\kappa^2, (r  - 1/2) \zeta}|  \ge 	|I_{-\gamma\kappa^2, (r-1)\zeta }| \ge \eps N \ge 2, &\text{ and }&\   \{ X_i : i \in I_{-\gamma\kappa^2, (r  - 1/2) \zeta}  \}\neq \{\kappa_-\}, \{\kappa_+\}. 
\end{eqnarray*}

With this, we have that, with probability tending to one as $N\to\infty$, condition~\eqref{suff_cond} of Lemma~\ref{lem:suff_cond} holds with $a =  (r  - 1/2) \zeta$. Besides, Lemma~\ref{lem:rate_as_finite} guarantees $I(x)<\infty$ $\mu$-a.s, thereby establishing the existence of the MPLE.

It remains to prove the claim. We will focus on the proof for the case~\eqref{exist-proof:eq2} since the other cases $\eqref{exist-proof:eq22},\eqref{exist-proof:eq1},\eqref{exist-proof:eq12}$ can be handled in a similar fashion. Note that $|m_i(x)|\le \gamma\kappa^2$, and so $$\kappa_+ - b_i(x) =\kappa_+-\Lambda'(\beta_0m_i(x)+h_0)\ge \kappa_+ - \Lambda'(|\beta_0| \gamma\kappa^2+ h_0) =: c_0 >0.$$

For $s \in \mathbb{Z}$, define the events
\begin{align*}
& \mathcal{B}^s_+  = \big \{ |I_{s\zeta , \gamma\kappa^2}| \ge \eps N, \  X_i = \kappa_+  \ \forall i \in I_{(s  - 1/2) \zeta, \gamma \kappa^2}  \big \}, \\
& \mathcal{B}^s_-  = \big \{ |I_{s\zeta , \gamma\kappa^2}| \ge \eps N, \  X_i = \kappa_-  \ \forall i \in I_{(s  - 1/2) \zeta, \gamma\kappa^2}  \big \}.
\end{align*}
On the event $\mathcal{B}^s_+$, we have 
\begin{align*}
    S^s_+(X) & := \sum_i (X_i -b_i(X))(m_i(X) - (s-1/2)\zeta)_+ \\
    & =  \sum_{i \in I_{(s-1/2)\zeta,\gamma\kappa^2}}(\kappa_+ -b_i(X))(m_i(X)  - (s-1/2)\zeta)_+ \\
    & \ge c_0\sum_{i \in I_{s \zeta,\gamma\kappa^2}}(m_i(X)  - (s-1/2)\zeta)_+ \ge \frac{c_0 \zeta \eps N}{2}.
\end{align*}
By the concentration result~\eqref{eq:concen1} in Corollary~\ref{lem:concen} applied to the $1$-Lipschitz function $\phi(t)  = (x - (s - 1/2) \zeta)_+$, we have 
\[
\bP(\mathcal{B}^s_+) \le \bP\Big(|S^s_+(X)|\ge\frac{c_0 \zeta \eps N}{2} \Big) \le \frac{c_1}{\zeta^2\eps^2 N},
\]
where the constant $c_1$ only depends on $\gb_0,h_0,\gamma, \mu$. Similarly, on $\mathcal{B}^s_-,$ we have 
\[  S^s_-(X)  := \sum_i (X_i -b_i(X))(m_i(X) - (s-1/2)\zeta)_+  \le  -  \frac{c_0 \zeta\eps N}{2},\]
which, together with the concentration result, yields  
\(
\bP(\mathcal{B}^s_-)  \le \frac{c_1}{\zeta^2\eps^2 N},
\)
By a union bound, the probability of the event 
\( \bigcup_{s: |s| \le K} \big( \mathcal{B}^s_+\cup \mathcal{B}^s_-\big) \)
is bounded above by 
\begin{equation}\label{eq:prob_bound1}
   2( 2K+1) \frac{c_1}{\zeta^2\eps^2 N}  \le \frac{c_2 K}{\zeta^2\eps^2 N},   
\end{equation}
where the constant $c_2$ again only depends on $\gb_0,h_0,\gamma,\mu$. As $N \to \infty,$
\[ \eps \asymp \delta,  \  \zeta \asymp \delta^{1/2}, \  K \asymp \delta^{-1/2}, \]
the hidden constants in $\asymp$ depends only on $\gamma.$ Therefore, the RHS of \eqref{eq:prob_bound1} can be bounded by 
\[ O_{\beta_0, h_0, \gamma,\mu} ( \delta^{-7/2} N^{-1})  = O_{\beta_0, h_0, \gamma,\mu} ( c_N^{-7/2}) \to 0.\]
This proves the claim and, therefore, concludes the proof of the existence of the MPLE.
\end{proof}

\begin{proof}[Proof of Theorem~\ref{thm:main-1}: Consistency part]

    Let $X \sim P_{\beta_0, h_0}$. For $(\beta, h ) \in \mathbb{R}^2$, let $\gl_1(\gb,h |X)\geq 0$ be the minimum eigenvalue of the negative Hessian matrix $H(\gb,h|X)$. 
Using \eqref{eq:hess-det} gives
\begin{align*}
\gl_1(\gb, h |X) & \ge \frac{\abs{H(\gb ,h |X)}}{\mathrm{Tr}(H(\gb,h|X))} =  \frac{\frac12 \sum_{i, j} \theta_{i}(\beta, h| X) \theta_{j}(\beta, h| X)  (m_i(X)-m_j(X))^2}{\mathrm{Tr}(H(\gb,h|X))},
\end{align*}
where $\theta_{i}(\beta, h| X) =\gL''(\beta m_i(X) + h)$.

To bound the right hand side above from below, if $\|\beta-\beta_0,h-h_0\|\le r$ for some fixed $r>0$, then we have
\begin{align*}
 &\sum_{i, j} \theta_{i}(\beta, h| X) \theta_{j}(\beta, h| X)  (m_i(X)-m_j(X))^2 
 \ge \theta_{\min}^2  \sum_{i, j \in [N]} (m_i(X)-m_j(X))^2,
\end{align*}
where $\theta_{\min} := \min_{|\gl|\le \gamma \kappa^2, \|\beta-\beta_0,h-h_0\|\le r} \gL''(\gl \beta +h)>0$  (see Definition~\ref{def:main-set-up}). Also we have $\gL''(\gl) \le \kappa^2$ and $\sum_im_i^2(X)\leq \gc^2\kappa^4N$, and so
\[
\mathrm{Tr}(H(\gb,h |X)) = \sum_i\theta_i(\gb ,h|X) (m_i^2(X)+1)\le N(1+\gamma^2\kappa^4) \cdot \kappa^2.
\]
Therefore, for  $\|(\beta-\beta_0,h-h_0)\|\le r$ we obtain 
\begin{align*}
    \gl_1(\gb, h |X) \ge \frac{\theta_{\min}^2}{2\kappa^2(1+\gamma^2\kappa^4) N}   \sum_{i, j =1}^N (m_i(X)-m_j(X))^2.
\end{align*}
Thus, there exists $\eta>0$ such that the following uniform lower bound for $\gl_1(\gb, h |X)$ holds:
\begin{align}\label{eq:mineig_uni_bd}
\inf_{\norm{\gb-\gb_0,h-h_0}\le r} \gl_1(\gb, h|X) \ge \frac{\eta }{N} \sum_{i, j =1}^N (m_i(X)-m_j(X))^2.
\end{align}
Next, we construct an interpolation between the pseudo-likelihood estimator and the ground truth by letting
\begin{align*}
    \gb(t)= t \hat{\gb} + (1-t)   \gb_0\,\,\mbox{and}\,\, h(t)= t  \hat{h} + (1-t)   h_0
\end{align*}
for $t\in [0,1].$
Set $F:[0,1]\to \bR$ as
\begin{align*}
 F(t)= (\hat{\gb} - \gb_0)  S(\gb(t), h(t)|X) + (\hat{h} - h_0) Q(\gb(t), h(t)|X) .
\end{align*}
Since $S(\hat{\gb},\hat{h}|X)=0$ and $Q(\hat{\gb},\hat{h}|X)=0$, it is clear
\begin{align*}
    \abs{F(1)- F(0)} = \bigl|(\hat{\gb} - \gb_0) S(\gb_0, h_0|X) + (\hat{h} - h_0) Q(\gb_0, h_0|X) \bigr|.
\end{align*}
By Corollary~\ref{lem:concen}, we have 
\(
S(\gb_0,h_0|X), Q(\gb_0,h_0|X) = O_p(\sqrt{N}).
\)
Applying the Cauchy-Schwarz inequality, we deduce 
\begin{align}\label{eq:asy-YN}
\abs{F(1) - F(0)} = O_p(\sqrt{N} Y_N),
\end{align}
where $Y_N:=\|(\hat{\gb}-\gb_0, \hat{h}-h_0 )\|_{2}$. 
Now to bound $Y_N$, we need to control $\abs{F(1) - F(0)}$.
For any $t\in[0,1]$, the derivative of $f$ is bounded from below by
\begin{align*}
    F'(t) &= -(\hat{\gb}-\gb_0, \hat{h}-h_0)  H(\gb(t),h(t)|X)   (\hat{\gb}-\gb_0, \hat{h}-h_0)^T \\
    &\le -\gl_1(\gb(t),h(t)|X) Y_N^2.
\end{align*}
From the fact $\| (\beta(t) - \beta_0,  h(t) - h_0)\|_2  = t Y_N $ and 
the uniform lower bound \eqref{eq:mineig_uni_bd}, we obtain 
\begin{align}
    F(0)-F(1) &= -\int_0^1 F'(t) dt \ge -\int_{0}^{\min ( 1, r/{Y_N}) } F'(t) dt \nonumber \\
    &\ge \min \big( 1, \frac{r}{Y_N}\big)    \frac{\eta Y_N^2}{N} \sum_{i, j } (m_i(X)-m_j(X))^2\nonumber\\
    & =2\min \big( 1, \frac{r}{Y_N}\big)    N\eta Y_N^2T_N(X)\nonumber
\end{align}
On the other hand, \eqref{eq:asy-YN} implies that
\( \bP( | F(1)  -  F(0)| \le c_N \sqrt{N} Y_N)\to 1,\)
for any sequence of non-negative reals $(c_N)_{N\ge1}$ diverging to $\infty$.
Combining the above two displays, with $P_{\beta_0, h_0}$-probability tending to $1$ we have
\begin{align*}
c_N\sqrt{N} Y_N \ge 2\min \big(1 , \frac{r}{Y_N}\big)    N\eta Y_N^2 T_N(X)\Rightarrow \min(Y_N, r) \le \frac{c_N}{2\eta \sqrt{N} T_N(X)}.
\end{align*}
Since $T_N=\omega_p(N^{-2/7})$, choosing $c_N= N^{1/5}$ we get $Y_N=\min(Y_N,r)$ with $P_{\beta_0,h_0}$ probability tending to $1$ for all $r>0$. Thus we conclude $Y_N=O_P\Big(\frac{c_N}{2\eta \sqrt{N} T_N(X)}\Big)$ for any  diverging sequence $(c_N)_{N\ge 1}$, thus giving us
$Y_N  = O_p(\frac{1}{\sqrt{N}T_N})$ as claimed.

\end{proof}

\section{Proofs for Applications}\label{sec:proof of example}
In this section we present the proofs of the main results in Section~\ref{sec:examples}. We repeatedly use
Lemmas~\ref{lem:trace-optr}{\rm (b)} and~\ref{lem:trace-optr-converse} in the following form. To verify the mean-field condition~\eqref{cond:GW},
it suffices to check Assumption~\ref{ass:standing}, the nontriviality condition~\eqref{cond:A_nontrivial},
the strong pseudo-regularity condition~\eqref{cond:strong-pseu-regu}, and
\(\operatorname{Tr}(\cR^2)=o(N/\log N)\), and then apply Lemma~\ref{lem:trace-optr}{\rm (b)}. When Theorem~\ref{thm:est-imposs} is applied, we also need
the spectral-gap condition~\eqref{cond:spec-gap}; by Lemma~\ref{lem:trace-optr-converse}, under strong
pseudo-regularity it is enough to check a uniform spectral gap for the unweighted
support graph \(B_{ij}:=\mathbf 1_{\{\cR_{ij}>0\}}\). In particular, if all off-diagonal
entries of \(\cR\) are positive and comparable, then the support graph is complete and
the second assertion of Lemma~\ref{lem:trace-optr-converse} gives~\eqref{cond:spec-gap}.

\subsection{Proofs for ERGMs} In this subsection, we prove the Theorem~\ref{thm:app-triangle} for edge-triangle ERGM and Theorem~\ref{thm:app-3-star} for edge-three-star ERGM. Note that for Theorem~\ref{thm:app-triangle}, it suffices to validate the conditions of the main results in Section~\ref{sec:intro}. For Theorem~\ref{thm:app-3-star}, we will directly prove the general ill-conditioned results. 

\subsubsection{Proofs for edge-triangle ERGM}
The proof of Theorem~\ref{thm:app-triangle} for the edge-triangle ERGM follows from the following two lemmas, which basically validate the conditions in the general results of Section~\ref{sec:intro}.

\begin{lem}\label{lem:triangle_mean}
    Suppose $A$ is the tensor defined in \eqref{eq:ergm-triangle}. Then 
    \begin{enumerate}
        \item[(a)]  $A$ satisfies Assumption~\ref{ass:standing} and is nontrivial in the sense of \eqref{cond:A_nontrivial}. Moreover, $\mathcal{R}_e$ does not depend on edge $e$ (and hence  \eqref{cond:regu} holds).
         \item[(b)] The strong pseudo-regularity condition
\eqref{cond:strong-pseu-regu} holds with $K=1$, and $\mathrm{Tr}(\cR^2) = O(n)$. Consequently, the mean-field condition \eqref{cond:GW} follows.
        \item[(c)] $A$ satisfies the spectral-gap
condition~\eqref{cond:spec-gap}. 
    \end{enumerate}
  
\end{lem}

\begin{proof}
\noindent (a) For the triangle ERGM tensor, for each edge $e_0$ we have
\[ \sum_{e_1,e_2} A_{e_0 e_1 e_2}
=
\frac{1}{n}\,\#\{(e_1,e_2): (e_0,e_1,e_2)\text{ forms a triangle}\}
=
\frac{2(n-2)}{n}.
\]
Indeed, if $e_0=(1,2)$, then for each $i\in[3,n]$ the other two edges of the triangle are
$(1,i)$ and $(2,i)$, and they can appear in either order as $(e_1,e_2)$, giving exactly
$2(n-2)$ ordered pairs $(e_1,e_2)$. Thus $\cR_{e_0}=\sum_{e_1,e_2}A_{e_0 e_1 e_2}$
is independent of $e_0$, and hence \eqref{cond:regu} holds. Besides, Assumption~\ref{ass:standing} holds with $\gamma = 2$. Moreover, 
\(
\bar{\cR}=\frac{1}{N}\sum_{e_0} \cR_{e_0}=\frac{2(n-2)}{n} \ge 1,
\)
for all $n \ge 4$, verifying \eqref{cond:A_nontrivial}.\\

\noindent (b) 
For two edges $e_0,e_1$, we have
\(
    \cR_{e_0e_1}
    :=
    \sum_{e_2} A_{e_0e_1e_2}
    =
    \frac{1}{n}\mathbf{1}_{\{e_0\sim e_1\}},
\)
where $e_0\sim e_1$ means that $e_0$ and $e_1$ share exactly
one vertex. Indeed, if $e_0\sim e_1$, there is a unique edge
$e_2$ such that $(e_0,e_1,e_2)$ forms a triangle.

Thus every positive entry of $\cR$ equals $1/n$, and therefore the
strong pseudo-regularity condition
\eqref{cond:strong-pseu-regu} holds with $K=1$. Moreover, every
edge of $K_n$ is adjacent to exactly $2(n-2)$ other edges. Hence
\[
\begin{aligned}
    \mathrm{Tr}(\cR^2) =
    \sum_{e_0,e_1} \cR_{e_0e_1}^2 =
    N\cdot 2(n-2)\cdot \frac{1}{n^2} =
    \frac{(n-1)(n-2)}{n} = O(n).
\end{aligned}
\]
Since $N=\binom{n}{2}$, it follows that $  \mathrm{Tr}(\cR^2)=
    o\left(N /\log N\right).$ Part~(a) verifies Assumption~\ref{ass:standing} and the
nontriviality condition \eqref{cond:A_nontrivial}. Therefore
Lemma~\ref{lem:trace-optr}{\rm (b)} gives the mean-field
condition \eqref{cond:GW}.
\\

\noindent (c) 
By part (b), $R_{e_0e_1}>0$ if and only if \(e_0\) and \(e_1\) share one vertex.
Thus the support graph $B_{e_0e_1}:=\mathbf 1_{\{\cR_{e_0e_1}>0\}}$ is the line graph of the complete graph \(K_n\). This graph is
\(2(n-2)\)-regular, and its adjacency spectrum is $2(n-2), n-4,$ and $ -2$ with multiplicities $1$, $n-1$, and $n(n-3)/2$, respectively. Hence
\[
1-\lambda_2(D_B^{-1}B)
=
1-\frac{n-4}{2(n-2)}
=
\frac{n}{2(n-2)}
\ge \frac12 .
\]
Since the strong pseudo-regularity constant in part (b) is \(K=1\), Lemma~\ref{lem:trace-optr-converse}
implies
\(
1-\lambda_2(D^{-1}R)\ge \frac12 .
\)
This proves \eqref{cond:spec-gap}.
\end{proof}

On the other hand, the MPLE does yield $\sqrt{N}$-consistent
estimation in the strong antiferromagnetic regime. By
Lemma~\ref{lem:triangle_mean}{\rm (a)-(b)} and
Lemma~\ref{lem:trace-optr}{\rm (b)}, the tensor $A$ satisfies the
mean-field condition \eqref{cond:GW}. The following lemma verifies
the variational-gap hypothesis of
Theorem~\ref{cor:gap_implies_estimation}. Hence
Theorem~\ref{cor:gap_implies_estimation}, together with
Theorem~\ref{thm:main-1}, implies that for every $h_0\in\mathbb{R}$
there exists $L>0$ such that, whenever $\beta_0\leq -L$, the MPLE
is jointly $\sqrt{N}$-consistent for $(\beta_0,h_0)$.

\begin{lem}\label{lem:triangle}
Assume $\kappa_-=0$ and $\mu(\{0\})>0$. For any $h\in\mathbb R$, there exists a constant
$L=L(h)>0$ such that, whenever $\beta\le -L$, there exist constants $\varepsilon>0$ and $\delta>0$
for which the following holds: for all sufficiently large $N$,
\[
\sup_{y \in E_N} \big( f(y)  -I(y) \big)
\le
\sup_{x\in[0,\kappa_+]^N} \big( f(x) - I(x) \big) -\delta N,
\]
where $E_N=\{y\in [0,\kappa_+]^N:\sum_i(y_i-\bar{y})^2\le N\varepsilon \}$.
\end{lem}

\begin{proof}
Since $\mu(\{0\})>0$, we have $I(0)=\sup_{\lambda\in\mathbb R}\big(-\Lambda(\lambda)\big)=-\log\mu(\{0\})<\infty.$
Moreover, since $\mu$ is nondegenerate and assigns positive mass to $(0,\kappa_+]$,
\[
\Lambda(h)=\log\int e^{hz}\,d\mu(z)>\log\mu(\{0\})=-I(0).
\]
Set $\Delta(h):=\Lambda(h)+I(0)>0$ and 
define
\[
\Pi(x):=\sum_{e,f,g}A_{efg}x_ex_fx_g,
\qquad
G_h(t):=ht-I(t),
\]
so that
\[
f(x)  -I(x)=\frac{\beta}{3} \Pi(x) +\sum_{e} G_h(x_e),\qquad x\in[0,\kappa_+]^N.
\]
Note that $G_h$ is concave. Let $V=V_1\sqcup V_2$ be a partition of the vertex set with $|V_1| = \lfloor n/2 \rfloor.$ For $u\in[0,\kappa_+]$,
define $x^{(u)}\in[0,\kappa_+]^N$ by
\[
x^{(u)}_e=
\begin{cases}
u, & e \text{ crosses }(V_1,V_2),\\
0, & e\subset V_1 \text{ or } e\subset V_2.
\end{cases}
\]
Every triangle contains at least one within-part edge, hence $\Pi(x^{(u)})=0$. Then
\[
\frac1N f(x^{(u)})=\frac{|V_1||V_2|}{N}\big(hu-I(u)\big)-\Big(1-\frac{|V_1||V_2|}{N}\Big)I(0).
\]
Maximizing over $u$ and using $\sup_{u\in[0,\kappa_+]}(hu-I(u))=\Lambda(h)$ yields
\begin{equation}\label{eq:global-lb}
\frac1N\sup_{x\in[0,\kappa_+]^N} \big( f(x) - I(x) \big)
\ge \frac{1}{2} \Lambda(h)-\frac{1}{2}I(0)   + o(1)
\ge -I(0)+\frac12\Delta(h) + o(1).
\end{equation}

Now let $y\in[0,\kappa_+]^N$ satisfy $\sum_i (y_i-\bar y)^2\le \varepsilon N$ and set $t=\bar y$. By concavity of $G_h$ and Jensen's inequality,
\begin{equation}\label{eq:jensen}
\sum_i G_h(y_i)\le N G_h(t).
\end{equation}
 For each edge  $e$,
\[
0 \le \partial_{x_e}\Pi(x)=3\sum_{f,g}A_{efg}x_fx_g
\le 3\kappa_+^2\sum_{f,g}A_{efg} = 6 \kappa_+^2 (n-2)/n \le 6 \kappa_+^2.
\]
since the number of ordered pairs $(f,g)$ completing a triangle with $e$ equals $2(n-2)$. Therefore $\|\nabla \Pi(x) \|_\infty\le 6\kappa_+^2$ on $[0,\kappa_+]^N$.
By the mean value theorem,
\begin{equation}
|\Pi(y)-\Pi(t\mathbf 1)|
\le 6\kappa_+^2\|y-t\mathbf 1\|_1
\le 6\kappa_+^2\sqrt N\,\|y-t\mathbf 1\|_2
\le 6\kappa_+^2 N\sqrt{\varepsilon}.
\label{eq:lip}
\end{equation}
Combining \eqref{eq:jensen} and \eqref{eq:lip}, we obtain
\begin{align*}
\big( f(y) - I(y) \big) \le \frac{\beta}{3}\Pi(t\mathbf 1)+N\big(ht-I(t)\big)+2\kappa_+^2|\beta|N\sqrt{\varepsilon}.
\end{align*}
For constant vectors, one has $\Pi(t\mathbf 1)=(n-1)(n-2)t^3$, hence
\[
\frac1N \big( f(y) - I(y) \big)\le -\frac23|\beta|t^3+ht-I(t) +2\kappa_+^2|\beta|\sqrt{\varepsilon}+o(1).
\]
Taking the supremum over all such $y\in[0,\kappa_+]^N$ satisfying $\sum_i (y_i-\bar y)^2\le \varepsilon N$ yields
\begin{equation}
\frac1N \sup_{y \in E_N}  \big( f(y) - I(y) \big)
\le S(\beta,h)+2\kappa_+^2|\beta|\sqrt{\varepsilon}+o(1),
\label{eq:nearconst-bound}
\end{equation}
where $S(\beta ,h):=\sup_{t\in[0,\kappa_+]}\big(\frac23 \beta t^3+ht-I(t)\big)$.
Let $\eta:=\Delta(h)/8 >0$. Since $I(0)<\infty$ and $I$ is convex, $I$ is continuous at $0$,
so $ht-I(t)\to -I(0)$ as $t\downarrow 0$. Choose $r := r(\eta, h) \in(0,\kappa_+)$ such that $\sup_{0\le t\le r}\big(ht-I(t)\big)\le -I(0)+\eta$, yielding that for any $\beta \le 0,$
\begin{equation}
\sup_{t\in[0, r]}\big(\frac23 \beta t^3+ht-I(t)\big) \le -I(0)+\eta.
\label{eq:r-choice}
\end{equation}
Also, for any $\beta < 0$ and $t\ge r$,
\[
\frac23\beta t^3+ht-I(t)\le \frac23\beta r^3+\sup_{t\in[0,\kappa_+]}\big(ht-I(t)\big)
= \frac23\beta r^3+\Lambda(h).
\]
Since the right-hand side above tends to $- \infty$ as $\beta \to  - \infty,$ we can choose $L=L(h)>0$ such that
\begin{equation}
-\frac23 L r^3+\Lambda(h)\le -I(0)+\eta.
\label{eq:L-choice}
\end{equation}
Then \eqref{eq:r-choice} and \eqref{eq:L-choice} imply that for every $\beta \le -L$,
\begin{equation}
S(\beta,h)\le -I(0)+\eta.
\label{eq:S-bound}
\end{equation}
Now fix any $\beta\le -L$. Choose
\(
\varepsilon:=\big(\frac{\eta}{2\kappa_+^2|\beta|}\big)^2,
\)
so that $2\kappa_+^2|\beta|\sqrt{\varepsilon} =  \eta$. Using \eqref{eq:nearconst-bound} and
\eqref{eq:S-bound}, we obtain that
\begin{equation}
\sup_{y \in E_N}  \big( f(y) - I(y) \big)
\le \big(-I(0)+2\eta\big)N + o(N) =  \big(-I(0)+\tfrac14 \Delta(h) \big)N + o(N)
\label{eq:left-final}
\end{equation}

On the other hand, from \eqref{eq:global-lb}
\[
\sup_{x\in[0,\kappa_+]^N} \big(f(x) - I(x) \big) \ge \big(-I(0)+\tfrac12\Delta(h)\big)N + o(N).
\]
Combining this with
\eqref{eq:left-final} yields for all large $N$
\[
\sup_{y \in E_N}  \big( f(y) - I(y)\big)
\le
\sup_{x\in[0,\kappa_+]^N} \big(f(x)-I(x)\big)-\delta N,
\]
where one may take, for instance, $\delta:=\frac18\Delta(h)>0$.
This completes the proof.
\end{proof}


\begin{proof}[Proof of Theorem~\ref{thm:app-triangle}]
By Lemma~\ref{lem:triangle_mean}, the edge--triangle tensor satisfies Assumption~\ref{ass:standing}, the mean-field condition~\eqref{cond:GW}, asymptotic regularity~\eqref{cond:regu}, nontriviality~\eqref{cond:A_nontrivial}, and the spectral gap condition~\eqref{cond:spec-gap}. Therefore Theorem~\ref{thm:est-imposs} applies under the assumptions in part~\textup{(a)} and gives \(T_N(X)=o_p(1)\).

For part~\textup{(b)}, Lemma~\ref{lem:triangle} verifies the variational gap hypothesis of Theorem~\ref{cor:gap_implies_estimation} for every fixed \(h_0\), provided \(\beta_0\) is sufficiently negative. Thus Theorem~\ref{cor:gap_implies_estimation} gives \(P_{\beta_0,h_0}(T_N(X)\ge c)\to1\) for some \(c>0\). Applying Theorem~\ref{thm:main-1} then yields existence of the MPLE with probability tending to one and joint \(\sqrt N\)-consistency for \((\beta_0,h_0)\).
\end{proof}

\subsubsection{Proofs for edge-three-star ERGM} Recall that $N=\binom n2$, and $A$ is the three-star tensor defined in \eqref{eq:ergm-three-star}. We begin with the following lemma for validating some conditions.

\begin{lem}\label{lem:three-star-mean-field}
The edge-three-star tensor has constant row sums and satisfies
Assumption~\ref{ass:standing}, the nontriviality condition
\eqref{cond:A_nontrivial}, and the strong pseudo-regularity condition
\eqref{cond:strong-pseu-regu}. Moreover, $\mathrm{Tr}(\cR^2)=O(n)$ and 
hence, the mean-field condition \eqref{cond:GW} holds. Finally, $A$ satisfies the spectral-gap condition~\eqref{cond:spec-gap}.
\end{lem}

\begin{proof}
For distinct edges $e,f$, write $e\sim f$ if they share a
vertex. Direct counting gives
\begin{equation}\label{eq:R_star}
    \cR_{ef}
    =
    \sum_g A_{efg}
    =
    \frac{n-3}{n^2}\mathbf{1}_{\{e\sim f\}}.
\end{equation}
Indeed, if $e\sim f$, there are exactly $n-3$ choices of a third
edge $g$ such that $e,f,g$ form a three-star. Since every edge shares
a vertex with exactly $2(n-2)$ other edges,
\begin{equation}
    \cR_e
    =
    \frac{2(n-2)(n-3)}{n^2},
    \label{eq:three-star-row-sum}
\end{equation}
independently of $e$. Thus Assumption~\ref{ass:standing} and
\eqref{cond:A_nontrivial} hold for all sufficiently large $n$.

All positive entries of $\cR$ are equal, so
\eqref{cond:strong-pseu-regu} holds with $K=1$. Furthermore,
\[
\begin{aligned}
    \mathrm{Tr}(\cR^2)
    =
    \sum_{e,f}\cR_{ef}^2 
    =
    N\cdot 2(n-2)
    \Big(\frac{n-3}{n^2}\Big)^2 
    =
    \frac{(n-1)(n-2)(n-3)^2}{n^3}
    =
    O(n).
\end{aligned}
\]
Since $N\asymp n^2$, this is $o(N/\log N)$. The result follows
from Lemma~\ref{lem:trace-optr}{\rm (b)}.

The spectral-gap condition follows exactly as in Lemma~\ref{lem:triangle_mean}{\rm (c)}. Indeed, by \eqref{eq:R_star}, the associated support graph $B_{e_0e_1}:=\mathbf 1_{\{\cR_{e_0e_1}>0\}}$ is again the line graph of the complete graph $K_n$. The remainder of the argument is identical.
\end{proof}

\begin{lem}\label{lem:three-star-variational}
Fix $\beta,h\in\mathbb{R}$, and let $f$ be the corresponding
edge-three-star Hamiltonian. There exist constants $c>0$ and
$C<\infty$, independent of $n$, such that
\[
    f(y)-I(y)
    \leq
    \sup_{u\in[\kappa_-,\kappa_+]^N} (f(u)-I(u) )
    -
    c\sum_e(y_e-\bar y)^2
    +
    Cn
\]
for every $y\in[\kappa_-,\kappa_+]^N$, where
$\bar y=N^{-1}\sum_e y_e$.
\end{lem}
\begin{proof}
If $I(y)=\infty$, the claim is immediate. Write $y_{ij}=y_{\{i,j\}}$ and set
$q_i=(n-1)^{-1}\sum_{j\ne i}y_{ij}$ and
$D(y)=\sum_i\sum_{j\ne i}(y_{ij}-q_i)^2$.
Since $\Lambda''\le\kappa^2$, the function
$I(t)-t^2/(2\kappa^2)$ is convex on the effective domain of $I$.
Thus Jensen's inequality, summed over $i$, gives
\begin{equation} \label{eq:star_bd1}
 h\sum_e y_e-I(y)
\le
\frac{n-1}{2}\sum_i\bigl(hq_i-I(q_i)\bigr)
-\frac{1}{4\kappa^2}D(y). 
\end{equation}
For each $i$, expanding $\big(\sum_{j\ne i}y_{ij}\big)^3$ and removing the terms with repeated indices gives
\[ \sum_{j,k,\ell\ne i\ {\rm distinct}}y_{ij}y_{ik}y_{i\ell}
=((n-1)q_i)^3-3(n-1)q_i\sum_{j\ne i}y_{ij}^2
+2\sum_{j\ne i}y_{ij}^3,\] and hence
\begin{equation} \label{eq:star_bd2}
\Big|
\la A,y^{\otimes3}\ra
-\frac{(n-1)^3}{n^2}\sum_iq_i^3
\Big|
\le
\frac1{n^2}\sum_i
\bigl(3(n-1)^2\kappa^3+2(n-1)\kappa^3\bigr)
\le 3\kappa^3n.
\end{equation}
Let
\[ M_n:=\sup_{t\in[\kappa_-,\kappa_+]}
\big( 2\beta(n-1)^2t^3/(3n^2)+ht-I(t) \big) .\]
Combining the bounds \eqref{eq:star_bd1} and \eqref{eq:star_bd2}, we obtain
\begin{equation}
f(y)-I(y)
\le NM_n-\frac{1}{4\kappa^2}D(y)+|\beta|\kappa^3n.
    \label{eq:three-star-upper}
\end{equation}
For constant vectors, direct counting gives
$\la A,(t\mv1)^{\otimes3}\ra=(n-1)(n-2)(n-3)t^3/n$. Since the cubic coefficient in $NM_n$ is $\beta(n-1)^3/(3n)$ and the absolute difference between the two cubic terms is at most $|\beta|\kappa^3n$, we deduce
 \begin{equation} \label{eq:three-star-lower}
\sup_u\bigl(f(u)-I(u)\bigr)\ge  \sup_{t \in [\kappa_{-}, \kappa_{+}] } \bigl(f(t \mv1)-I(t\mv1)\bigr) \ge   NM_n-|\beta|\kappa^3n.
\end{equation}
Expanding around
the row means, we have
\begin{equation}
    2\sum_e(y_e-\bar y)^2
    =
    D(y)+(n-1)\sum_i(q_i-\bar y)^2.
    \label{eq:three-star-variance-decomposition}
\end{equation}
Moreover,
\(
    \sum_i q_i=n\bar y,
\)
and symmetry of $y$ gives
\(
    n(q_j-\bar y)
    =
    \sum_{i\neq j}(y_{ij}-q_i).
\)
Thus, by the Cauchy-Schwarz inequality,
\[
    n^2\sum_j(q_j-\bar y)^2
    \leq
    (n-1)D(y).
\]
Substituting this into
\eqref{eq:three-star-variance-decomposition} yields
\[
    2\sum_e(y_e-\bar y)^2
    \leq
    \Big(1+\frac{(n-1)^2}{n^2}\Big)D(y)
    \leq
    2D(y),
\]
which implies  $ D(y)\geq\sum_e(y_e-\bar y)^2.$
Combining this with \eqref{eq:three-star-upper} and
\eqref{eq:three-star-lower} proves the lemma with $c=(4\kappa^2)^{-1}$ and $ C=2|\beta|\kappa^3$.
\end{proof}

Now we prove Theorem~\ref{thm:app-3-star}.
\begin{proof}[Proof of Theorem~\ref{thm:app-3-star}]
For each edge $e$, note that
\[
    b_e(X)
    =
    \Lambda'\bigl(\beta_0m_e(X)+h_0\bigr),
    \qquad
    \bar b(X)
    =
    \frac1N\sum_e b_e(X).
\]
By Lemma~\ref{lem:three-star-mean-field},
Assumption~\ref{ass:standing} and the mean-field condition
\eqref{cond:GW} hold.

Fix $\varepsilon>0$. Applying
Lemma~\ref{lem:three-star-variational} at
$(\beta,h)=(\beta_0,h_0)$, whenever
\(
    \sum_e(y_e-\bar y)^2\geq\varepsilon N
\)
we have
\[
    f(y)-I(y)
    \leq
    \sup_u ( f(u)-I(u) )
    -
    c\varepsilon N
    +
    Cn.
\]
Since $n=o(N)$, for all sufficiently large $n$ the right-hand side
is at most
\[
    \sup_u ( f(u)-I(u) )
    -
    \frac{c\varepsilon}{2}N.
\]
Proposition~\ref{thm:conc} therefore gives
\begin{equation}
    \frac1N\sum_e
    \bigl(b_e(X)-\bar b(X)\bigr)^2
    =
    o_p(1).
    \label{eq:three-star-b-homogeneous}
\end{equation}
Suppose first that $\beta_0\neq0$. By
\eqref{eq:three-star-row-sum},
\(
    |m_e(x)|
    \leq
    \kappa^2\cR_e
    \leq
    2\kappa^2
\)
uniformly in $e$ and $x$. Hence
\(
    a_0
    :=
    \inf_{|u|\leq2\kappa^2}
    \Lambda''(\beta_0u+h_0)
    >
    0.
\)
The mean value theorem gives
\[
    |b_e(X)-b_f(X)|
    \geq
    |\beta_0|a_0
    |m_e(X)-m_f(X)|.
\]
Using the pairwise variance identity \eqref{eq:TN},
\[
\begin{aligned}
    T_N(X)
    =
    \frac1{2N^2}\sum_{e,f}
    \bigl(m_e(X)-m_f(X)\bigr)^2 
    \le
    \frac1{\beta_0^2a_0^2}
    \frac1N\sum_e
    \bigl(b_e(X)-\bar b(X)\bigr)^2 
    \overset{\eqref{eq:three-star-b-homogeneous}}{=}
    o_p(1)
\end{aligned}
\]
It remains to prove the $\beta_0=0$ case. In this case,
\(
P_{0,h_0}=Q_t,
    \ \text{with} \ 
    t:=\Lambda'(h_0).
\)
Using the notation of Lemma~\ref{lem:iid-measure}, for a fixed edge
$e$ there are exactly $2(n-2)(n-3)$ ordered pairs $(f,g)$ for which
$A_{efg}>0$. Therefore
\(
    \|A_e\|_F^2
    =
    \frac{2(n-2)(n-3)}{n^4},
\)
and hence
\[
    \sum_e\|A_e-\bar A\|_F^2 = \sum_{e} \| A_e\|_F^2 - N \|\bar A\|_F^2
    \leq
    \sum_e\|A_e\|_F^2
    =
    O(1).
\]
Set
\(
    r_n
    :=
    \frac{2(n-2)(n-3)}{n^2},
\)
the row sums satisfy $\cR_e=r_n$ for every $e$, and
\[
    \sum_{e,f}
    \big(\cR_{ef}-\frac1N \cR_f\big)^2
    =
    \mathrm{Tr}(\cR^2)-r_n^2
    =
    O(n).
\]
Thus the first term in Lemma~\ref{lem:iid-measure}(i) vanishes,
while its other two terms are respectively $O(1)$ and $O(n)$.
Consequently,
\(
    N\mathbb{E}_{Q_t}T_N(X)=O(n).
\)
Since $N\asymp n^2$, then 
\(
    \mathbb{E}_{Q_t}T_N(X)
    =
    O(n^{-1})
\)
Markov's inequality now gives
\(
    T_N(X)=o_p(1).
\)
\end{proof}

\subsection{Proofs for 3-term Arithmetic Progressions}
We first introduce the following lemmas to validate the conditions of our general results.

\begin{lem}\label{lem:3ap_mean}
For both the cyclic tensor $A^{\mathrm{cyc}}$ and the integer
tensor $A^{\mathrm{int}}$, Assumption~\ref{ass:standing}, the
nontriviality condition \eqref{cond:A_nontrivial}, and the strong
pseudo-regularity condition
\eqref{cond:strong-pseu-regu} hold. Moreover, $\mathrm{Tr}(\cR^2)=O(1)$ and 
 the mean-field condition \eqref{cond:GW} follows. 
 
 In the cyclic case, the row sums are constant, and hence the
asymptotic regularity condition \eqref{cond:regu} holds. In the
integer case,
\(
    \sum_i (\cR_i-\bar \cR)^2=\Omega(N).
\)
\end{lem}

\begin{proof}
For distinct $a,b$, let
\[
    q_{ab}
    :=
    \#\bigl\{
        c:\{a,b,c\}\text{ forms a nontrivial $3$-AP}
      \bigr\}.
\]
Then, in either model,
\(
    \cR_{ab}
    =
    \sum_c A_{abc}
    =
    \frac{q_{ab}}{N}.
\) For a fixed pair $a,b$, the possible completions arise from
\[
    c=2a-b,
    \qquad
    c=2b-a,
    \qquad
    2c=a+b.
\]
In the integer model, the last equation has at most one solution,
so $q_{ab}\leq 3$. In the cyclic model, the equation
$2c=a+b$ has at most two solutions, so $q_{ab}\leq 4$. Therefore,
in both models,
\[
    \frac1N
    \leq \cR_{ab}
    \leq \frac4N
    \qquad\text{whenever }\cR_{ab}>0.
\]
Thus the strong pseudo-regularity condition
\eqref{cond:strong-pseu-regu} holds with $K=4$. Also,
\(
    \cR_a
    =
    \sum_b \cR_{ab}
    \leq 4,
\) so Assumption~\ref{ass:standing} holds, for example with
$\gamma=4$. Furthermore,
\[
    \mathrm{Tr}(R^2)
    =
    \sum_{a,b}R_{ab}^2
    \leq
    N^2 (4/N)^2
    \leq 16.
\]

We next verify nontriviality. In the cyclic model, translation
invariance shows that $\cR_a=r_N$ is independent of $a$. For fixed
$a$, consider all $b$ such that $b\neq a$ and
$2(a-b)\not\equiv0\pmod N$. Setting $c=2a-b$, the triple
$(b,a,c)$ is a nontrivial cyclic $3$-AP. Hence
\[
    \cR_a
    \geq
    \frac{N-2}{N}
    \geq \frac12 \quad \text{for $N\geq4$}.
\]
For the integer model, a direct count gives
\[
    \cR_i
    =
    \frac{2}{N}
    \Big(
        \Big\lfloor\frac{i-1}{2}\Big\rfloor
        +
        \Big\lfloor\frac{N-i}{2}\Big \rfloor
        +
        \min\{i-1,N-i\}
    \Big).
\]
In particular, $\cR_i\geq 1/2$ for every $i$ and every $N\geq4$.
Thus \eqref{cond:A_nontrivial} holds in both models. Since
\(
    \mathrm{Tr}(\cR^2)
    =
    O(1)
    =
    o\left( N / \log N\right),
\)
Lemma~\ref{lem:trace-optr}{\rm (b)} yields the mean-field
condition \eqref{cond:GW}.

In the cyclic case, translation invariance also shows that the
row sums are constant, so \eqref{cond:regu} holds. In the integer
case, uniformly for $\alpha$ in compact subsets of $(0,1)$,
\[
    \cR_{\lfloor\alpha N\rfloor}
    =
    1+2\min\{\alpha,1-\alpha\}+O(N^{-1}).
\]
Since the limiting profile is nonconstant, a Riemann-sum argument
gives
\(
    \sum_i(\cR_i-\bar  \cR)^2=\Omega(N).
\)
\end{proof}

\begin{lem}\label{lem:3ap_gap}
For the cyclic $3$-AP tensor $A^{\mathrm{cyc}}$, the spectral-gap
condition \eqref{cond:spec-gap} holds.
\end{lem}

\begin{proof}
 Let \(B_{ab}:=\mathbf 1_{\{\cR_{ab}>0\}}\) be the support graph of \(\cR\) in the
cyclic model. We first identify this graph. If \(a\ne b\) and
\(2(a-b)\not\equiv 0\pmod N\), then \(c=2a-b\) is distinct from both \(a\) and \(b\),
so \(\{b,a,c\}\) is a nontrivial cyclic 3-AP. Thus \(B_{ab}=1\).

It remains only to consider the case \(2(a-b)\equiv0\pmod N\), which can occur only
when \(N\) is even and \(b=a+N/2\). If \(N\equiv0\pmod4\), then the equation
\(2c=a+b\pmod N\) has solutions, and any such solution is distinct from \(a,b\);
hence \(B_{ab}=1\). If \(N\equiv2\pmod4\), then this equation has no solution, and
the two other possible completions give \(c=a\) or \(c=b\). Hence in this case the
support graph is the complete graph with the perfect matching
\(\{a,a+N/2\}\) removed.

Therefore \(B\) is either the complete graph, or the complete graph with a perfect
matching removed. In both cases the simple random walk on \(B\) has
\(
1-\lambda_2(D_B^{-1}B)\ge 1.
\)
Indeed, for the complete graph all nontrivial eigenvalues are \(-1/(N-1)\), while
for the complete graph with a perfect matching removed the second largest eigenvalue
is \(0\).

By Lemma~\ref{lem:3ap_mean}, the strong pseudo-regularity condition holds with \(K=4\). Lemma~\ref{lem:trace-optr-converse}
therefore gives
\(
1-\lambda_2(D^{-1}R)\ge \frac14 .
\)
This proves the spectral-gap condition~\eqref{cond:spec-gap}.
\end{proof}

The next lemma supplies the additive-combinatorial input needed for showing variational gap for cyclic $3$-AP in the strong antiferromagnetic regime. It
 is a special case of \cite[Proposition~2.1]{SahSawhneyZhao21}, but we include a proof for completeness.  
 \color{black}
\begin{lem}\label{lem:T-3ap-upper}
There exist constants $\delta_0,K>0$ such that, for every $0<\delta\le \delta_0$ and every sufficiently large $N\in\mathbb N$, there exists a subset $A\subseteq \mathbb Z/N\mathbb Z$ with $|A|\ge \delta N$ such that
\[
T_3^{\mathrm{cyc}}(A) \le \delta^{K\log(1/\delta)}N^2.
\]
\end{lem}

\begin{proof}

Fix $c>0$ from Behrend's theorem, so that for every sufficiently large integer $M$ there exists a set $B\subseteq [M]$ with $T_3^{\mathrm{int}}(B)=0$ and $|B|\ge M e^{-c\sqrt{\log M}}$. Choose $\delta_0>0$ so small that for every $0<\delta\le \delta_0$, the integer
\[
M:=\lfloor \exp((4c^2)^{-1}\log^2(1/\delta))\rfloor
\]
is large enough for Behrend's theorem to apply and also satisfies $\delta^{1/2}\ge 12\delta$. With this choice, we have $ e^{-c\sqrt{\log M}} M\ge  \delta^{1/2} M \ge 12\delta M$. Hence there exists a set $B\subseteq [M]$ such that
\[T_3^{\mathrm{int}}(B)=0
\qquad\text{and}\qquad
|B|\ge 12\delta M.
\]
Also, $M\ge \delta^{-C_0\log(1/\delta)}$ for some absolute constant $C_0>0$.
Set $L:=3M$, $m:=\lfloor N/(6M)\rfloor$, and
\[
A:=\bigcup_{r=0}^{m-1}(rL+B)\subseteq \mathbb Z/N\mathbb Z.
\]
Since $B\subseteq [M]$, the translates $rL+B$ are disjoint. Moreover, for every $0\le r\le m-1$ and $b\in B$ we have $rL+b\le (m-1)L+M\le mL\le N/2$, so $A$ is contained in the first half of the cycle. Moreover, $|A|=m|B|\ge 2\delta N-12\delta M,$ which implies that $|A|\ge \delta N$ for all sufficiently large $N$.

We next bound $T_3^{\mathrm{cyc}}(A)$. Let $a,b,c\in A$ satisfy $a+c\equiv 2b\pmod N$. Since $a,b,c\in \{1,\dots,\lfloor N/2\rfloor\}$, we have $|a+c-2b|<N$, hence in fact $a+c=2b$ in $\mathbb Z$. Writing uniquely
\[
a=r_aL+s_a,\quad b=r_bL+s_b,\quad c=r_cL+s_c, \text{ \ \  with $0\le r_a,r_b,r_c\le m-1$ and $s_a,s_b,s_c\in B$,}
\]
 we obtain $(r_a+r_c-2r_b)L=2s_b-s_a-s_c.$ The right-hand side has absolute value $<2M<L$, while the left-hand side is a multiple of $L$, so both sides vanish. Hence $r_a+r_c=2r_b$ and $s_a+s_c=2s_b$. Since $T_3^{\mathrm{int}}(B)=0$, the latter implies $s_a=s_b=s_c$. Therefore every nontrivial cyclic $3$-term arithmetic progression in $A$ is obtained by fixing some $s\in B$ and choosing a nontrivial integer $3$-term arithmetic progression among the indices $r_a, r_b, r_c$ in $\{0,1,\dots,m-1\}$. Consequently,
\[
T_3^{\mathrm{cyc}}(A)\le |B|\,T_3^{\mathrm{int}}(\{0,1,\dots,m-1\})\ll |B|m^2\ll \frac{N^2}{M} \le \delta^{C_0\log(1/\delta)}N^2.
\]
where in the last inequality we used $M\ge \delta^{-C_0\log(1/\delta)}$.
This completes the proof.
\end{proof}

In the $\mathrm{Ber}(1/2)$ case, the log moment generating function is $\Lambda(\lambda)=\log\bigl((1+e^\lambda)/2\bigr)$ for $\lambda\in\mathbb R$, and the associated rate function is $I(t)=t\log(2t)+(1-t)\log(2(1-t))$ for $t\in[0,1]$. For $x\in[0,1]^N$, we write $I(x):=\sum_i I(x_i)$.

\begin{lem}\label{lem:3AP}
In the cyclic case, for any $h\in\mathbb R$, there exists a constant
$L=L(h)>0$ such that, whenever $\beta\le -L$, there exist constants $\varepsilon>0$ and $\delta>0$
for which the following holds: for all sufficiently large $N$, we have
\[
\sup_{x \in E_N } \big( f(x) - I(x) \big)
\le
\sup_{x\in[0,1]^N} \big( f(x) - I(x) \big)- \delta N,
\]
where $E_N := \{ x\in[0,1]^N: \sum_i (x_i-\bar x)^2\le \varepsilon N\}.$
Consequently, for every $K>0$ there exists $L<\infty$, such that joint estimation at rate $\sqrt{N}$ is possible in the parameter regime $(\beta\le -L, |h|\le K)$ using the pseudo-likelihood estimator.
\end{lem}

\begin{proof}
Let
\[
\Pi(x):=\sum_{a,b,c} A^{\mathrm{cyc}}_{a,b,c}x_ax_bx_c,
\qquad
G_h(t):=ht-I(t),
\]
so that $f(x)-I(x)=\frac{\beta}{3} \Pi(x)+\sum_i G_h(x_i)$.

By Lemma~\ref{lem:3ap_mean}, $\cR_i=\bar \cR$ is independent of $i$, and $\max_i \cR_i\le C_0$ for some
absolute constant $C_0$. Moreover, $\bar \cR$ is bounded away from $0$ uniformly in $N$:
indeed, fixing $i$ and counting progressions in which $i$ is the middle term gives
$\cR_i\ge 2\lfloor (N-1)/2\rfloor/N\ge 1/2$.

Now let $x \in[0,1]^N$ and set $t:=\bar x$. Since $G_h$ is concave, Jensen gives
$\sum_i G_h(x_i)\le N G_h(t)$. Also, as in the proof of Lemma~\ref{lem:triangle},
$\|\nabla \Pi\|_\infty\le 3C_0$ on $[0,1]^N$, and therefore
$|\Pi(x)-\Pi(t\mathbf 1)|\le 3C_0\|x-t\mathbf 1\|_1$.
Hence, whenever $\sum_i (x_i-\bar x)^2\le \varepsilon N$, we have
\[
f(x)-I(x)
\le
N\bigl(\frac{\beta}{3} \bar \cR t^3+ht-I(t)\bigr)+C_1\frac{|\beta|}{3}N\sqrt{\varepsilon},
\]
for some absolute constant $C_1>0$.

Set $\Delta(h):=\Lambda(h)+I(0) = \log(1+ e^h)>0$. Exactly as in the proof of Lemma~\ref{lem:triangle}, given $\eta>0$
one can choose $r=r(h,\eta)\in(0,1)$ such that
$\sup_{0\le t\le r}(ht-I(t))\le -I(0)+\eta$.
Since $\bar \cR \ge 1/2$, we may then choose $L=L(h)>0$ so large that for every $\beta\le -L$,
\[
\sup_{t\in[0,1]}\bigl(\frac{\beta}{3} \bar \cR  t^3+ht-I(t)\bigr)\le -I(0)+\eta
\]
for all $N$. Consequently,
\[
\sup_{x \in E_N }
\bigl(f(x)-I(x)\bigr)
\le
\bigl(-I(0)+\eta+C_1|\beta|\sqrt{\varepsilon}\bigr)N.
\]

For the lower bound on the unrestricted supremum, fix $\rho>0$ to be chosen later, and let
$A\subseteq \mathbb Z/N\mathbb Z$ be given by Lemma~\ref{lem:T-3ap-upper}, so that $|A|\ge \rho N$ and
$T_3^{\mathrm{cyc}}(A)\le \rho^{K_0\log(1/\rho)}N^2$ for all sufficiently large $N$, where
$K_0>0$ is absolute. 

Let $u\in[0,1]$ maximize $hu-I(u)$, so that $hu-I(u)=\Lambda(h)$, and
set $x^{(u)}:=u\,\mathbf 1_A$. Since each cyclic $3$-AP contributes at most six ordered triples, $\Pi(x^{(u)})\le \frac{6}{N}T_3^{\mathrm{cyc}}(A).$ On the other hand, 
$\sum_i \bigl(hx_i^{(u)}-I(x_i^{(u)})\bigr)=|A|\bigl(hu-I(u)\bigr)+(N-|A|)\bigl(-I(0)\bigr)$.
As $hu-I(u)=\Lambda(h)$, this becomes
\[ \sum_i \bigl(hx_i^{(u)}-I(x_i^{(u)})\bigr)=-NI(0)+|A|\bigl(\Lambda(h)+I(0)\bigr)
=-NI(0)+|A|\Delta(h)\ge -NI(0)+\rho N\Delta(h).\]
Therefore, for all sufficiently large $N$,
\[
\frac1N\bigl(f(x^{(u)})-I(x^{(u)})\bigr)
\ge
-I(0)+\rho\Delta(h)-2|\beta|\,\rho^{K_0\log(1/\rho)}.
\]

Now choose $\rho=\rho(\beta,h)>0$ so small that
$2|\beta|\,\rho^{K_0\log(1/\rho)}\le \rho\Delta(h)/4$, and set
$\eta:=\rho\Delta(h)/8$. Then
\[
\sup_{x\in[0,1]^N}\bigl(f(x)-I(x)\bigr)\ge \bigl(-I(0)+6\eta\bigr)N
\]
for all sufficiently large $N$.

Finally choose $\varepsilon>0$ so that $C_1|\beta|\sqrt{\varepsilon}\le \eta$. Combining the
last two bounds gives, for all sufficiently large $N$,
\[
\sup_{x \in E_N }
\bigl(f(x)-I(x)\bigr)
\le
\sup_{x\in[0,1]^N}\bigl(f(x)-I(x)\bigr)-4\eta N.
\]
This proves the lemma, with $\delta:=4\eta$.
\end{proof}

We now present the proof of Theorem~\ref{thm:app-3ap}.
\begin{proof}[Proof of Theorem~\ref{thm:app-3ap}]
We first consider the cyclic model.
Lemma~\ref{lem:3ap_mean} verifies
Assumption~\ref{ass:standing}, nontriviality
\eqref{cond:A_nontrivial}, asymptotic regularity
\eqref{cond:regu}, and the
mean-field condition \eqref{cond:GW}.
Moreover, Lemma~\ref{lem:3ap_gap} gives the
spectral-gap condition \eqref{cond:spec-gap}.

Since
\(\mu=\frac12\delta_0+\frac12\delta_1
\)
is supported on the nonnegative real line, it is stochastically
nonnegative, and $\kappa_-=0$. Therefore
Theorem~\ref{thm:est-imposs} gives $   T_N(X)=o_p(1)$
when $\beta_0>0$ and $h_0\geq0$. This proves part~(i)(a).

For part~(i)(b), Lemma~\ref{lem:3AP} verifies the variational-gap
hypothesis of Theorem~\ref{cor:gap_implies_estimation} whenever
$\beta_0$ is sufficiently negative. Hence
Theorem~\ref{cor:gap_implies_estimation} gives $ P^{\mathrm{cyc}}_{\beta_0,h_0}
    \bigl(T_N(X)\geq c\bigr)
    \to 1$
for some $c>0$. The existence and joint $\sqrt{N}$-consistency
of the MPLE then follow from Theorem~\ref{thm:main-1}.

For the integer model, Lemma~\ref{lem:3ap_mean} gives
\(
    \sum_i(\cR_i-\bar \cR)^2=\Omega(N).
\)
 Moreover, for the Bernoulli reference measure,
\(
    \Lambda'(h_0)
    =
    \frac{e^{h_0}}{1+e^{h_0}}
    >0
\)
for every finite $h_0$. Therefore Theorem~\ref{thm:positive-res}
gives $T_N(X)=\Omega_p(1)$. Applying
Theorem~\ref{thm:main-1} again yields existence and joint
$\sqrt{N}$-consistency of the MPLE. This proves part~(ii).
\end{proof}

\subsection{Proofs for Inhomogeneous Random Hypergraph Model}
We present the proof of Theorem~\ref{thm:app-hypergraph} similarly as before by validating various conditions in the following lemmas.
\begin{lem}
\label{lem:hypergraph-basic-original-route}
For the random tensor \(A\) defined, the following hold.
\begin{enumerate}
\item[(a)] Assumption~\ref{ass:standing} holds deterministically, with \(\gamma=1\).
 \item[(b)] Deterministically, we have
    $\mathrm{Tr}(\cR^2)\leq 1$. Moreover,  $\big \| \sum_i A_i^2 \big {\|}_{\rm{op}} \le N^{-1} $, which implies that the mean-field condition \eqref{cond:GW} holds.

\item[(c)] With respect to the randomness of \(A\),
\[
        \frac1N\sum_i(\cR_i-\bar{\cR})^2
        \xrightarrow{\bP_A}
        \int_0^1\bigl(G(x)-\bar{G}\bigr)^2\,dx.
\]
\end{enumerate}
\end{lem}

\begin{proof}
For part (a), since \(0\le A_{ijk}\le N^{-2}\) by definition, we have 
\(
        \sum_{j,k=1}^N A_{ijk}
        \le  1.
\)
This proves Assumption~\ref{ass:standing}.

For part~(b), for every $i,j$,
\[
    0\leq \cR_{ij}
    =
    \sum_k A_{ijk}
    \leq
    \frac{N-2}{N^2}
    \leq \frac1N.
\]
Consequently,
\[
    \mathrm{Tr}(\cR^2)
    =
    \sum_{i, j} \cR_{ij}^2
    \leq
    N^2\frac1{N^2}
    =
    1.
\]
Since entries of $A_i$ are nonnegative, $ \| A_i {\|}_{\rm{op}} \le \max_{j} \sum_k A_{ijk} = \max_j \cR_{ij} \le N^{-1}.$ Since $A_i$ is symmetric, and by triangle inequality,
\[ \big \| \sum_i A_i^2 \big {\|}_{\rm{op}} \le \sum_i \big \|  A_i^2 \big {\|}_{\rm{op}} = \sum_i \big \|  A_i \big {\|}_{\rm{op}}^2 \le N.N^{-2} = N^{-1}.  \]
For part (c), let \(r_i:=\bE_A \cR_i\) and \(\bar r:=N^{-1}\sum_i r_i\). Since each $\cR_i = 2N^{-2} \sum_{j< k, j, k \ne i }\xi_{\{i,j,k\}} $ is a normalized sum of \(O(N^2)\) independent   Bernoulli variables with weights at most \(O(N^{-2})\),
\[
\operatorname{Var}_A(\cR_i) \le \frac{4}{N^4}\binom{N-1}{2}= O(N^{-2}), \quad       \bE_A\Big[\frac1N\sum_i(\cR_i-r_i)^2\Big]=O(N^{-2}).
\]
Hence, by Markov's inequality,
\[     \frac1N\sum_i\bigl((\cR_i-\bar{\cR})-(r_i-\bar r)\bigr)^2 \le \frac1N\sum_i(\cR_i-r_i)^2
        \xrightarrow{\bP_A}0.
\]
By continuity of $g$, we have $  r_i=G(i/N)+o(1)$
uniformly in $i$. In particular, $\bar r  = \bar G + o(1)$ and  a Riemann-sum argument gives
\[\frac1N\sum_i(r_i-\bar r)^2 
        \to
        \int_0^1\bigl(G(x)-\bar{G}\bigr)^2 dx.
\]
By the Cauchy-Schwarz inequality and the uniform
boundedness of $\cR_i$ and $r_i$, we have \begin{align*}
&\Big|
\frac1N\sum_i(\cR_i-\bar{\cR})^2
-
\frac1N\sum_i(r_i-\bar r)^2
\Big|\\
&\qquad\le
\Big[
\frac1N\sum_i
\bigl((\cR_i-\bar{\cR})-(r_i-\bar r)\bigr)^2
\Big]^{1/2} \times
\Big[
\frac1N\sum_i
\bigl((\cR_i-\bar{\cR})+(r_i-\bar r)\bigr)^2
\Big]^{1/2}
\xrightarrow{\mathbb P_A}0.
\end{align*}
Combining the last two displays proves the claim.

\end{proof}

\begin{lem}
\label{lem:hypergraph-homogeneous-original-route}
Assume that $G\equiv\bar G$ and that $g$ is strictly positive
on $[0,1]^3$. Then, with probability tending to one over the
randomness of $A$, the tensor $A$ satisfies the asymptotic
regularity condition \eqref{cond:regu}, the nontriviality
condition \eqref{cond:A_nontrivial}, the strong
pseudo-regularity condition \eqref{cond:strong-pseu-regu}, and
the spectral-gap condition \eqref{cond:spec-gap}. 
\end{lem}

\begin{proof}
The asymptotic regularity condition \eqref{cond:regu} follows
from Lemma~\ref{lem:hypergraph-basic-original-route}(c), since
$G\equiv\bar G$. We next verify strong pseudo-regularity. By continuity and strict
positivity of $g$ on the compact set $[0,1]^3$, there exists
$c_*>0$ such that $g(x,y,z)\geq c_*$ on $[0,1]^3.$
For $i\neq j$,
\[
    N^2\cR_{ij}
    =
    \sum_{k\notin\{i,j\}}
    \xi_{\{i,j,k\}}
\quad \text{and} \quad 
    \mathbb{E}_A[N^2\cR_{ij}]
    \geq c_*(N-2).
\]
A Chernoff bound gives
\(
    \mathbb{P}_A\left(
        N^2\cR_{ij}
        \leq \frac{c_*}{2}(N-2)
    \right)
    \leq e^{-cN}
\)
for some $c>0$ independent of $i,j$. Taking a union bound over
the at most $N^2$ ordered pairs $(i,j)$, we obtain that, with
probability tending to one,
\(
    N^2\cR_{ij}
    \geq \frac{c_*}{2}(N-2)
    \
    \text{for every } \ i\neq j.
\)
On this event, for all sufficiently large $N$,
\(
    \frac{c_*}{4N}
    \leq \cR_{ij}
    \leq \frac1N,
    \ \text{for} \ i\neq j.
\)
It follows that the strong pseudo-regularity condition
\eqref{cond:strong-pseu-regu} holds with
\(
    K=\frac4{c_*}.
\) The same lower bound gives
\[
    \cR_i
    =
    \sum_j \cR_{ij}
    \geq
    \frac{c_*(N-1)}{4N}
    \geq \frac{c_*}{8}
\]
for all sufficiently large $N$. Therefore, $\cR_i>0$ for all $i$ and $\cR\geq c_*/8$
and hence  \eqref{cond:A_nontrivial} holds. 

On the same high-probability event, all off-diagonal entries of \(\cR\) are positive and
\(
\max_{i\ne j} \cR_{ij}\le \frac{4}{c_*}\min_{i\ne j}\cR_{ij}.
\)
Thus the support graph \(B_{ij}=\mathbf 1_{\{\cR_{ij}>0\}}\) is the complete graph and
condition~\eqref{eq:RR} holds with \(K=4/c_*\). By Lemma~\ref{lem:trace-optr-converse}, we have
\(
1-\lambda_2(D^{-1}\cR)\ge \frac{c_*}{4},
\)
so the spectral-gap condition~\eqref{cond:spec-gap} holds with probability tending
to one.

\end{proof}

\begin{proof}[Proof of Theorem~\ref{thm:app-hypergraph}]
We first prove part~(a). By
Lemma~\ref{lem:hypergraph-basic-original-route}(a),
Assumption~\ref{ass:standing} holds. If $G$ is not constant, then
\(
    \int_0^1\bigl(G(x)-\bar G\bigr)^2\,dx>0.
\)
Therefore
Lemma~\ref{lem:hypergraph-basic-original-route}(c) gives
\(
    \sum_i(\cR_i-\bar \cR)^2=\Omega(N)
\)
with $\mathbb{P}_A$-probability tending to one. On this
high-probability event, Theorem~\ref{thm:positive-res} applies
whenever $\Lambda'(h_0)\neq0$, and yields
$T_N(X) = \Omega_{P^{A}_{\beta_0,h_0}}(1).$
The existence of the MPLE and joint
$\sqrt{N}$-consistency then follow from
Theorem~\ref{thm:main-1}. This proves part~(a).

For part (b), Lemmas~\ref{lem:hypergraph-basic-original-route} and \ref{lem:hypergraph-homogeneous-original-route} show that, with $\mathbb{P}_A$-probability tending to one, the tensor satisfies Assumption~\ref{ass:standing}, the mean-field condition
\eqref{cond:GW}, asymptotic regularity~\eqref{cond:regu}, nontriviality~\eqref{cond:A_nontrivial},
and the spectral-gap condition~\eqref{cond:spec-gap}. Therefore all hypotheses of
Theorem~\ref{thm:est-imposs} are satisfied. Under the stated ferromagnetic assumptions, Theorem~\ref{thm:est-imposs}
gives $T_N(X)=o_{P^{A}_{\beta_0,h_0}}(1)$ on the same high probability event.
\end{proof}

\section{Proofs for Auxiliary lemmas}
\label{sec:proof of auxiliary}

\begin{proof}[Proof of Lemma~\ref{lem:rate_as_finite}]
Let $X \sim \mu$. Suppose for $c \in \{\kappa_-, \kappa_+\}$, we have $P(X= c)>0$. Thus for any $\gl \in \bR$, we have $\E e^{\gl X} \ge e^{\gl c}P(X=c)$.
Taking $\log$ on both sides, it gives 
\(
\gL(\gl)\ge \gl c + \log P(X=c).
\)
Equivalently $\gl c - \gL(\gl) \le - \log P(X=c) <\infty$. This shows that if the boundary points $\kappa_-, \kappa_+$ have positive mass, then $I(\kappa_-), I(\kappa_+) <\infty$. Now we turn to the interior part. 

Since $\kappa_-,\kappa_+ \in \text{supp}(\mu)$, we have 
\[
\lim_{\gl \to \infty} \frac{\gL(\gl)}{\gl} = \kappa_+ \quad \text{and} \quad \lim_{\gl \to -\infty} \frac{\gL(\gl)}{\gl} = \kappa_-.
\]
It implies that for each $x \in (\kappa_-,\kappa_+)$, there exist a $\gl_0>0$ such that
\[
 \gl x - \gL(\gl) = \gl\Big(x - \frac{\gL(\gl)}{\gl}\Big)< 0 \quad \text{for all $\abs{\gl}\ge \gl_0$}. 
\]
Since the function $\gl \mapsto \gl x - \gL(\gl)$ is continuous, we have $\sup_{\abs{\gl}\le \gl_0}(\gl x - \gL(\gl))<\infty$. Therefore, we have shown that for each $x \in (\kappa_-,\kappa_+)$, $I(x)< \infty$. The proof is complete.
\end{proof}

\begin{proof}[Proof of Lemma~\ref{lem:low_complex}]
Let $\bar{\kappa} = \frac{\kappa_-+\kappa_+}{2}$, fix $\delta>0$ and set \[
L=\big\lceil\frac{\kappa_+ - \kappa_-}{\delta}\big\rceil,\qquad  \mathcal{D}_N(\delta):=\Big\{a {\bf 1}: a\in \Big\{\bar{\kappa}, \bar{\kappa} \pm \frac{\delta}{2} , \bar{\kappa}\pm \delta, \cdots, \bar{\kappa}\pm L\frac{\delta}{2} \Big\}\Big\}.
\] 
Then $|\mathcal{D}_N(\delta)|=2L+1$, which is free of $N$, and only depends on $\delta$. Finally, for any $y\in E_N$, let $a\in \Big\{\bar{\kappa}, \bar{\kappa} \pm \frac{\delta}{2} , \bar{\kappa}\pm \delta, \cdots, \bar{\kappa}\pm L\frac{\delta}{2} \Big\}$ be such that $|a-\bar{y}|\le \frac{\delta}{2}$. Then we have
\begin{align*}
    \sum_i(y_i-a)^2= \sum_i(y_i-\bar{y})^2+N(\bar{y}-a)^2\le N\varepsilon_N+N\frac{\delta^2}{4}\le N\delta^2,
\end{align*}
where the last inequality holds for all $N$ large enough (depending on $\delta>0$). Thus we have shown that for all $\delta>0$, the set $\mathcal{D}_N(\delta)$ is a $\delta\sqrt{N}$ net for $E_N$ in Euclidean metric with bounded size, and so $E_N$ is of low complexity, as desired.
\end{proof}

\begin{proof}[Proof of Lemma~\ref{lem:aux-ineq1}]
By a direct calculation, the RHS equals
\[ \tfrac{1}{3}\big( a^3 + b^3+ c^3 \big ) - \tfrac{1}{3}\big( a^{3/2}b^{3/2}+ b^{3/2}c^{3/2} + c^{3/2}a^{3/2}\big ). \]
By the AM-GM inequality, we have 
\[ \tfrac{1}{3}\big( a^{3/2}b^{3/2}+ b^{3/2}c^{3/2} + c^{3/2}a^{3/2}\big ) \ge (a^3 b^3 c^3)^{1/3} = abc,\]
and the proof follows. 
\end{proof}

\begin{proof}[Proof of Lemma~\ref{lem:elementary}]
    Let $Y$ be a random variable such that  $\bP(Y = y_i) = N^{-1}$ for all $i$. Define $Z = Y^c$. To prove the lemma, we need to show that $\mathrm{Var}(Y) \le \mathrm{Var}(Z)^{1/c} .$ Let $Y_1$ and $Z_1$ be independent copy of $Y$ and $Z$ respectively. We write
    \begin{align*}
        \mathrm{Var}(Y)  = \tfrac{1}{2} \E (Y - Y_1)^2  = \tfrac{1}{2} \E \big (Z^{1/c} - Z_1^{1/c} \big)^2 \le 
        \tfrac{1}{2} \E [(Z - Z_1)^{2/c}],
    \end{align*}
where the last step uses the inequality $|a^{1/c} - b^{1/c}| \le |a-b|^{1/c}$ for any $a, b \ge 0,$ which is a simple consequence of the concavity of $t \mapsto t^{1/c}.$ Since $c>1,$ by Jensen's inequality, it follows that 
   \begin{align*}
               \tfrac{1}{2} \E [(Z - Z_1)^{2/c}] \le 
               \tfrac{1}{2} \big( \E [(Z - Z_1)^{2} ]\big)^{1/c} =  \tfrac{1}{2} \big( 2 \ \mathrm{Var}(Z) \big)^{1/c} \le \mathrm{Var}(Z)^{1/c}.
    \end{align*}
\end{proof}

\end{document}